\documentclass[11pt,a4paper,reqno]{amsart}
\usepackage{cmap}
\usepackage[T1]{fontenc}
\usepackage[normalem]{ulem}
\usepackage{graphicx}
\usepackage{epstopdf}
\usepackage{subcaption}
\usepackage{setspace,mathrsfs,amsthm,amsmath,amssymb,amsfonts,lipsum,appendix,mathtools,cite,tabularx,fontenc}
\usepackage[shortlabels]{enumitem}

\makeatletter
\newtheorem*{rep@theorem}{\rep@title}
\newcommand{\newreptheorem}[2]{%
	\newenvironment{rep#1}[1]{%
		\def\rep@title{#2 \ref{##1}}%
		\begin{rep@theorem}}%
		{\end{rep@theorem}}}
\makeatother

\newtheorem{theorem}{Theorem}[section]
\newtheorem{lemma}[theorem]{Lemma}

\newtheorem{corollary}[theorem]{Corollary}
\newtheorem{proposition}[theorem]{Proposition}
\theoremstyle{definition}
\newtheorem{remark}[theorem]{Remark}
\newtheorem{definition}[theorem]{Definition}

\usepackage[nointegrals]{wasysym}
\usepackage[misc]{ifsym}
\usepackage[dvipsnames]{xcolor}
\definecolor{ao}{rgb}{0.0, 0.5, 0.0}
\definecolor{lasallegreen}{rgb}{0.03, 0.47, 0.19}
\usepackage{hyperref}
\hypersetup{
	colorlinks=true,
	linkcolor=blue,
	filecolor=magenta,      
	urlcolor=violet,
	citecolor=OrangeRed,
}
\usepackage[margin=0.75in]{geometry}

\let\oldnorm\norm
\def\norm{\@ifstar{\oldnorm}{\oldnorm*}}

\newcommand{\al} {\alpha}
\newcommand{\pa} {\partial}

\newcommand{\ga} {\gamma}
\newcommand{\Ga} {\Gamma}

\newcommand{\Om} {\Omega}

\newcommand{\R}{{\mathbb R}}

\usepackage[hyperpageref]{backref}

\usepackage{orcidlink}

\usepackage{a4wide}
\usepackage{csquotes}

\usepackage{xpatch}
\makeatletter   
\xpatchcmd{\@tocline}
{\hfil\hbox to\@pnumwidth{\@tocpagenum{#7}}\par}
{\ifnum#1<0\hfill\else\dotfill\fi\hbox to\@pnumwidth{\@tocpagenum{#7}}\par}
{}{}

\def\l@subsection{\@tocline{2}{0pt}{2pc}{6pc}{}}
\def\l@subsubsection{\@tocline{3}{0pt}{8pc}{8pc}{}}
\makeatother

\usepackage{amsfonts}
\usepackage{amssymb,amsthm}
\usepackage{amsmath}

\allowdisplaybreaks
\usepackage{mathtools}
\mathtoolsset{showonlyrefs}

\numberwithin{equation}{section}

\usepackage{setspace}

\usepackage{enumitem}
\setlist{nosep}

\begin{document}
\singlespacing
\title{Neumann domains of planar analytic eigenfunctions}

\author[T.V.~Anoop]{T.V.~Anoop}
\author[V.~Bobkov]{Vladimir Bobkov}
\author[M.~Ghosh]{Mrityunjoy Ghosh}

\address[T.V.~Anoop]{\newline\indent
	Department of Mathematics,
	Indian Institute of Technology Madras, 
	\newline\indent
	Chennai 36, India
	\newline\indent
	\orcidlink{0000-0002-2470-9140} 0000-0002-2470-9140 
}
\email{anoop@iitm.ac.in}

\address[V.~Bobkov]{\newline\indent
	Institute of Mathematics, Ufa Federal Research Centre, RAS,
	\newline\indent 
	Chernyshevsky str. 112, 450008 Ufa, Russia
	\newline\indent 
	\orcidlink{0000-0002-4425-0218} 0000-0002-4425-0218
}
\email{bobkov@matem.anrb.ru}

\address[M.~Ghosh]{\newline\indent
	Tata Institute of Fundamental Research,
 \newline\indent
	Centre for Applicable Mathematics, 
	\newline\indent
	Sharadanagar, Bengaluru 560065, India
	\newline\indent
	\orcidlink{0000-0003-0415-2821} 0000-0003-0415-2821 
}
\email{ghoshmrityunjoy22@gmail.com}


\begin{abstract}
    Along with the partition of a planar bounded domain $\Omega$ by the nodal set of a fixed eigenfunction of the Laplace operator in $\Omega$, one can consider another natural partition of $\Omega$ by, roughly speaking, gradient flow lines of a special type (separatrices) of this eigenfunction. Elements of such partition are called Neumann domains and their boundaries are Neumann lines. When the eigenfunction is a Morse function, this partition corresponds to the Morse--Smale complex and its fundamental properties have been systematically investigated by Band \& Fajman \cite{band2016topological}. Although, in the case of general position, eigenfunctions are always of the Morse type, particular eigenfunctions can possess degenerate critical points. In the present work, we propose a way to characterize Neumann domains and lines of an \textit{arbitrary} eigenfunction. 
    Instead of requiring the nondegeneracy of critical points of the eigenfunction, its real analyticity is principally used. 
    The analyticity allows for the presence of degenerate critical points but significantly limits their possible diversity.
    Even so, the eigenfunction can possess curves of critical points, which have to belong naturally to the Neumann lines set, as well as critical points of a saddle-node type. 
    We overview all possible types of degenerate critical points in the eigenfunction's critical set and provide a numerically based evidence that each of them can be observed for particular eigenfunctions. Alongside with \cite{band2016topological}, our approach is inspired by a little-known note of Weinberger \cite{weineffect} that appeared back in 1963, where a part of the Neumann line set, under the name of ``effectless cut'', was explicitly introduced and studied for the first eigenfunctions in domains with nontrivial topology. In addition, we provide an asymptotic counting of Neumann domains for a disk and rectangles in analogy with the Pleijel constant. 
\end{abstract} 

\maketitle

\begin{quote}
	\setcounter{tocdepth}{2}
	\tableofcontents
	\addtocontents{toc}{\vspace*{0ex}}
\end{quote}

\section{Introduction}\label{sec:intro}

Let $\Om\subset \R^2$ be a bounded simply or multiply connected domain. 
Denote by $\Gamma^D$ the union of some connected components of $\partial \Omega$ and let $\Gamma^N = \partial \Omega \setminus \Gamma^D$, where either $\Gamma^D$ or $\Gamma^N$ is allowed to be empty.
Assume that $\Gamma^N$ is (real) analytic. 
Consider the eigenvalue problem
\begin{equation}\label{cutproblem2d}
	\left\{   
	\begin{aligned}
		-\Delta u &= \lambda u \quad\text{in} \quad\Omega,\\
		u &=0 \quad\;\;\text{on}\quad  \Ga^D,\\
		\frac{\partial u}{\partial \nu}&=0 \quad\;\;\text{on}\quad  \Gamma^N,
	\end{aligned}
	\right.
\end{equation}
where $\nu$ denotes a unit normal vector to $\Gamma^N$.
We further assume that $\Gamma^D$ is sufficiently regular to guarantee that eigenfunctions are (real) analytic in a neighborhood of $\overline{\Omega}$, see Remark~\ref{rem:regularity} below.
As is well known, the spectrum of \eqref{cutproblem2d} consists of a discrete sequence of eigenvalues $\{\lambda_k\}$ such that
$$
0 \leq \lambda_1 < \lambda_2 \leq \lambda_3 \leq \ldots \leq \lambda_k \leq \ldots, 
\quad 
\text{and}~ \lambda_k \to +\infty ~~\text{as}~~ k \to +\infty.
$$

Let $\Ga_0$ stand for the outer boundary of $\Omega$. 
Assuming that $\Ga_0 \subset \Ga^D$ and that there is another connected component $\Ga_1 \subset \Gamma^D$, \textsc{Weinberger} \cite{weineffect} studied the \textit{first} eigenfunction $u$ of \eqref{cutproblem2d} and
established the existence of a curve $\widetilde{\gamma} \subset \overline{\Omega}$
with the following properties:
\begin{enumerate}
\item $\widetilde{\gamma}$ consists of a finite number  of analytic arcs of finite lengths,
\item $\partial u/\partial \nu=0$ along each of the analytic arcs,
\item $\widetilde{\gamma}$ separates $\Ga_1$ from $\Ga_0$ and from any other possible connected component of the Dirichlet boundary $\Ga^D$.
\end{enumerate}  
As a consequence, after ``cutting'' $\Omega$ into two pieces along $\widetilde{\gamma}$, the first eigenvalue (fundamental tone) of each part under the corresponding mixed Dirichlet-Neumann boundary conditions remains equal to $\lambda_1$. 
Because of this reason, \textsc{Weinberger} called $\widetilde{\gamma}$ an \textit{effectless cut}.
This result was used by \textsc{Hersch} \cite{hers} to prove the following fact:

\begin{displayquote}
Among all multiply connected membranes which are fixed along their outer boundary $\Gamma_0$ and one inner boundary component $\Gamma_1$ (and otherwise free), with given measure $A$ and lengths $L_{\Gamma_0}$, $L_{\Gamma_1}$ satisfying $L_{\Gamma_0}^2-L_{\Gamma_1}^2 = 4\pi A$, the annular membrane has maximal first eigenvalue.
\end{displayquote}

The proof of \textsc{Weinberger} is based on analysis of the gradient-descent system of $u$, 
\begin{equation}\label{eq:gradflow}
\dot{z}(t) = -\nabla u(z(t)), \quad t \in \mathbb{R},
\end{equation}
where $z=(x,y)$.
Thanks to the regularity of eigenfunctions, the Cauchy problem for the equation \eqref{eq:gradflow} with an initial condition $z(0)=z_0 \in \Omega$ has a unique solution (flow line).
By construction, the effectless cut $\widetilde{\gamma}$ is the union of some dividing flow lines (separatrices) of $u$ and arcs of critical points of $u$, and we will discuss this construction in more detail below.

About 50 years later, independently from \cite{weineffect} and each other, \textsc{Zelditch} \cite{zelditch2013eigenfunctions} and \textsc{McDonald \& Fulling} \cite{mcdonald2014neumann} 
(see also \cite{mcdonald2008neumann}) 
raised interest and importance of studying the decomposition of $\Omega$ by means of  dividing flow lines of \textit{higher} eigenfunctions $u$. 
Assuming that $u$ is a Morse function, this decomposition corresponds to the Morse--Smale complex of $u$, which is defined via stable and unstable manifolds of saddle points. 
This decomposition is called the Neumann partition.
The connected components of this decomposition and their boundaries are called \textit{Neumann domains} and \textit{Neumann lines}, respectively.
Such a Neumann partitioning of $\Omega$ has some advantages over the classical partitioning of $\Omega$ by the nodal set of $u$. 
For example, the Neumann domains demonstrate better stability to perturbations of $u$ along the eigenspace, see a discussion in \cite{mcdonald2014neumann}.
At the same time, properties of Neumann domains and lines are generally harder to study than those of nodal domains.
Several fundamental properties of Neumann partitioning have been recently investigated by
\textsc{Alon, Band, Bersudsky, Cox, Egger, Fajman, Taylor} in 
\cite{alon2020neumann,band2016topological,band2020defining,band2021spectral}, and we discuss some of them in more detail below.

Although, in the case of general position, every eigenfunction is a Morse function \cite{uhlenbeck1976generic}, even the simplest domains such as a disk and square do possess eigenfunctions with degenerate critical points. 
Moreover, some eigenfunctions might not even have saddle points, e.g., when they are radial.
Consequently, the main definitions provided in \cite{mcdonald2014neumann,band2016topological,band2020defining} and most of the results obtained therein cannot directly cover non-Morse eigenfunctions.

The aim of the present work is to extend the definitions and some results from \textsc{Band \& Fajman}  \cite{band2016topological} 
to \textit{arbitrary} eigenfunctions of \eqref{cutproblem2d} by
employing several ideas of the effectless cut approach of \textsc{Weinberger} \cite{weineffect}. More precisely, we generalize and improve the aforementioned works in several directions as highlighted below:
\begin{enumerate}
    \item 
    \underline{\textit{Morse to analytic}}: 
    Our eigenfunctions are allowed to have degenerate critical points, which brings more complexity to the structure of Neumann domains and Neumann lines.
    Inspired by \cite{weineffect}, we overcome most of the difficulties that arise due to a possible presence of degenerate critical points of an eigenfunction $u$ by relying on the analyticity of $u$.
The analyticity, together with the equation satisfied by $u$, yields additional information on critical points of $u$ and the lengths of its flow lines.
Despite this, an eigenfunction can possess curves of critical points, which have to belong naturally to the Neumann lines set, as well as isolated degenerate critical points of a saddle-node type (see \cite{chenmyrtaj2019} and Section~\ref{sec:classification}).
In the present Euclidean settings, the analyticity always holds,
thereby leading to somewhat more robust definitions of Neumann domains and lines. 
See \cite[Remark~3.3]{band2016topological} for a related discussion.    

    \item 
    \underline{\textit{Mixed boundary conditions on multiply connected domains}}: 
    Our eigenfunctions are allowed to satisfy either Dirichlet or Neumann boundary conditions on each connected component of $\partial \Omega$. 
    The Dirichlet part $\Gamma^D$ and Neumann part $\Gamma^N$ of $\partial \Omega$ have different impacts on the structure of the Neumann line set. Thus, a separate consideration of the behavior of Neumann lines in  neighborhoods of $\Gamma^D$ and $\Gamma^N$ is required.
    In particular, although $\Gamma^N$ should naturally belong to the Neumann line set, it might not consist of dividing flow lines. 
    The geometry of Neumann domains for the Dirichlet eigenfunctions that are Morse has been studied in 
    \cite{band2016topological}, and we refer to \cite[Remark~3.14]{band2016topological} for a brief comment about the Neumann eigenfunctions. 

    \item 
    \underline{\textit{Neumann domain counting}}:
    We also provide some results on the asymptotic number of Neumann domains in a disk and rectangles, in the spirit of \textsc{Pleijel}'s results \cite{Pleijel} on the asymptotic number of nodal domains.
    The consideration of these model domains leads to some natural conjectures. 
    We refer to \cite{mcdonald2014neumann,band2021spectral,band2016topological,band2020defining} for related analysis.
\end{enumerate}

Even in the analytic settings, the structure of the critical set and properties of critical points of eigenfunctions are not completely understood.
We refer to \cite{magnanini2016introduction,massimo2021number} for surveys on critical points of solutions of elliptic equations, to 
\cite{alessandrini1992index,arango2010critical,deng2023number,judge2022some,JudgeCPDE,Nadirashvili, Alberti, Grossi,paganimasciadri1993remarks}  for some general results on the number and types of critical points, and to 
\cite{arango2006morse,feehan2020morse,grossi2020morse} for Morse and Morse--Bott type aspects of degenerate critical points of solutions.

\medskip
The rest of the paper is organized as follows.
In Section~\ref{sec:properties}, we provide auxiliary results on properties of eigenfunctions, their critical points and flow lines.
In Section~\ref{sec:mainresults}, we introduce some definitions generalizing those from \cite{mcdonald2014neumann,band2016topological,weineffect} and formulate our main results, which we collect in Theorem~\ref{thm:main-properties}.
Section~\ref{sec:GWL} contains further
auxiliary results needed for the proof of Theorem~\ref{thm:main-properties}, the latter being proved in Section~\ref{sec:proofs-main}.
In Section~\ref{sec:trivialcrit}, we discuss the presence of degenerate critical points in the eigenfunction's landscape.
In Section~\ref{sec:counting}, we study the Neumann domain counting for a disk and rectangles.
Finally, in Section~\ref{sec:remarks}, we discuss a few natural remarks and problems.

\section{Properties of eigenfunctions and their flow lines}\label{sec:properties}

Hereinafter, we denote by $u$ an eigenfunction of \eqref{cutproblem2d} corresponding to an eigenvalue $\lambda > 0$. 
Thus, $u$ is \textit{not} a constant function, since otherwise there is nothing to investigate.

\begin{remark}\label{rem:regularity}
Recall that $\Gamma^N$ is assumed to be analytic, and $\Gamma^D$ is sufficiently regular to guarantee that $u$ is analytic in a neighborhood of $\overline{\Omega}$.
For example, if $\Gamma^D$ is analytic, then it satisfies this assumption, which is a classical result in the regularity theory. 
To provide explicit references, we cite, e.g., \cite[Remark, p.~668]{nirenberg1955remarks} for the $C^\infty(\overline{\Omega})$-regularity of $u$ and then \cite[Chapter~8, Theorem~1.2]{lions2012non} for the analyticity. 
We also refer to \cite[pp.~341-343]{magenes1958} for a similar statement, and to \cite{toth-zelditch} for a discussion. 
Note that these references cover much more general settings; moreover,  
\cite[Chapter~8, Theorem~1.2]{lions2012non} does not require $u$ to be a solution of \eqref{cutproblem2d}.
Another example of an appropriate $\Gamma^D$ includes the model case when $\Omega$ is a rectangle with $\partial \Omega = \Gamma^D$, since eigenfunctions are given in a closed form. 
\end{remark}

\begin{remark}\label{rem:regularity2}
	Since $u$ is analytic in a neighborhood of $\overline{\Omega}$, it continues to satisfy the equation in \eqref{cutproblem2d} in this neighborhood, which follows from the unique continuation property of the analytic function $\Delta u + \lambda u$.
\end{remark}

Throughout the text, \textit{local minimum points}, \textit{local maximum points} (\textit{local extremum points} when referring to both), and \textit{saddle points} will be regarded with respect to the extension of $u$ to a neighborhood of $\overline{\Omega}$. 
In particular, they are \textit{critical points} of $u$. 
We will use the notation

\begin{align}
        \mathcal{C}
&=
\{z\in \overline{\Om}:~ \nabla u(z)=0\},\\
	\mathcal{M}_-
	&=
	\{
	q\in \mathcal{C}:~
	q~\text{is a local minimum point of}~u
	\},\\
	\mathcal{M}_+
	&=
	\{
	p\in \mathcal{C}:~
	p~\text{is a local maximum point of}~u
	\},\\
	\mathcal{S}
	&=
	\{
	r\in \mathcal{C}:~
	r~\text{is a saddle point of}~u
	\}.
\end{align}

\begin{remark}\label{rem:boundary_critical}
Let us provide some comments: 
\begin{enumerate}[label={\rm(\roman*)}]
	\item\label{rem:boundary_critical:Mpmempty} 
 $\mathcal{M}_+ \cap \mathcal{M}_- =\emptyset$, which follows from our assumption that $u$ is not constant in $\Omega$ together with the analyticity of $u$. 

	\item $\mathcal{M}_\pm \cap \Gamma^D = \emptyset$, i.e., any $z \in \Gamma^D$ is either a regular point or a saddle point of $u$, see Lemma~\ref{lem:crit}~\ref{lem:crit:singsadleisolated}.
\end{enumerate}
\end{remark}

\subsection{Classification of critical points}\label{sec:classification}

As noted above, an eigenfunction $u$ might have \textit{degenerate} critical points, i.e., those critical points at which the Hessian determinant is zero.
However, the fact that $u$ is analytic and satisfies \eqref{cutproblem2d} imposes certain restrictions on possible degeneracy of its critical points, compared to that for general $C^\infty$-functions.
In this subsection, we overview all types of critical points that $u$ can possess.

Let $z_0 = (x_0,y_0)$ be a critical point of the eigenfunction $u$. 
We say that $u$ and a function $v: \mathbb{R}^2 \to \mathbb{R}$ at the points $z_0$ and $(0,0)$, respectively, are \textit{right-equivalent} if there exist neighborhoods $U_1$ of $z_0$ and $U_2$ of $(0,0)$ and a $C^1$-diffeomorphism $\phi: U_1 \to U_2$ such that $\phi(z_0) = (0,0)$ and $v(z) = u(\phi^{-1}(z))$ for any $z \in U_2$; cf. \cite[p.~58]{poston2014catastrophe}.
For brevity, we will omit to mention that the right-equivalence takes place \textit{at the points $z_0$ and $(0,0)$}.

The following classification of critical points of $u$ holds:
\begin{enumerate}[\rm(I)]
\item $z_0$ is \textit{non-degenerate}, i.e., the Hessian determinant of $u$ at $z_0$ does not vanish. 
The Morse lemma says that $u$ is right-equivalent to one of the functions
\begin{equation}\label{eq:nondegen}
	(x,y) \mapsto u(z_0) \pm x^2 \pm y^2.
\end{equation}
We can have either $\Delta u(z_0) \neq 0$ or $\Delta u(z_0) = 0$, and, respectively, either $u(z_0) \neq 0$ or $u(z_0) = 0$ by \eqref{cutproblem2d}.

\item $z_0$ is \textit{semi-degenerate}, i.e., the Hessian matrix of $u$ at $z_0$ has rank $1$. 
We discuss in Remark~\ref{rem:Hadamard} that $u$ is right-equivalent to one of the functions 
\begin{align}
	\label{eq:semidegen1}
	&(x,y) \mapsto u(z_0) \pm y^2,\\
	\label{eq:semidegen2}
	&(x,y) \mapsto u(z_0) \pm x^k \pm y^2
	\quad \text{for a natural number}~ k \geq 3.	
\end{align}
We have $\Delta u(z_0) \neq 0$ and also $u(z_0) \neq 0$ by \eqref{cutproblem2d}.

\item $z_0$ is \textit{fully degenerate}, i.e., the Hessian matrix of $u$ at $z_0$ is a zero matrix.
We discuss in Remark~\ref{rem:cheng} that 
there exists a natural number $M \geq 3$ such that $u$ is right-equivalent to a homogeneous harmonic polynomial of degree $M$, which can be written in the polar coordinates $(\rho,\vartheta)$ as
\begin{equation}\label{eq:degen}
	(\rho,\vartheta) \mapsto \rho^M \sin(M \vartheta).
\end{equation}
We have $\Delta u(z_0) = 0$ and also $u(z_0) = 0$ by \eqref{cutproblem2d}. 
\end{enumerate}

\noindent
Throughout the text, the functions \eqref{eq:nondegen}, \eqref{eq:semidegen1}, \eqref{eq:semidegen2}, \eqref{eq:degen} will be called \textit{prototypical}. 
See Figure~\ref{fig:levelsets} for their level sets.

\begin{figure}[t]
  \begin{subfigure}{0.31\textwidth}
    \includegraphics[width=\linewidth]{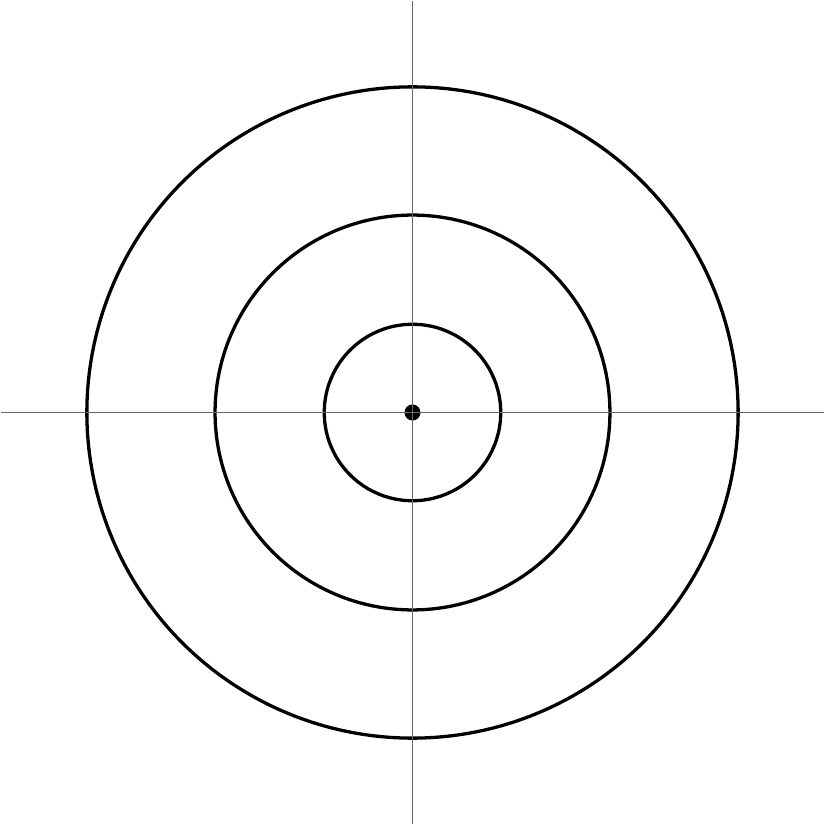}
    \caption{$x^2+y^2$} \label{fig:1a}
  \end{subfigure}%
  \hspace*{\fill}
  \begin{subfigure}{0.31\textwidth}
    \includegraphics[width=\linewidth]{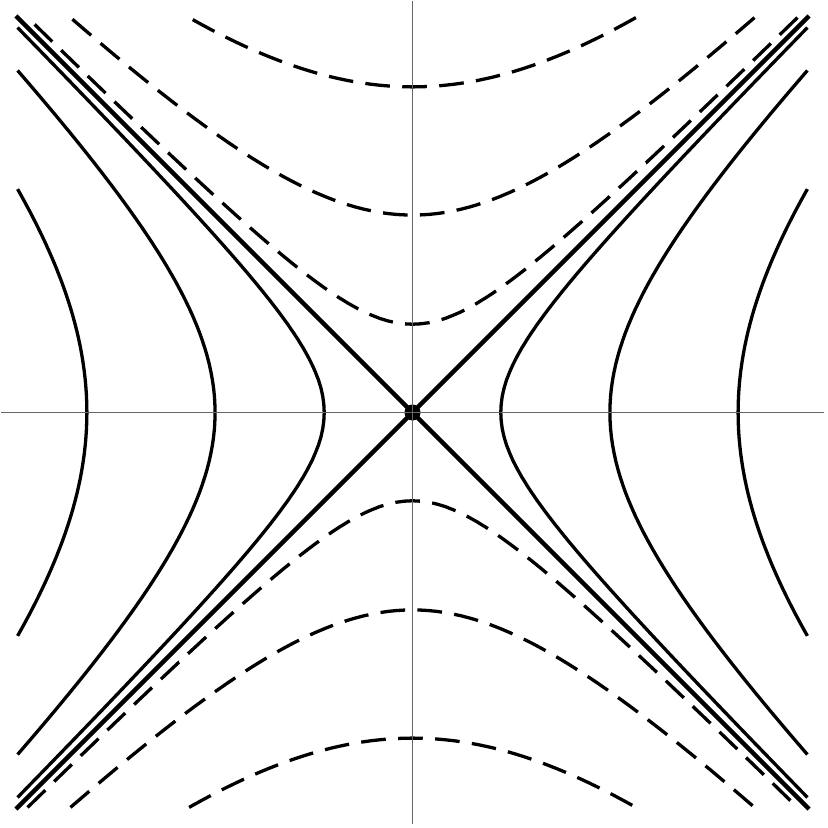}
    \caption{$x^2-y^2$} \label{fig:1b}
  \end{subfigure}%
  \hspace*{\fill}
  \begin{subfigure}{0.31\textwidth}
    \includegraphics[width=\linewidth]{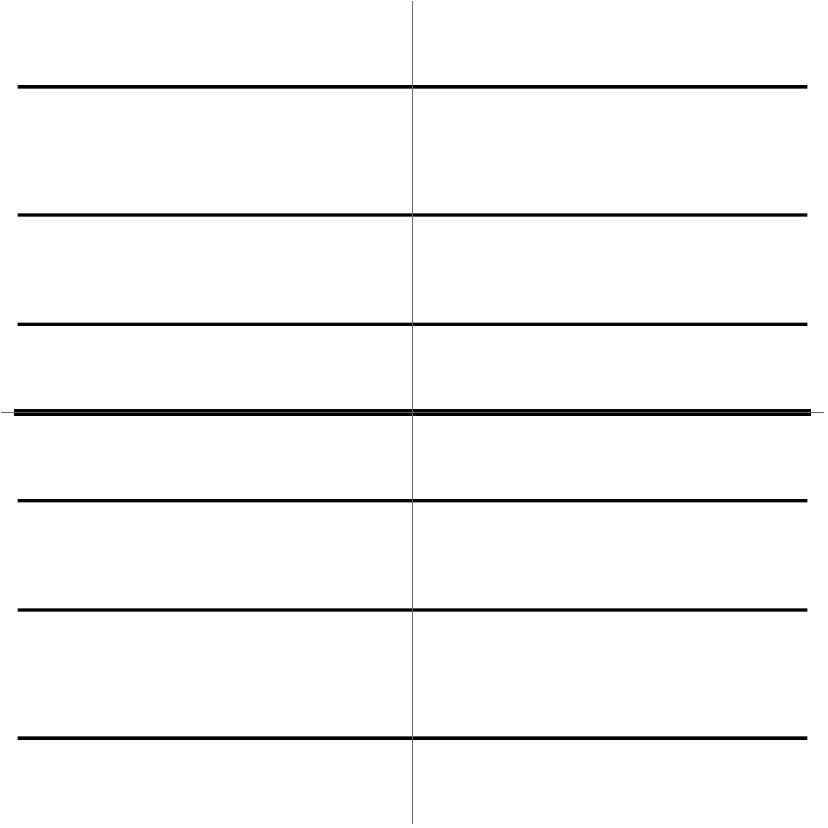}
    \caption{$y^2$} \label{fig:1c}
    \end{subfigure}
    \\[1em]
   \begin{subfigure}{0.31\textwidth}
    \includegraphics[width=\linewidth]{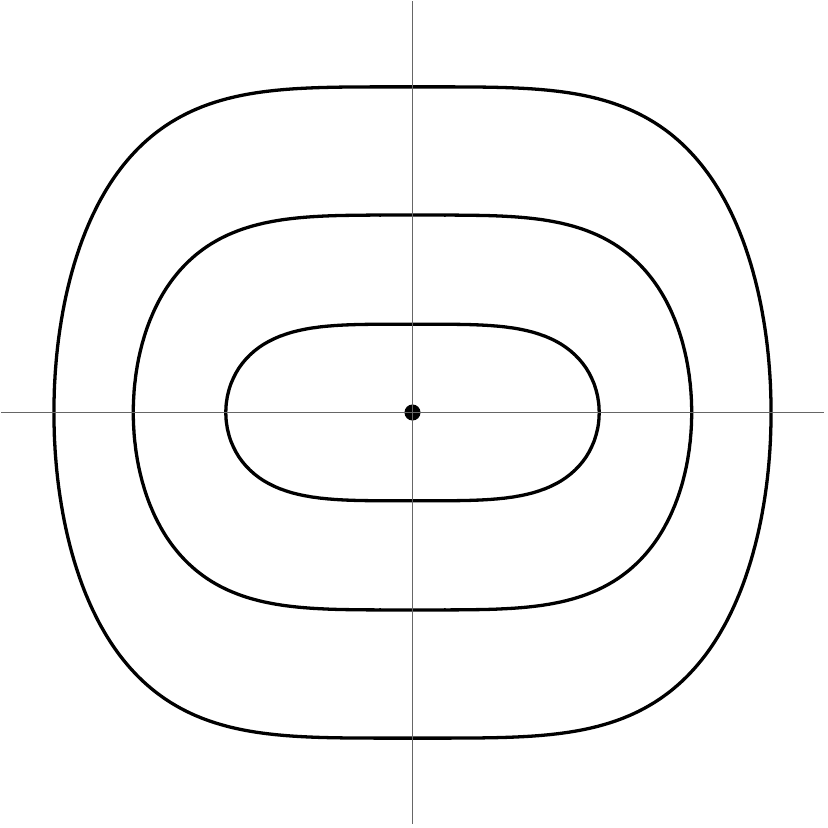}
    \caption{$x^4+y^2$} \label{fig:1d}
  \end{subfigure}%
  \hspace*{\fill}
  \begin{subfigure}{0.31\textwidth}
    \includegraphics[width=\linewidth]{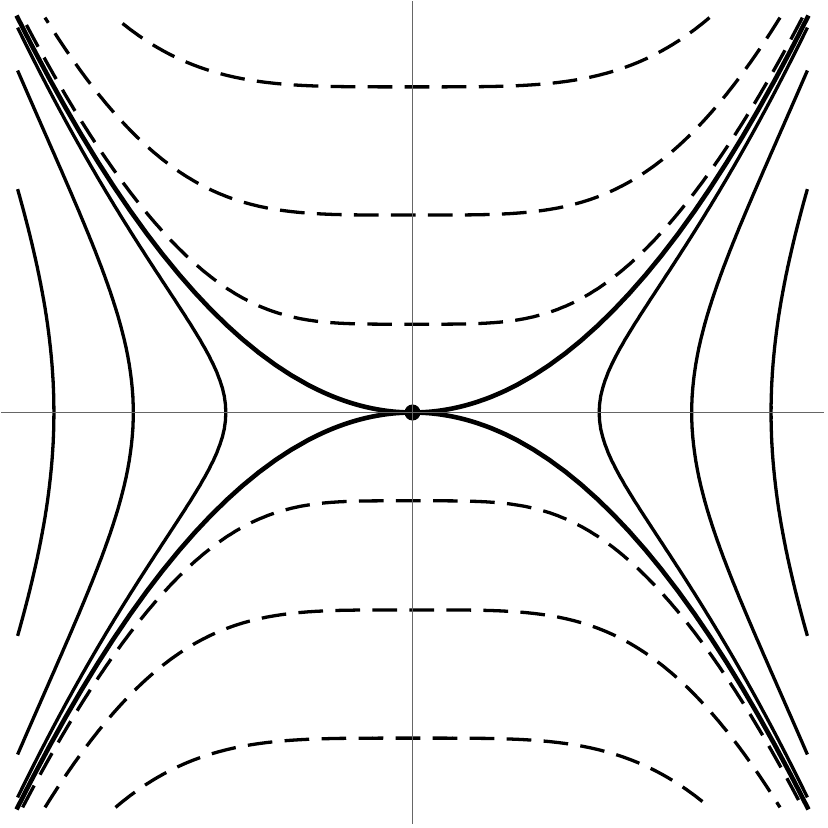}
    \caption{$x^4-y^2$} \label{fig:1e}
  \end{subfigure}%
  \hspace*{\fill} 
  \begin{subfigure}{0.31\textwidth}
    \includegraphics[width=\linewidth]{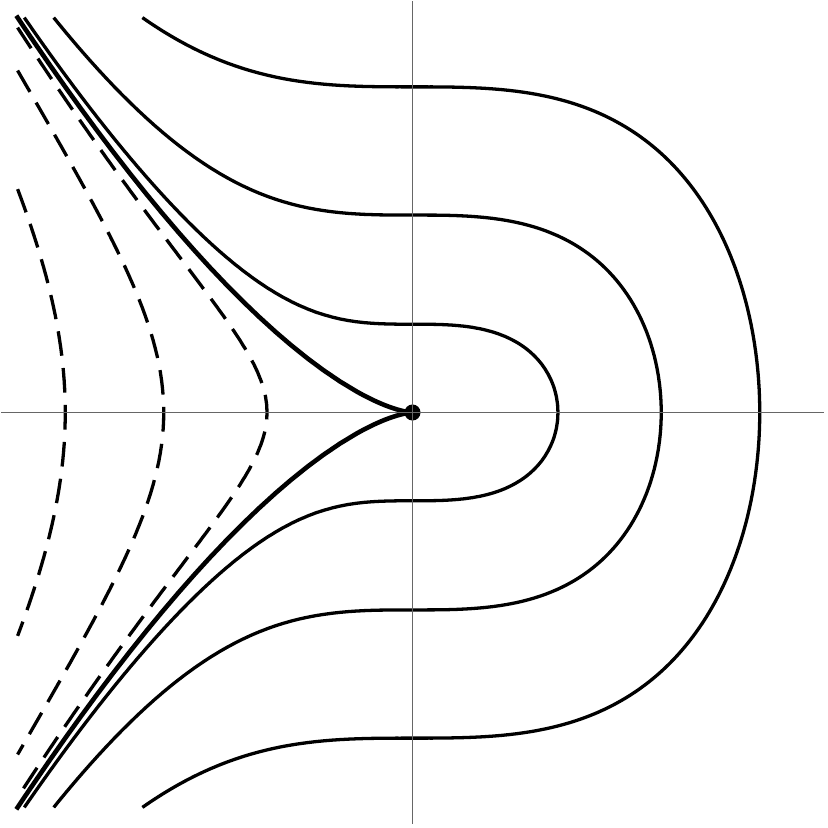}
    \caption{$x^3+y^2$} \label{fig:1f}
    \end{subfigure}
    \\[1em]
  \begin{subfigure}{0.31\textwidth}
    \includegraphics[width=\linewidth]{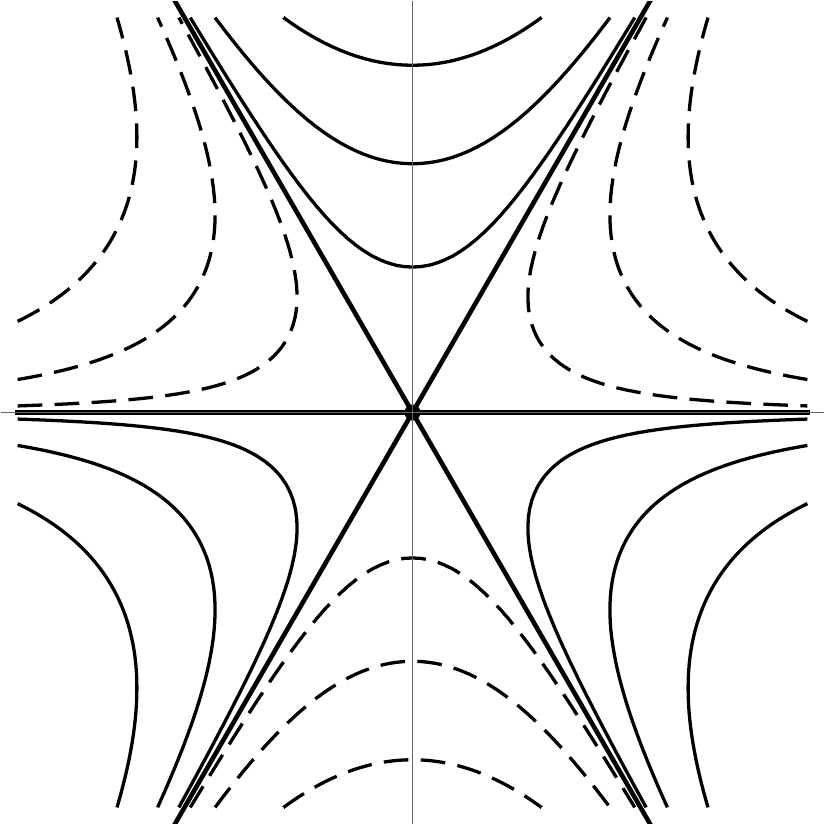}
    \caption{$\rho^3 \sin(3 \vartheta)$} \label{fig:1g}
    \end{subfigure}
    
\caption{Level sets of some prototypical functions. Thick solid lines and the origin -- the zero level set (nodal set), solid lines - level sets with levels $>0$, dashed lines - level sets with levels $<0$.} \label{fig:levelsets}
\end{figure}

\begin{remark}\label{rem:Hadamard}
Let $z_0$ be a semi-degenerate critical point of $u$. 
According to the splitting lemma (a.k.a.~the Morse lemma with parameters, see, e.g., \cite[Theorem~4.5]{poston2014catastrophe}, \cite[Theorem~4]{feehan2020morse}, and references therein), $u$ is right-equivalent to one of the functions 
$$
(x,y) \mapsto u(x_0,y_0) + h(x) \pm y^2,
$$ 
where $h$ is an analytic function (defined in a neighborhood of the origin in $\mathbb{R}$) such that $h(0)=h'(0)=h''(0)=0$.
If $h \equiv 0$, then $u$ is right-equivalent to (one of the functions) $(x,y) \mapsto u(x_0,y_0) \pm y^2$.
If $h \not\equiv 0$, then the analyticity of $h$ implies the existence of a natural number $k \geq 3$ such that $h^{(k)}(0) \neq 0$.
It is a corollary of the Hadamard lemma (see, e.g., \cite[Theorem~4.4]{poston2014catastrophe} for an explicit reference) that $h$ is right-equivalent to
$$
x \mapsto \pm x^k.
$$
We conclude that, in this case, $u$ is right-equivalent to
$$
(x,y) \mapsto 
u(x_0,y_0) \pm x^k \pm y^2
\quad \text{for a natural number}~ k \geq 3.
$$
\end{remark}

\begin{remark}\label{rem:terminology}
	According to the terminology from \cite{alessandrini1992index} and up to multiplication by $-1$,
	in the case of $(x,y) \mapsto u(z_0) + x^k - y^2$ with even $k \geq 2$, the saddle point $z_0$ is called \textit{simple saddle point}, and in the case of $(x,y) \mapsto u(z_0) + x^k - y^2$ with odd $k \geq 3$, the saddle point $z_0$ is called \textit{trivial point}. 
	The latter notion corresponds to the fact that, in a neighborhood of $z_0$, the topology of the level set $u^{-1}(c)$ does not change as $c$ crosses the critical level $u(z_0)$.
	In the theory of dynamical systems, simple saddle points are known as \textit{topological saddles}, and trivial points are known as \textit{saddle-nodes}, see, e.g., \cite{perko2013differential}.
	In fact, saddle-nodes (trivial points)  cause the main difficulties in the investigation of the Neumann partitioning; see  Section~\ref{sec:classification-of-manifolds}.
\end{remark}

\begin{remark}\label{rem:cheng}
By \cite[Theorem~1 and Eq.~($5^{\prime\prime}$)]{hartman1953local}, if $u$ vanishes at $z_0$ with order $M-1$, then, up to a rotation and scaling, the function $z \mapsto u(z+z_0)$ is equal to $(\rho,\vartheta) \mapsto \rho^M \sin(M \vartheta) + o(\rho^M)$ whenever $\rho$ is sufficiently small. 
Then \cite[Corollary, p.~202]{kuiper} implies the right-equivalence of $u$ and $(\rho,\vartheta) \mapsto \rho^M \sin(M \vartheta)$. 
Alternatively, observing that both $u_x$ and $u_y$ also satisfy the \textit{equation} in \eqref{cutproblem2d}, the right-equivalence follows by applying \cite[Theorem~1]{hartman1953local} to $u_x$ and $u_y$ and using \cite[Corollary~2.9.4]{pagani-thesis}.
\end{remark}

\begin{remark}\label{rem:isolated-critical-points}
A critical point $z_0$ is always \textit{isolated} except for the case when $u$ is right-equivalent to $(x,y) \mapsto u(z_0) \pm y^2$. 
In this exceptional case, $z_0$ belongs to an analytic curve $\theta$ of critical points, see Lemma~\ref{lem:isol}.
Note also that such $z_0$ is a critical point of the Morse--Bott type, i.e., $u$ is degenerate in a tangent direction and non-degenerate in a normal direction to $\theta$ at $z_0$, cf.\ the proof of Lemma~\ref{lem:morse-bott}.
\end{remark}

\begin{remark}
	Let us note that every type of critical points from the classification above can occur in the critical set of an eigenfunction $u$.
	Indeed, let $\Omega$ be a disk and let $\partial \Omega = \Gamma^D$.
	Then the fourth eigenfunction has non-degenerate critical points: four local extrema and one saddle at the center of the disk.
	The sixth eigenfunction (which is the second radial eigenfunction) has a circle of critical points, which are therefore semi-degenerate.
	The seventh eigenfunction has a fully degenerate critical point (monkey saddle) at the center of the disk. See Figure~\ref{fig:eigenfunctions} and also Section~\ref{subsect:disk} for analytics on the degeneracy of critical points.
	The appearance of three remaining types of critical points, namely, isolated semi-degenerate critical points (see Figures~\ref{fig:1d}, \ref{fig:1e}, \ref{fig:1f}), is discussed in Section~\ref{sec:trivialcrit}, and we refer to \cite{chenmyrtaj2019} where it is shown that saddle-nodes can occur already for Dirichlet eigenfunctions in a square.
\end{remark}

\begin{figure}[t]
  \begin{subfigure}[c]{0.31\textwidth}
    \includegraphics[width=\linewidth]{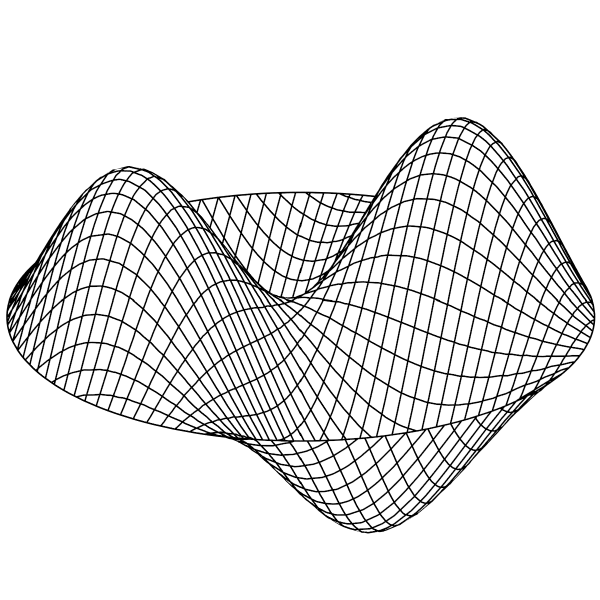}
    \caption{$u_4$} \label{fig:2a}
  \end{subfigure}%
  \hspace*{\fill}
  \begin{subfigure}[c]{0.31\textwidth}
    \includegraphics[width=\linewidth]{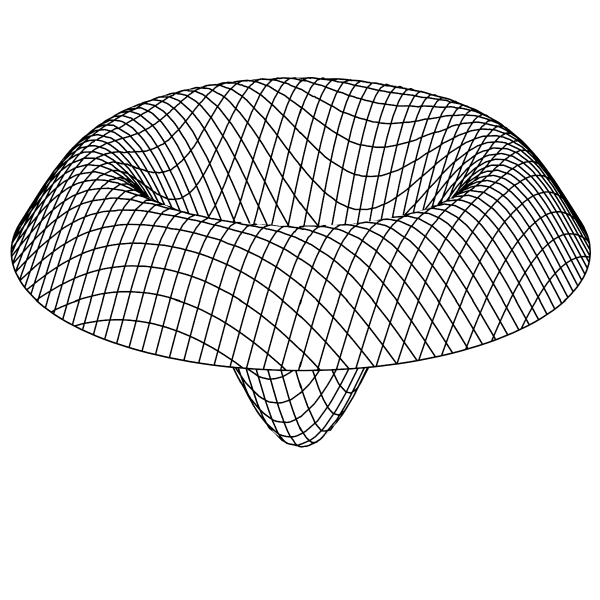}
    \caption{$u_6$} \label{fig:2b}
  \end{subfigure}%
  \hspace*{\fill}
  \begin{subfigure}[c]{0.31\textwidth}
    \includegraphics[width=\linewidth]{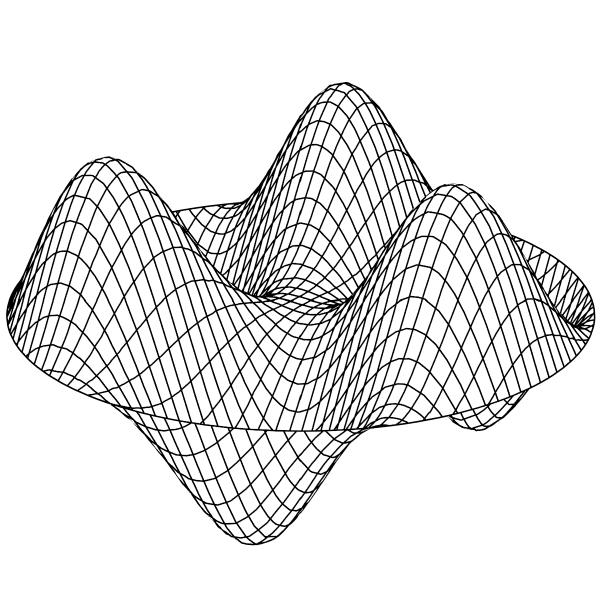}
    \caption{$u_7$} \label{fig:2c}
    \end{subfigure}
\caption{Fourth, sixth, and seventh eigenfunction of the Dirichlet Laplacian in a disk.} 
\label{fig:eigenfunctions}
\end{figure}

\begin{remark}
    One could wonder whether the radial eigenfunctions are the only possible eigenfunctions having a curve of critical points.
    In the recent work \cite{enciso2023schiffer}, it is shown that there exist \textit{nonradial}, smooth, doubly connected domains $\Omega$ which admit Neumann eigenfunctions that are locally constant on the boundary.
    That is, $\partial \Omega$ is a union of two \textit{noncircular} curves of critical points of certain eigenfunctions. 
    On the other hand, the first eigenfunction in a simply connected domain has only isolated critical points, see \cite[Corollary~3.4]{alessandrini1992index} and also \cite{arango2010critical}.
\end{remark}

\begin{remark}\label{rem:level-sets-diffeomorhpic}
The right-equivalence implies that the level sets of $u$ in a neighborhood of $z_0$ are diffeomorphic to those of a corresponding prototypical function in a neighborhood of $(0,0)$.
In particular, if $z_0$ is a saddle point of $u$, then the level set of $u$  at the level $u(z_0)$ in a neighborhood of $z_0$ consists of at most finitely many curves intersecting at $z_0$. 
\end{remark}

\begin{remark}
In the semi-degenerate case, a result related to the classification above is given in \cite[Section~3]{arango2006morse}, and it can also be deduced from Remark~\ref{rem:level-sets-diffeomorhpic}.
Namely, the level sets of the function $u-u(z_0)$ in a neighborhood of $z_0$ are \textit{homeomorphic} to those of one of the following functions in a neighborhood of $(0,0)$, up to multiplication of $u$ by $-1$: 
\begin{align*}
	f_a(x,y) &= x^4 + y^2, 
	\quad 
	f_b(x,y) = y^2,\\
	f_c(x,y) &= x^4 - y^2,
	\quad 
	f_d(x,y) = x^3 + y^2.
\end{align*}
\end{remark}

\medskip
After overviewing main properties of the gradient-descent system \eqref{eq:gradflow} associated with $u$ in Section~\ref{sec:gradflow}, 
we will return to the discussion of critical points of $u$ in Section~\ref{sec:classification-of-manifolds} by considering a relation between stable and unstable sets of a critical point $z_0$ of $u$ and those of the critical point $(0,0)$ of a corresponding prototypical function.

\subsection{Properties of eigenfunctions}\label{subsec:prop}
Let us recall that the set $\mathcal{C}$ denotes the \textit{critical set} of the eigenfunction $u$ (see Section~\ref{sec:properties}). 
We define the \textit{singular set} of $u$ as
\begin{equation*}
\mathcal{C}_{\mathrm{sing}}
=
\{z\in \mathcal{C}:~  u(z)=0\}.
\end{equation*}
We start with some well-known properties of $\mathcal{C}$ and $\mathcal{C}_{\mathrm{sing}}$.

\begin{lemma}\label{lem:crit}
The following assertions hold:  
\begin{enumerate}[\rm(i)]
	\item\label{lem:crit:critlevelsfinite} The set of critical levels $\{u(z):\, z\in \mathcal{C}\}$ is finite.
	\item\label{lem:crit:singularfinite} The singular set $\mathcal{C}_{\mathrm{sing}}$ is finite. 
	\item\label{lem:crit:singsadleisolated}
	Any $z_0 \in \mathcal{C}_{\mathrm{sing}}$ is an isolated critical point and also a saddle point.
\end{enumerate}\end{lemma}
\begin{proof}
The assertion~\ref{lem:crit:critlevelsfinite} follows from the Morse--Sard theorem (see, e.g., \cite[Theorem~1]{souvcek1972morse}). 
The assertion~\ref{lem:crit:singularfinite} follows from, e.g., \cite[Remark~3.1]{caffarelli1985partial}, where we consider the extension of $u$ to a neighborhood of $\overline{\Omega}$ if $\mathcal{C}_{\mathrm{sing}} \cap \partial \Omega \neq \emptyset$. 
The assertion~\ref{lem:crit:singsadleisolated} is a direct corollary of \ref{lem:crit:critlevelsfinite}, \ref{lem:crit:singularfinite}, and the maximum principles.
Indeed, if $z_0 \in \mathcal{C}_{\mathrm{sing}}$ is nonisolated, then there exists a sequence $\{z_n\} \subset \mathcal{C} \setminus \{z_0\}$ such that $z_n \to z_0$. 
According to the assertion~\ref{lem:crit:critlevelsfinite}, we have $u(z_n) = u(z_0) = 0$ for all sufficiently large $n$, i.e., $z_n \in \mathcal{C}_{\mathrm{sing}}$, which contradicts the assertion~\ref{lem:crit:singularfinite}.
If $z_0 \in \mathcal{C}_{\mathrm{sing}}$ is a local extremum point $u$, then we get a contradiction to the strong maximum principle.
\end{proof}

In the following lemma, we describe the structure of the critical set $\mathcal{C}$, see  \cite{alessandrini1992index,arango2010critical,judge2022some} and \cite{weineffect} 
for closely related results. 
We provide details for the sake of completeness.

\begin{lemma}\label{lem:isol} 
The critical set $\mathcal{C}$ is nonempty and consists of 
(at most finitely many) isolated critical points and analytic curves. 
Each curve $\theta \subset \mathcal{C}$ 
has the following properties:
\begin{enumerate}[\rm(i)]
	\item\label{lem:isol:simple} $\theta$ is simple, isolated in $\mathcal{C}$, has a finite length, and either closed or has both end points on $\Gamma^N$.
	\item\label{lem:isol:gamma-N} If $\theta$ intersects $\Gamma^N$ at infinitely many points, then $\theta$ is a connected component of $\Gamma^N$.
	\item\label{lem:isol:nonconstant} $u$ is a nonzero constant on $\theta$.
\end{enumerate}
\end{lemma}
\begin{proof}  
The nonemptiness of $\mathcal{C}$ is obvious.
If $\mathcal{C}$ consists only of isolated critical points, then $\mathcal{C}$ is a finite set and we are done.  
Let $z_0\in \mathcal{C}$ be a nonisolated critical point of $u$, that is, there exists a sequence $\{z_n\} \subset \mathcal{C} \setminus \{z_0\}$ such that $z_n \to z_0$.
We know from Lemma~\ref{lem:crit}~\ref{lem:crit:singsadleisolated} that $u(z_0)\ne 0$ and hence $\Delta u(z_0) \neq 0$. 
Without loss of generality, we may assume that $u_{yy}(z_0) \neq 0$. 
By applying the implicit function theorem (see, e.g., \cite[Theorem 2.3.1]{KrantzParks}) to the function $u_y$ at $z_0$ (and considering an 
 extension of $u$ to a neighborhood of $\overline{\Omega}$ when $z_0 \in \partial \Omega$), we obtain two open intervals $I_0$ and $J_0$ such that $z_0=(x_0,y_0)\in I_0\times J_0$ and 
\begin{enumerate}[label={\rm(\alph*)}]
	\item\label{lem:isol:proof:b} there exists an analytic function $g: I_0\to J_0$   such that $g(x_0)=y_0$ and $u_y(x,g(x))=0$ for all $x\in I_0$;
	\item\label{lem:isol:proof:c} all zeros of $u_y$ in $I_0\times J_0$ lie on the graph of $g$.
\end{enumerate}
Now we consider an analytic function $h$ on $I_0$ defined as $h(x)=u_x(x,g(x))$. 
Since any $z_n$ is a critical point of $u$, \ref{lem:isol:proof:c} implies that $h(x_n)=0$ for any sufficiently large $n$.
Using the analyticity of $h$ and the identity theorem \cite[Corollary~1.2.7]{KrantzParks}, we conclude that $h\equiv 0$ on $I_0$. 
Thus, we deduce from \ref{lem:isol:proof:b}  that $\{(x,g(x)):\, x\in  \overline{I_0}\} \cap \overline{\Omega}$ is an analytic arc in $\mathcal{C}$ that contains $z_0$. 
Let $\theta$ be the union of all analytic arcs in $\mathcal{C}$ that contain $z_0$.
By replacing $z_0$ with any other point of $\theta$ and repeating the above arguments, we conclude the following:
\begin{enumerate}[\rm(I)]
	\item\label{lem:isol:proof:i}  each point on $\theta$ lies on an analytic arc (part of $\theta$) with a finite length; 
	\item\label{lem:isol:proof:ii} $\theta$ is a closed set in $\overline{\Omega}$, and hence $\theta$ is compact;
	\item\label{lem:isol:proof:iii}  $\theta \cap \partial \Omega\subset \Gamma^N$.
\end{enumerate}
Now, using \ref{lem:isol:proof:i}, we see that $\theta$ is a simple curve. 
Moreover, using \ref{lem:isol:proof:ii} and recalling \ref{lem:isol:proof:c}, we deduce that $\theta$ has a finite length. 
If $\theta$ does not meet the boundary, then the maximality of $\theta$ ensures that $\theta$ is a closed curve. 
The implicit function theorem applied to the extension of $u$ guarantees that $\theta$ is isolated in $\mathcal{C}$. Therefore, such analytic curves are finite in number.

Combining all the facts above, we conclude the assertion~\ref{lem:isol:simple}.
Since both $\theta$ and $\Ga_N$ are analytic, the assertion~\ref{lem:isol:gamma-N} comes from the identity theorem \cite[Corollary~1.2.7]{KrantzParks}.
Finally, by Lemma~\ref{lem:crit}, $u$ is a nonzero constant on each analytic arc in $\mathcal{C}$ containing $z_0$. 
Consequently, since $\theta$ is the union of all such arcs, the assertion~\ref{lem:isol:nonconstant} follows directly.
\end{proof}

\begin{lemma}\label{lem:numsaddlesOmega}
Every saddle point of $u$ in $\overline{\Omega}$ is an isolated critical point, and hence $\mathcal{S}$ is a finite set.
\end{lemma}
\begin{proof}
Suppose $z_0 \in \overline{\Omega}$ is  a nonisolated critical point of $u$. 
In view of Remarks~\ref{rem:level-sets-diffeomorhpic} and \ref{rem:isolated-critical-points} (see also \cite[Lemma~1 and Remark~2]{arango2006morse}), for a sufficiently small $|\alpha|$ the level sets $\{z:\, u(z) = u(z_0) + \alpha\}$ in a neighborhood of $z_0$ are homeomorphic to the level sets of one of 
the functions $(x,y) \mapsto \pm y^2$ in a neighborhood of the origin.
In particular,
$\{z:\, u(z) = u(z_0) + \alpha\} = \emptyset$ for either $\alpha<0$ or $\alpha>0$, depending on the sign in front of $y^2$. 
Consequently, $z_0$ cannot be a saddle point, and this completes the proof.
\end{proof}

The following result is a direct corollary of Lemmas~\ref{lem:isol} and \ref{lem:numsaddlesOmega}. 
\begin{corollary}\label{cor:no-saddles-on-theta}
Let $\theta \subset \mathcal{C}$ be a curve of critical points of $u$. 
Then $\theta$ consists either of local minimum points or of local maximum points of $u$.
\end{corollary}

In Lemma~\ref{lem:morse-bott}, we provide further properties of curves of critical points of $u$.

\subsection{Gradient flow}\label{sec:gradflow}
Recall that the gradient flow lines (also called as integral curves, trajectories, etc.) of $u$ are defined as solutions of the Cauchy problem
\begin{equation}\label{eq:gradflow2}
\dot{z}(t) = -\nabla u(z(t)), \quad t \in I,
\qquad z(0)=z_0 \in \overline{\Omega},
\end{equation}
where $z(t)=(x(t),y(t))$ and $I = I(z_0) \subset \mathbb{R}$ stands for the \textit{maximal interval} such that the solution of \eqref{eq:gradflow2} belongs to $\overline{\Omega}$.
Denote by $\gamma(\cdot,z_0): I \to \mathbb{R}^2$ the solution of \eqref{eq:gradflow2} and recall that it exists, unique, and analytic (cf.\ \cite[Section 2.3, Remark 1]{perko2013differential}). 
Moreover, if $|\nabla u(z_0)| > 0$, then $u$ is strictly monotone along $\gamma(\cdot,z_0)$, namely, $u(\gamma(t_1,z_0)) > u(\gamma(t_2,z_0))$ for any admissible $t_1<t_2$. 
Let us explicitly note that we consider $\gamma(\cdot,z_0)$ for both positive and negative ``times'', with $\gamma(0,z_0) = z_0$. 
Evidently, if $z_0$ is a critical point of $u$, then $\gamma(t,z_0)=z_0$ for any $t \in \mathbb{R}$. 
Throughout the text, when using the collocation \textit{flow line}, we always mean the solution $\gamma(\cdot,z_0)$ starting with a \textit{regular point} $z_0$ of $u$.

\begin{remark}
\textsc{Band \& Fajman} in \cite[Section~3]{band2016topological} prefer to redefine flow lines so that they exist for all times, and if a flow line intersects $\Gamma^D$ at a point $z_0$, then it is defined to either stop at $z_0$ or emanate from $z_0$, depending on the sign of the normal derivative, and being continued by a constant vector $\nabla u(z_0)$ to all other times.
The consideration of the set $I$ in the Cauchy problem~\eqref{eq:gradflow2} makes the formal side of our analysis slightly different from that in \cite{band2016topological}, but this difference does not affect the behavior of flow lines in $\overline{\Omega}$.
\end{remark}

The following result, known as the Lojasiewicz inequality \cite{lojastopo,lojasiewicz1982}, is among fundamental properties of analytic functions.
For convenience, we state it for the eigenfunction $u$. 
\begin{lemma}\label{lem:lojas}
Let $z_0 \in \overline{\Omega}$ be a critical point of $u$. 
Then there exist a neighborhood $U$ of $z_0$, an exponent $\rho \in [1/2,1)$, and a constant $C>0$ such that 
$$
|\nabla u(z)| \geq C |u(z_0) - u(z)|^\rho 
\quad \text{for any}~ z \in U.
$$
\end{lemma}

The Lojasiewicz inequality implies that if $\gamma(t_n,z_0) \to z \in \overline{\Omega}$ for \textit{some} sequence $\{t_n\}$ diverging to $\infty$, then 
$\gamma(t_n,z_0) \to z$ for \textit{any} sequence $\{t_n\}$ diverging to $\infty$, i.e.,
$$
z=\lim_{t \to \infty}\gamma(t,z_0),
$$
the length of the whole flow line $\gamma(I,z_0)$ is finite, and $z$ is a critical point of $u$, see \cite{lojastopo,lojasiewicz1982} and also the proof of Proposition~\ref{prop:absil-x}.
Here, $\infty$ stands for either $+\infty$ or $-\infty$.
As a consequence, we have the following alternative for each direction of the flow line of any $z_0 \in \Omega$: 
\begin{equation}\label{alternative}
\parbox{\dimexpr\linewidth-10em}{%
	\strut
	either $\gamma(\cdot,z_0)$ reaches $\partial\Omega$ in a finite time, or $\gamma(\cdot,z_0)$ converges to a critical point of $u$ in $\overline{\Omega}$.
	\strut
}
\end{equation}
We will clarify the alternative \eqref{alternative} in Lemma~\ref{lem:classif-gamma}, as it will be convenient in our proofs.
Hereinafter, we will use the notation
$$
\gamma(\pm\infty,z_0) 
= 
\lim_{t \to \pm\infty}\gamma(t,z_0).
$$

\begin{remark}\label{rem:thom-conjecture}
In the case when $\gamma(\cdot,z_0)$ converges to a critical point of $u$ (either as $t \to +\infty$ or $t \to -\infty$), it is known from \cite[p.~768]{kurdyka2000proof} that the following limit exists:
$$
\lim_{t \to \infty} \frac{\dot{\gamma}(t,z_0)}{|\dot{\gamma}(t,z_0)|},
$$
that is, normalized tangents converge. 
(Recall that we deal with $\mathbb{R}^2$, while the existence of this limit is an open problem in higher dimensions.)
Consequently, one can naturally define the angle between limits of normalized tangents to two flow lines converging to the same critical point.
\end{remark}

The following lemma says that the Neumann part $\Gamma^N$ of $\partial \Omega$ can be parameterized by flow lines except at critical points. 
The result is rather evident from the very definition of $\Gamma^N$, so we omit the proof.
\begin{lemma}\label{lem:traj-neum}
Let $z_0 \in \Gamma^N$. 
Then 
$\gamma(t,z_0) \in \Gamma^N$ for all $t \in \mathbb{R}$.
\end{lemma}

The following useful corollary of Lemma~\ref{lem:traj-neum} is the observation that flow lines starting in $\Omega \cup \Gamma^D$ cannot reach $\Gamma^N$ in a finite time.
\begin{corollary}\label{cor:boundary}
	Let $z_0 \in \overline{\Omega} \setminus \Gamma^N$. 
	Then $\gamma(t,z_0) \in \overline{\Omega} \setminus \Gamma^N$ for any $t \in I$.
\end{corollary}

As a consequence of Lemmas~\ref{lem:traj-neum} and~\ref{lem:isol}, we have the following description of $\Gamma^N$.
\begin{corollary}\label{cor:boundary-GN}
$\Gamma^N$ consists of at most finitely many flow lines (together with their end points) and closed curves of critical points of $u$.
\end{corollary}

We complement Lemma~\ref{lem:traj-neum} by discussing the behavior of flow lines starting on $\Gamma^D$.
\begin{remark}\label{rem:traj-boundary}
If $z_0 \in \Gamma^D$ is a regular point of $u$, then $\gamma(\cdot,z_0)$ is defined either in $I = (-\infty,0]$ or in $I = [0,+\infty)$ depending on whether $\frac{\partial u}{\partial \nu}(z_0) < 0$ or 
$\frac{\partial u}{\partial \nu}(z_0) > 0$, respectively, where $\nu$ is pointed outwards. 
Indeed, since $u$ strictly decreases along $\gamma(\cdot,z_0)$ with increasing time, the flow line cannot meet $\Gamma^D$ again; 
moreover, such $\gamma(\cdot,z_0)$ cannot meet $\Gamma^N$ by Corollary~\ref{cor:boundary}. 
Therefore, $\gamma(\cdot,z_0) \cap \Gamma^D = \{z_0\}$, and $\gamma(t,z_0) \in \Omega$ either for any $t<0$ or for any $t>0$, depending on the sign of the normal derivative of $u$ at $z_0$.
\end{remark}

Using Corollary~\ref{cor:boundary} and Remark~\ref{rem:traj-boundary}, we clarify the alternative \eqref{alternative} as follows.
\begin{lemma}\label{lem:classif-gamma}
Let $z_0 \in \overline{\Omega}$ be a regular point of $u$.
Then $(-\infty,0] \subset I$ and/or $[0,+\infty) \subset I$, and the flow line $\gamma = \gamma(\cdot,z_0)$ has
either of the following behaviors at the boundary of the maximal time interval $I$:
\begin{enumerate}[label={(\arabic*)}]
\item Both ends of $\gamma$ are local extremum points. 
\item One end of $\gamma$ is a local extremum point and the other end is a saddle point.
\item Both ends of $\gamma$ are saddle points. 
\item One end of $\gamma$ is a local extremum point and $\gamma$ reaches $\Gamma^D$ in a finite time.
\item One end of $\gamma$ is a saddle point and $\gamma$ reaches $\Gamma^D$ in a finite time. 
\end{enumerate}
\end{lemma}
 
It is known from \cite[Theorem~3]{absil2006stable} that
a critical point $z_0 \in \overline{\Omega}$  of $u$ is a local extremum point if and only if $z_0$ is stable in the sense of Lyapunov. 
Since the ``if'' part of this result will be frequently used below, we formulate it explicitly and also include the restriction to $\overline{\Omega}$.
In what follows, $B_R(z)$ stands for the open disk of radius $R>0$ centered at $z=(x,y)$.

\begin{proposition}\label{prop:absil}
	Let $z_0 \in \overline{\Omega}$ be a local minimum point of $u$.
	Then for any $\varepsilon>0$ there exists $\delta>0$ such that if $z \in B_\delta(z_0) \cap \overline{\Omega}$, then $\gamma(t,z) \in B_\varepsilon(z_0) \cap \overline{\Omega}$ for any $t \geq 0$ and hence $\gamma(+\infty,z) \in \overline{B_\varepsilon(z_0)} \cap \overline{\Omega}$.
\end{proposition}
\begin{proof}
For $z_0 \in \Omega$ the result is given by \cite[Theorem~3]{absil2006stable}, whose proof is based on the Lemma~\ref{lem:lojas}. 
Assume that $z_0 \in \partial \Omega$ is a local minimum point of $u$.
According to Lemma~\ref{lem:crit}~\ref{lem:crit:singsadleisolated}, we have $z_0 \in \Gamma^N$.
The proof in this case goes along the same lines as in \cite[Theorem~3]{absil2006stable}, with an additional reference to Lemma~\ref{lem:traj-neum} and Corollary~\ref{cor:boundary} to consider only flow lines from $\overline{\Omega}$. 
\end{proof}

\begin{remark}
	Note that Proposition~\ref{prop:absil}
	does not require $z_0$ to be isolated and the result remains valid also for local maximum points by considering negative times.
\end{remark}

In fact, the assumption of Proposition~\ref{prop:absil} that $z_0$ is a local extremum point of $u$ can be weakened to an extent which allows to cover critical points of the saddle-node type.
\begin{proposition}\label{prop:absil-x}
	Let $z_0 \in \overline{\Omega}$ be a critical point of $u$. 
	Assume that there are two flow lines $\gamma_1, \gamma_2 \subset \overline{\Omega}$ converging to $z_0$ as $t \to +\infty$. 
	Assume that there exists a neighborhood $U$ of $z_0$ such that $u(z) \geq u(z_0)$ for any $z \in \mathcal{K}$, where $\mathcal{K} \subset \overline{\Omega \cap U}$ is an open curvilinear sector bounded by $\gamma_1$, $\gamma_2$, and $\partial U$.
	Then for any $\varepsilon>0$ there exists $\delta>0$ such that if $z \in B_\delta(z_0) \cap \mathcal{K}$, then $\gamma(t,z) \in B_\varepsilon(z_0) \cap \mathcal{K}$ for any $t \geq 0$ and hence $\gamma(+\infty,z) \in \overline{B_\varepsilon(z_0)} \cap \overline{\mathcal{K}}$.
\end{proposition}
\begin{proof}
	The proof uses Lojasiewicz's arguments in a similar way as in \cite[Section~3]{absil2006stable}.
	We provide details for the sake of clarity.
	For convenience, adding a constant, we assume that $u(z_0) = 0$.
	In view of Lemma~\ref{lem:lojas}, there exist $\sigma>0$, $\rho \in [1/2,1)$, and $C>0$ such that 
	\begin{equation}\label{eq:lojas1}
		|\nabla u(z)| \geq C |u(z)|^\rho 
		\quad \text{for any}~ z \in B_\sigma(z_0).
	\end{equation}
	Let us take any regular point $z \in B_{\sigma}(z_0) \cap \mathcal{K}$ and consider its flow line $\gamma(\cdot,z)$.
	Parametrizing $\gamma(t,z)$ by the arc-length $s$ starting from $z$ and denoting it as $\al(s)$, we have
	$$
	\dot{\al}(s) = -\frac{\nabla u(\al(s))}{|\nabla u(\al(s))|}.
	$$
	In view of \eqref{eq:lojas1}, 
	for any $s \geq 0$ such that $\al(s) \in B_{\sigma}(z_0)$
	we have
	$$
	\frac{d u}{ds}
	=
	\left<
	\nabla u, \dot{\al}(s)
	\right>
	=
	-|\nabla u|
	\leq 
	-C |u|^\rho 
	=
	-C u^\rho,
	$$
	and hence
	\begin{equation}\label{eq:lojas2}
		\frac{d (u^{1-\rho})}{ds} \leq -C (1-\rho) < 0.
	\end{equation}
	If, for some $0 \leq s_1 < s_2$,
	$\al(s) \in B_{\sigma}(z_0) \cap \mathcal{K}$ for any $s \in [s_1,s_2]$, then, integrating \eqref{eq:lojas2} over $[s_1,s_2]$, we get
	\begin{equation}\label{eq:lojas3}
		s_2-s_1
		\leq 
		\frac{1}{C(1-\rho)}
		\left(
		u^{1-\rho}(\al(s_1))
		-
		u^{1-\rho}(\al(s_2))
		\right)
		\leq
		\frac{1}{C(1-\rho)} \, u^{1-\rho}(\al(s_1)).
	\end{equation}
	Let $\varepsilon>0$ be given. Observe that it is sufficient to prove the statement for $\varepsilon < \sigma$. 
	Thanks to our assumption on $\mathcal{K}$, since $\rho < 1$ and $u$ is continuous,
	there exists $\delta \in (0,\varepsilon/2)$ such that 
	\begin{equation}\label{eq:lojas4}
		0 =u^{1-\rho}(z_0)\leq u^{1-\rho}(\xi) < \frac{C(1-\rho) \varepsilon}{2}
		\quad \text{for any}~ \xi \in B_\delta(z_0) \cap \mathcal{K}.
	\end{equation}
	Let us take any $z \in B_\delta(z_0) \cap \mathcal{K}$. 
	In view of the continuity of flow lines, 
	there exists a maximal $T \in (0,+\infty]$ such that 
	$\gamma(t,z)$ stays in $B_\varepsilon(z_0) \cap \mathcal{K}$ for all $t \in [0,T)$.
	Thus, we conclude from \eqref{eq:lojas4} and \eqref{eq:lojas3} (with $s_1=0$) that the length of the flow line $\gamma(\cdot,z)$ in $[0,T)$ is bounded by $\varepsilon/2$.
	Since $\delta<\varepsilon/2$, we have $|z_0-\gamma(T,z)| < \varepsilon$.
	Consequently, if $T$ is finite, then $\gamma(T,z) \in B_\varepsilon(z_0) \cap \overline{\mathcal{K}}$ and $\gamma(T,z)$ is a regular point. 
 Moreover, recalling that the solution of the Cauchy problem \eqref{eq:gradflow2} is unique, we get $\gamma(T,z) \not\in \gamma_1, \gamma_2$, and hence $\gamma(T,z) \in B_\varepsilon(z_0) \cap \mathcal{K}$, which contradicts the maximality of $T$.
	Therefore, we have $T=+\infty$, i.e., 
	$\gamma(t,z)$ stays inside $B_\varepsilon(z_0) \cap \mathcal{K}$ for any $t \geq 0$.
\end{proof}

Let us now provide a development of Lemma~\ref{lem:isol} concerning the properties of curves of critical points of $u$.

\begin{figure}[!ht]
\begin{center}
  \begin{subfigure}{0.4\textwidth}
    \includegraphics[width=\linewidth]{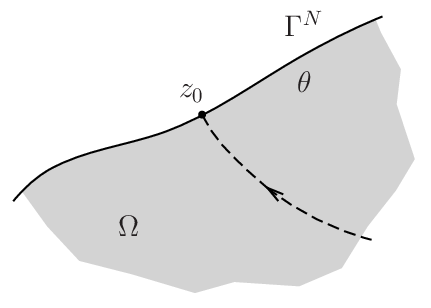}
    \caption{}
  \end{subfigure}%
  \hspace*{\fill}
  \begin{subfigure}{0.4\textwidth}
    \includegraphics[width=\linewidth]{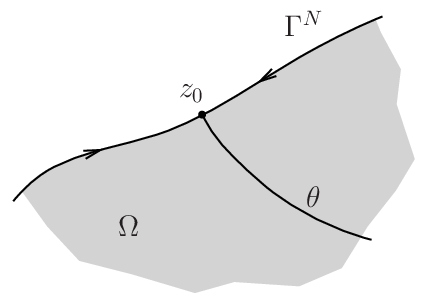}
    \caption{}
  \end{subfigure}%
\end{center}
\caption{Schematic plot for Lemma~\ref{lem:morse-bott}.}
\label{fig:theta12x}
\end{figure}

\begin{lemma}\label{lem:morse-bott}
    Let $\theta \subset \mathcal{C}$ be a curve of critical points.
    Then $\theta$ has the following properties:
    \begin{enumerate}[\rm(i)]
    \item\label{lem:morse-bott:fiber} 
    For any $z_0 \in \theta$ there exist at least one and at most two flow lines (in $\overline{\Omega}$) converging to $z_0$. 
    Any such flow line is orthogonal to $\theta$ at $z_0$. 
    \item\label{lem:morse-bott:gamma-N}
    If $\theta \cap \Gamma^N \neq \emptyset$, then either $\theta$ is a connected component of $\Gamma^N$, or $\theta \cap \Gamma^N$ consists of two points. At such points $\theta$ intersects $\Gamma^N$ orthogonally and there are flow lines in $\Gamma^N$ converging to them.
    \end{enumerate}
\end{lemma}
\begin{proof}
    Let $z_0 \in \theta$.
    Since $u$ is a \textit{nonzero constant} on $\theta$ (see Lemma~\ref{lem:isol}~\ref{lem:isol:nonconstant}), $u$ has to be non-degenerate with respect to the normal direction to $\theta$.
If not, then $z_0$ would be a fully degenerate critical point of $u$ and hence $0 = -\Delta u(z_0) = \lambda u(z_0)$, which is impossible.
    That is, $z_0$ a Morse--Bott critical point of $u$. 
    The assertion \ref{lem:morse-bott:fiber} then follows from \cite[Theorem~A.9]{austin1995morse}. 
    The number of flow lines is two provided $z_0 \in \Omega$, and either one or two provided $z_0 \in \Gamma^N$, see Figure~\ref{fig:theta12x}.

    Let us justify the assertion \ref{lem:morse-bott:gamma-N}.
	In view of Lemma~\ref{lem:isol}~\ref{lem:isol:gamma-N}, either $\theta$ is a connected component of $\Gamma^N$ or $\theta$ intersects $\Gamma^N$ at finitely many points.
    Consider the latter case and take any $z_0 \in \theta \cap \Gamma^N$.
    Since curves of critical points of $u$ are isolated in $\mathcal{C}$ (see Lemma~\ref{lem:isol}~\ref{lem:isol:simple}), $z_0$ is an isolated critical point on $\Gamma^N$, and hence there is a flow line $\gamma \subset \Gamma^N$ converging to $z_0$, see Lemma~\ref{lem:traj-neum}.
    According to the assertion \ref{lem:morse-bott:fiber}, $\gamma$ approaches $\theta$ in the normal direction, i.e., $\theta$ and $\Gamma^N$ are orthogonal at $z_0$. 
    Therefore, recalling that $\theta$ is smooth, it cannot be a closed curve and hence, by Lemma~\ref{lem:isol}~\ref{lem:isol:simple}, the other  end point of $\theta$ also belongs to $\Gamma^N$, at which we get the same orthogonality. 
\end{proof}

In the proofs of our main results, we will often  appeal to the continuous dependence of the Cauchy problem \eqref{eq:gradflow2} on the initial data.
For reader's convenience, we provide an explicit statement that we need. 
In the case $z_0 \in \Gamma^N$, the statement follows from Lemma~\ref{lem:traj-neum} and Corollary~\ref{cor:boundary}. 
\begin{lemma}\label{lem:Cuchy-contin}
Let $z_0 \in \Omega \cup \Gamma^N$ be a regular point of $u$. 
Let $t_0 \in I$ be such that $\gamma(t_0,z_0) \in \Omega \cup \Gamma^N$.
Then for any $\varepsilon>0$ there exists $\delta>0$ such that 
$\gamma(t_0,B_\delta(z_0) \cap \Omega) \subset B_\varepsilon(\gamma(t_0,z_0)) \cap \Omega$.  
\end{lemma}

\subsection{Stable and unstable sets}\label{sec:classification-of-manifolds}

For a critical point $z_0 \in \overline{\Omega}$ of an eigenfunction $u$, let $W^{\mathfrak{s}}(z_0)$ and $W^{\mathfrak{u}}(z_0)$ be the \textit{stable} and \textit{unstable sets} defined as
\begin{align*}
W^{\mathfrak{s}}(z_0) &= \{z \in \overline{\Omega}:~ \gamma(+\infty,z) = z_0\},\\
W^{\mathfrak{u}}(z_0) &= \{z \in \overline{\Omega}:~ \gamma(-\infty,z) = z_0\}.
\end{align*}
Note that we necessarily have $\gamma(t,z) \in \overline{\Omega}$ for any $t \in [0,+\infty)$ in the definition of $W^{\mathfrak{s}}(z_0)$, and $\gamma(t,z) \in \overline{\Omega}$ for any $t \in (-\infty,0]$ in the definition of $W^{\mathfrak{u}}(z_0)$, see Lemma~\ref{lem:classif-gamma}.

Let us discuss the relation between stable/unstable sets of a critical point $z_0$ of $u$ and stable/unstable sets of a critical point $(0,0)$ of a corresponding prototypical function classified in Section~\ref{sec:classification}, see Table~\ref{table:manifolds} and Figure~\ref{fig:wsu}.

\begin{table}[htbp]
\centering
\begin{tabular}{l|c|c}
&  $W^{\mathfrak{s}}(0,0)$  &  $W^{\mathfrak{u}}(0,0)$ \\
\hline
$x^2+y^2$   & $\mathbb{R}^2$  &  $(0,0)$ \\
$x^2-y^2$   & $\{(x,y):~y=0\}$  &  $\{(x,y):~x=0\}$ \\
$y^2$   &  $\{(x,y):~ x=0\}$ &   $(0,0)$ \\
$x^{2k-1} + y^2$  & $\{(x,y):~x \geq 0\}$  &   $\{(x,y):~x \leq 0, ~y=0\}$   \\
$x^{2k} + y^2$   & $\mathbb{R}^2$  &  $(0,0)$ \\
$x^{2k} - y^2$   & $\{(x,y):~y=0\}$  &  $\{(x,y):~x=0\}$ \\
$\rho^M \sin(M \vartheta)$ & $\bigcup_{j=1}^M \{(\rho,\vartheta):~ \vartheta = \frac{j \pi}{M}\}$  & $\bigcup_{j=1}^M \{(\rho,\vartheta):~ \vartheta = \frac{j \pi}{M}+\frac{\pi}{2M}\}$
\end{tabular}
\caption{Stable and unstable sets of the critical point $(0,0)$ of some prototypical functions, where $k \geq 2$, $M \geq 3$.}
\label{table:manifolds}
\end{table}

\begin{figure}[!ht]
  \begin{subfigure}{0.31\textwidth}
    \includegraphics[width=\linewidth]{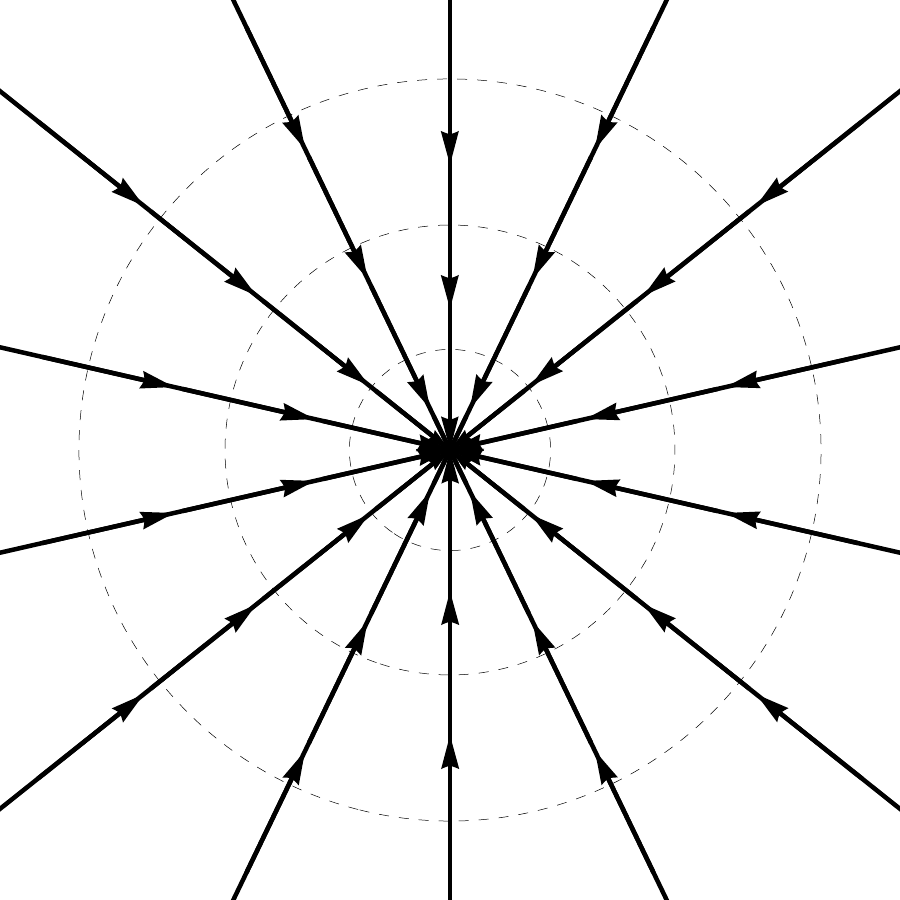}
    \caption{$x^2+y^2$} \label{fig:3a}
  \end{subfigure}%
  \hspace*{\fill} 
  \begin{subfigure}{0.31\textwidth}
    \includegraphics[width=\linewidth]{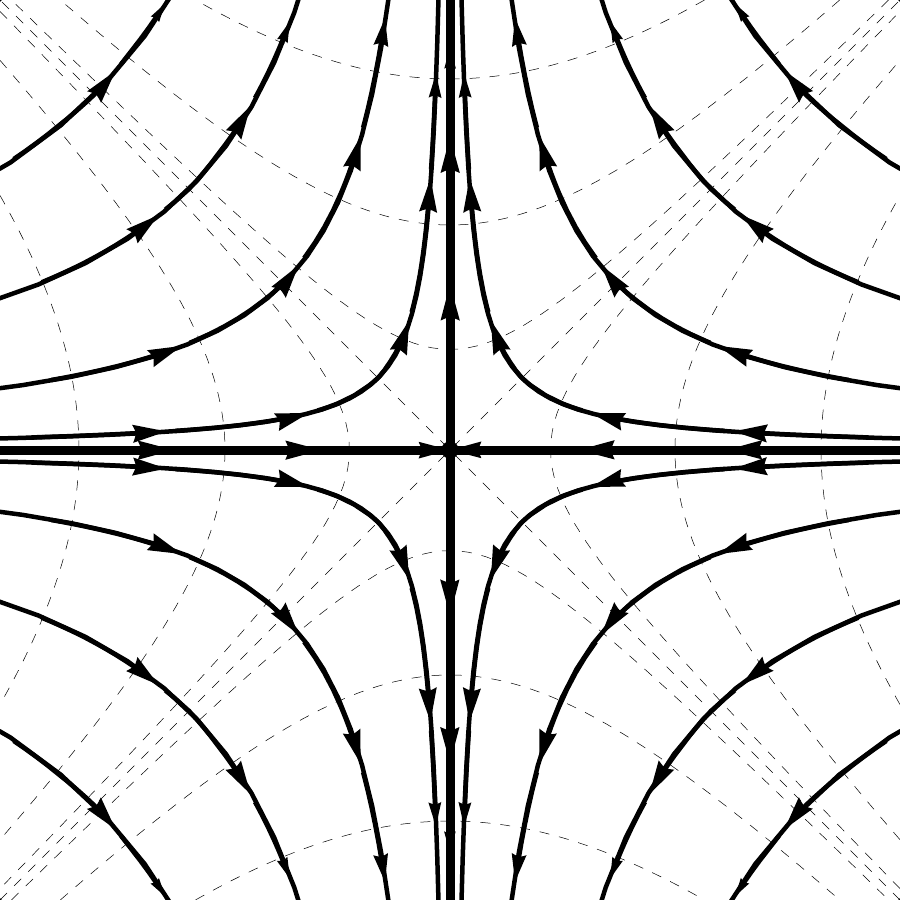}
    \caption{$x^2-y^2$} \label{fig:3b}
  \end{subfigure}%
  \hspace*{\fill} 
  \begin{subfigure}{0.31\textwidth}
    \includegraphics[width=\linewidth]{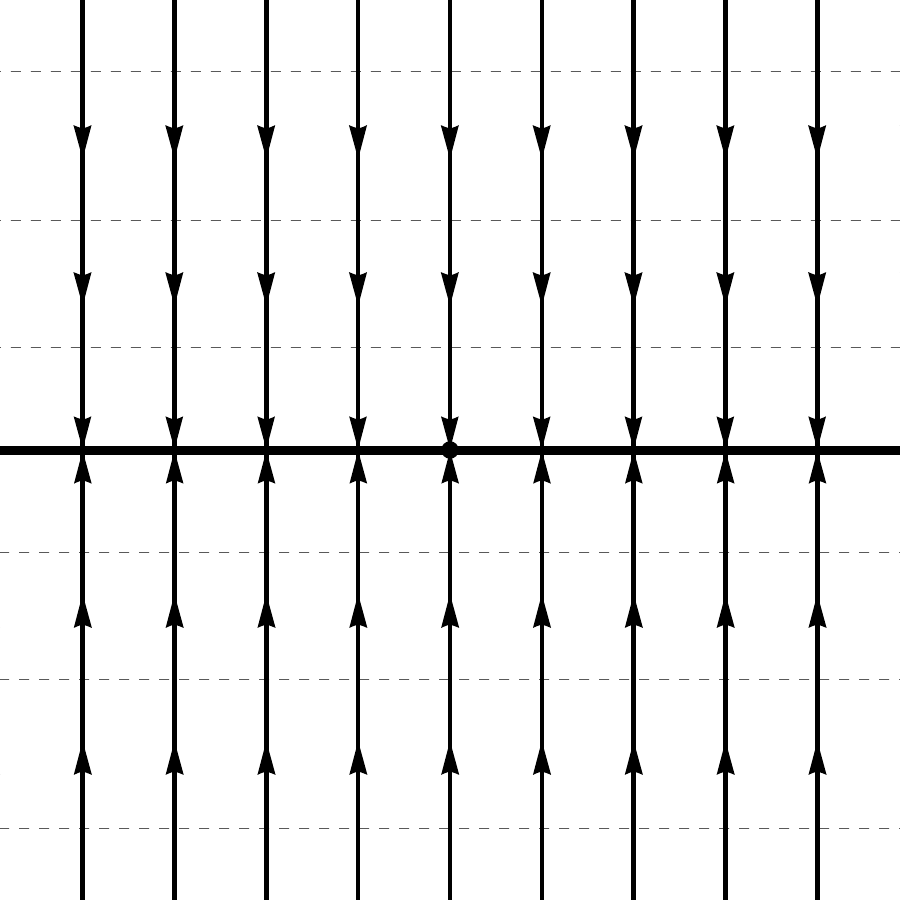}
    \caption{$y^2$} \label{fig:3c}
    \end{subfigure}
    \\[1em]
   \begin{subfigure}{0.31\textwidth}
    \includegraphics[width=\linewidth]{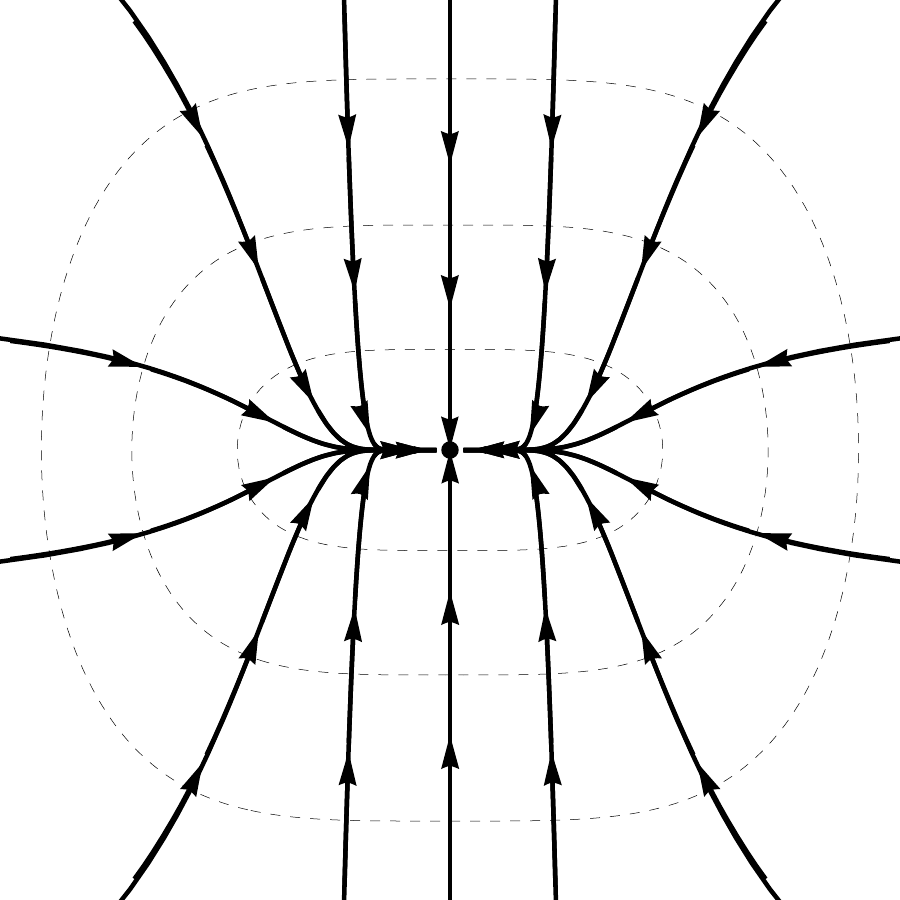}
    \caption{$x^4+y^2$} \label{fig:3d}
  \end{subfigure}%
  \hspace*{\fill}  
  \begin{subfigure}{0.31\textwidth}
    \includegraphics[width=\linewidth]{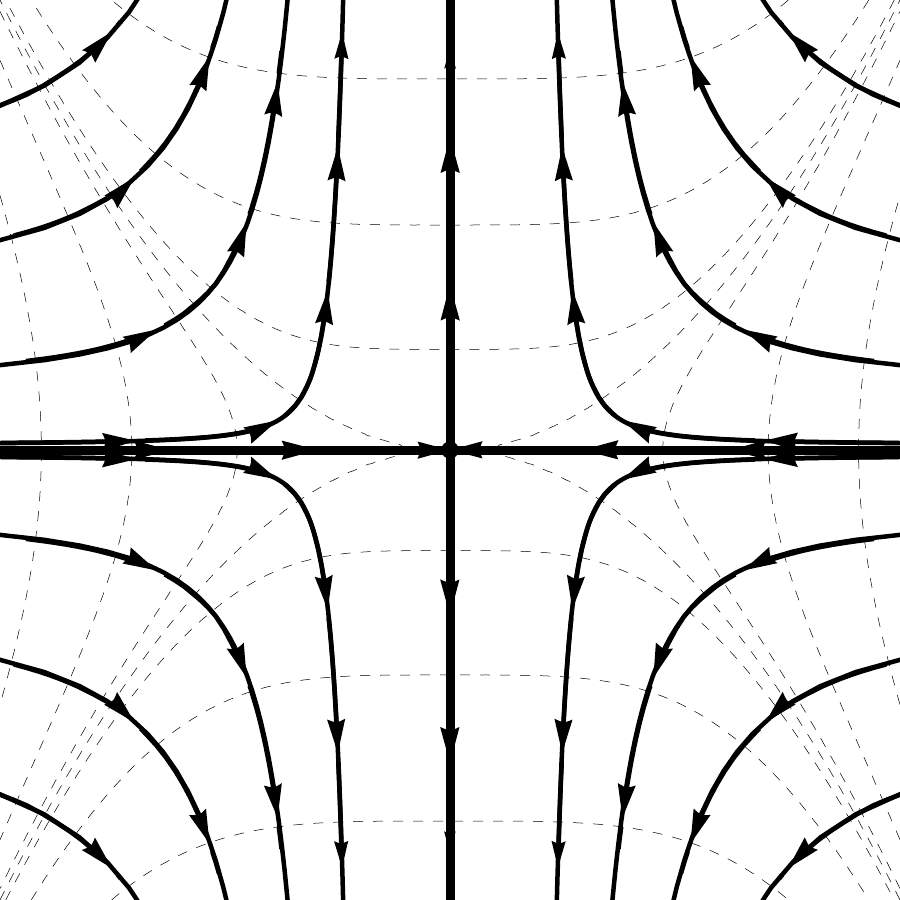}
    \caption{$x^4-y^2$} \label{fig:3e}
  \end{subfigure}%
  \hspace*{\fill} 
  \begin{subfigure}{0.31\textwidth}
    \includegraphics[width=\linewidth]{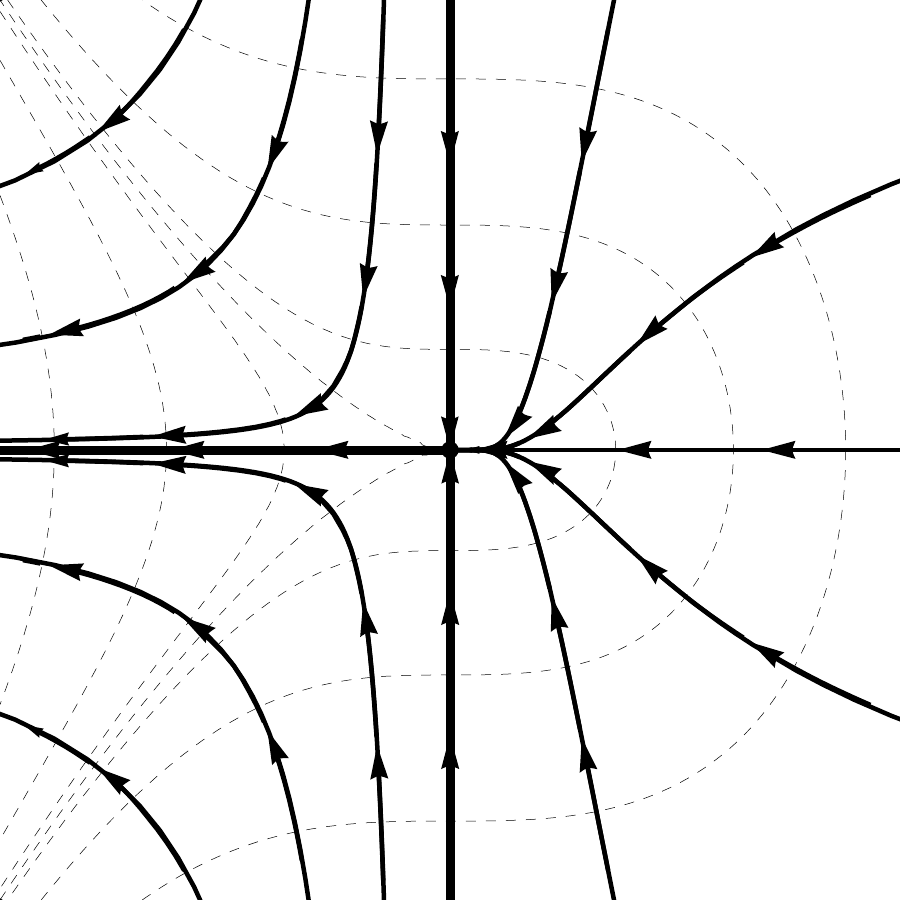}
    \caption{$x^3+y^2$} \label{fig:3f}
    \end{subfigure}
    \\[1em]
  \begin{subfigure}{0.31\textwidth}
    \includegraphics[width=\linewidth]{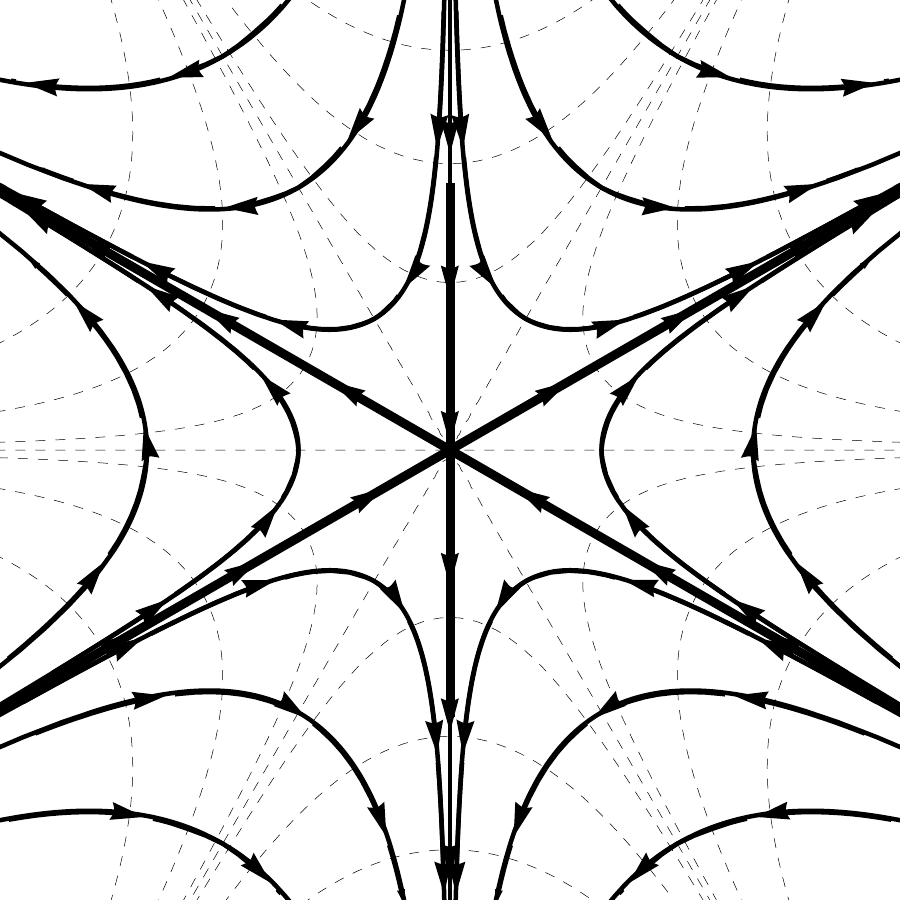}
    \caption{$\rho^3 \sin(3 \vartheta)$} \label{fig:3g}
    \end{subfigure}
    
\caption{Stable and unstable sets (solid) and  level lines (dashed) of some prototypical functions.} \label{fig:wsu}
\end{figure}

\begin{enumerate}[\rm(I)]
\item Let $z_0$ be non-degenerate. Then, locally, $W^{\mathfrak{s}}(z_0)$ and $W^{\mathfrak{u}}(z_0)$ are homeomorphic to those of one of the functions
\begin{equation}\label{eq:prot-mors}
(x,y) \mapsto u(z_0) \pm x^2 \pm y^2.
\end{equation} 
This is the content of the Hartman--Grobman theorem, see, e.g., \cite[Section~2.8]{perko2013differential}.
For a related result of global nature, we refer to the stable/unstable manifold theorem for Morse functions, see, e.g., \cite[Theorem~4.2]{banyaga2004lectures} (the case of compact manifolds) and \cite[Lemma~3.6]{band2016topological} (the case of manifolds with boundary).

\item Let $z_0$ be semi-degenerate. Then, locally, $W^{\mathfrak{s}}(z_0)$ and $W^{\mathfrak{u}}(z_0)$ are homeomorphic to those of one of the functions
\begin{align}
\label{eq:prot-semidegen1}
&(x,y) \mapsto u(z_0) \pm y^2,\\
\label{eq:prot-semidegen2}
&(x,y) \mapsto u(z_0) \pm x^k \pm y^2
\quad \text{for a natural number}~ k \geq 3.	
\end{align}
 In the case of \eqref{eq:prot-semidegen1}, one can see (using Lemma~\ref{lem:isol}) that $z_0$ is a Morse--Bott critical point, i.e., $u$ is degenerate along a curve $\theta$ of critical points containing $z_0$ and non-degenerate in a normal direction to $\theta$ at $z_0$, and the desired equivalence follows from, e.g., \cite[Theorem~A.9]{austin1995morse}, see also Lemma~\ref{lem:morse-bott}.
In the case of \eqref{eq:prot-semidegen2}, the result follows from \cite[Theorem~65, p.~340]{andronov1974qualitative} (see also \cite[Section~2.11]{perko2013differential}). 
Let us observe that in the case of \eqref{eq:prot-semidegen2} with odd $k$, i.e., when $z_0$ is a saddle-node, either
$W^{\mathfrak{s}}(z_0)$ or $W^{\mathfrak{u}}(z_0)$ has nonempty interior. 

\item Let $z_0$ be fully degenerate. 
In this case, we are not aware of the relation between $W^{\mathfrak{s}}(z_0)$ and $W^{\mathfrak{u}}(z_0)$ and those of the prototypical function 
$$
(\rho,\vartheta) \mapsto \rho^M \sin(M \vartheta), \quad M \geq 3,
$$
although we anticipate that they are homeomorphic.
In Lemma~\ref{lem:observation}, we provide a partial result in this regard which will be sufficient for our purposes.
Observe that, in the case of general analytic functions, if $z_0$ is a fully degenerate critical point, then its stable/unstable sets might be sensitive to diffeomorphisms of coordinates. 
In particular, their dimensions might change after the application of a diffeomorphism (see, e.g., \cite[Example~9]{szafraniec2021stable} for an explicit example and further results in this direction), although the level sets in a neighborhood of $z_0$ remain homeomorphic. 
\end{enumerate}

\section{Main results}\label{sec:mainresults}

\subsection{Definitions of \texorpdfstring{$G$}{G}, \texorpdfstring{$W$}{W}, \texorpdfstring{$L$}{L}}
Our definitions of Neumann domains and lines proposed in Section~\ref{sec:definitions-neumanndomains} below will rely on the consideration of three sets $G$, $W$,  $L$, defined by \eqref{def:G}, \eqref{def:W},  \eqref{eq:L}, respectively.
These sets collect flow lines in $\Omega$ of the following three types: 
\begin{enumerate}
    \item flow lines that reach $\Gamma^D$ in a finite time,

    \item flow lines which connect two local extremum points,

    \item flow lines having a saddle point on at least one of their ends.
\end{enumerate} 
Thus, according to Lemma~\ref{lem:classif-gamma}, we exhaust all possible behaviors of the flow lines.

First, we consider the set of all points in $\Omega$ whose flow lines reach the boundary $\partial\Omega$ in a finite (positive or negative) time:
\begin{equation}\label{def:G}
G=\{z\in \Om:~ \gamma(t_z, z) \in \Gamma^D\; \text{for some}\;
t_z \in \mathbb{R}\}.
\end{equation}
Recall that $\Gamma^N$ cannot be reached by flow lines starting in $\Omega$ in a finite time, see Corollary~\ref{cor:boundary}. 
This fact and the strict monotonicity of $u$ along the flow lines imply that $\gamma(t,z) \in {\Omega}$ for any $z \in G$ and $|t|<|t_z|$.
Observe that $G$ is empty provided $\Gamma^D = \emptyset$. 
In Section~\ref{sec:GWL}, we collect few main properties of $G$. 
In particular, $G$ is open, $\Gamma^D \subset \partial G$, any flow line starting in $\partial G \setminus \Gamma^D$ stays in $\partial G$ for all times, and $\partial G$ contains of at most finitely many flow lines, arcs of critical points, and isolated critical points.

\begin{remark}\label{rem:effectless-cut}
In \cite{weineffect},  \textsc{Weinberger} considered a multiply connected domain $\Omega$ such that the outer boundary $\Ga_0$ and at least one more connected boundary component $\Ga_1$ are contained in $\Gamma^D$.  
Then, for the first eigenfunction $u$ of \eqref{cutproblem2d}, he defined the set 
$$
G_1 = \{z\in \Om:~ \gamma(t_z, z) \in \Ga_1\; \text{for some}\;
t_z \in \mathbb{R}\}.
$$
Notice that, $G_1\subset G.$ It was proved that $G_1$ is open, $\widetilde{\gamma}=\partial G_1 \cap \Omega$ separates $\Ga_1$ from $\Gamma^D \setminus \Ga_1$ and consists of a finite number of analytic arcs of finite lengths, along each of which $\partial u/\partial \nu = 0$.
The set $\widetilde{\gamma}$ was called the \textit{effectless cut}, and its existence was successively used by \textsc{Hersch} \cite{hers} to derive an upper bound for the first eigenvalue, see Section~\ref{sec:intro}.
\end{remark}

\begin{remark}
The set $G$ is larger than $G_1$, and it is introduced to capture global properties of an arbitrary eigenfunction near $\Gamma^D$.
However, the consideration of simple model domains, as well as the analysis of definitions and results from \cite{mcdonald2014neumann,band2016topological}, 
indicate that $\partial G$ alone might not describe all the dividing flow lines (separatrices) of $u$ even in a neighborhood of $\Gamma^D$. 
We provide an explicit example in Remark~\ref{rem:GWL}~\ref{rem:GWL:G}.
\end{remark}

Next, we consider the set of all points in $\Omega$ whose flow lines connect two local extremum points of $u$:
\begin{equation}\label{def:W}
W 
= 
\Omega \cap \left\{\bigcup_{q \in \mathcal{M}_-~}
\bigcup_{p \in \mathcal{M}_+}
\left(
W^{\mathfrak{s}}({q})
\cap 
W^{\mathfrak{u}}({p})
\right)\right\},
\end{equation}
where $\mathcal{M}_-$ and $\mathcal{M}_+$ are the sets of local minimum and local maximum points of $u$ in $\overline{\Omega}$, respectively (see Section~\ref{sec:properties}).
Notice that $W$ is empty provided $\mathcal{M}_- = \emptyset$ or $\mathcal{M}_+ = \emptyset$.
Indeed, this happens, e.g., when $u$ is the first Dirichlet eigenfunction in a concentric ring or in any strictly convex domain \cite{payne1973two}. 
At the same time, in these examples, $G$ is \textit{not} empty.
In Section~\ref{sec:GWL}, we obtain a few main properties of $W$: $W$ is open, $G \cap W = \emptyset$, any flow line starting in $\partial W$ stays in $\partial W$ for all times, and $\partial W$ consists of at most finitely many flow lines, arcs of critical points, and isolated critical points.

\begin{remark}
Let us provide some remarks on the definition \eqref{def:W} of $W$:
\begin{enumerate}[label={\rm(\roman*)}]
\item The set $W$ can be equivalently defined as
\begin{equation}\label{eq:W1}
	W 
	= 
	\bigcup_{q \in \mathcal{M}_-~}
	\bigcup_{p \in \mathcal{M}_+}
	\{z \in \Omega:~ 
	\gamma(+\infty,z) = q,~
	\gamma(-\infty,z) = p\}.
\end{equation}
Moreover, $I=\mathbb{R}$ for any $z \in W$, see Lemma~\ref{lem:classif-gamma}.

\item The main purpose of the presence of intersection with $\Omega$ in \eqref{def:W} is to make $W$ open (we prove this fact in Lemma~\ref{lem:W:open}).
Otherwise, one might have $\Gamma^N \subset W$.
This is made only for convenience and, at large, is not an essential restriction. 

\item 
Recall that $\mathcal{M}_+ \cap \mathcal{M}_- = \emptyset$, see Remark~\ref{rem:boundary_critical}~\ref{rem:boundary_critical:Mpmempty}.
Since $u$ is strictly monotone along flow lines, there are no flow lines connecting two local maximum points or two local minimum points of $u$, and hence such cases are not included in the definition of $W$.

\item 
Local extremum points never occur on the Dirichlet boundary $\Gamma^D$ (see Lemma \ref{lem:crit}~\ref{lem:crit:singsadleisolated}), while they can occur on $\Gamma^N$.
The consideration of critical points on $\Gamma^N$ in the definition of $W$ can be motivated by the following example. 
Consider a \textit{third radial} Neumann eigenfunction $u$ in a concentric ring. 
Without loss of generality, we may assume that $u$ has a circle $C \subset \Omega$ of local maximum points and two circles (constituting $\Gamma^N$) of local minimum points.
If one does not include critical points on the boundary to the definition of $W$, then $W$ is empty, 
but any natural definition of the Neumann line set should contain the circle $C$.
Notice that $G$ is empty in this example since $\Gamma^D=\emptyset$. 
\end{enumerate}
\end{remark}

Finally, we consider the set of all points in $\overline{\Omega}$ whose flow lines have a saddle point on at least one of their ends.
More precisely, for any $r \in \mathcal{S}$, we denote
\begin{equation}\label{def:L}
L(r) 
= 
W^{\mathfrak{s}}({r})
\cup 
W^{\mathfrak{u}}({r}),
\end{equation} 
and let
\begin{equation}\label{eq:L}
L
=
\bigcup_{r \in \mathcal{S}}
\overline{L(r)},
\end{equation}
where $\mathcal{S}$ is the set of saddle points of $u$ in $\overline{\Omega}$ (see Section \ref{sec:properties}).
Observe that $L$ is empty provided $\mathcal{S} = \emptyset$. 
For instance, this happens for the first Dirichlet eigenfunction in any strictly convex domain \cite{payne1973two} and the second Neumann eigenfunction in a disk.
However, in the former case we have $G \neq \emptyset$, $W=\emptyset$, and in the latter case we have $G=\emptyset$, $W \neq \emptyset$.
In Section~\ref{sec:GWL}, we describe the structure of the sets $L(r)$ and $L$ in terms of the behavior of flow lines on their boundaries.

\begin{remark}\label{rem:L}
We list some basic observations related to the sets $L(r)$ and $L$:
\begin{enumerate}[label={\rm(\roman*)}]
\item The set $L$ can be equivalently defined as
\begin{equation}\label{eq:L1}
	L
	=
	\bigcup_{r \in \mathcal{S}}
	\overline{
		\{z \in \overline{\Omega}:~ \text{either}~ \gamma(-\infty,z) = r ~\text{or}~ \gamma(+\infty,z)=r\}
	}.
\end{equation}

\item 
Since the number of saddle points of $u$ is finite (see Lemma~\ref{lem:numsaddlesOmega}), we get
$$
L 
=
\overline{
	\bigcup_{r \in \mathcal{S}}
	L(r)
} 
=
\overline{\bigcup_{r \in \mathcal{S}} (W^{\mathfrak{s}}({r})
	\cup 
	W^{\mathfrak{u}}({r}))}.
$$
If $u$ is a Morse function and $\Gamma^N = \emptyset$, then this definition coincides with \cite[Definition~1.2]{band2016topological} of the Neumann line set; cf.~Definition~\ref{definition:NS} below.

\item\label{rem:L:boundaryL-Lr} 
Observe that 
$$
\partial L \subset \bigcup_{r \in \mathcal{S}} \partial L(r).
$$

\item By \eqref{def:L}, a saddle point $r$ is the only critical point contained in $L(r)$. 
However, the set $L$, being a closed set, may contain local extremum points from $\overline{\Om}$.

\item 
The set $L(r)$ has a nonempty interior provided $r \in \mathcal{S}$ is a saddle-node, see Remark~\ref{rem:terminology}. 
The possibility of appearance of such a critical point in the critical set of $u$ is shown in \cite{chenmyrtaj2019}, see also Section~\ref{sec:trivialcrit}, and we refer to \cite[Section~5]{arango2010critical} and \cite[Section~2]{massimo2021number} for related discussion. 
That is, such type of saddle points has to be taken into account and, in fact, it causes most of further complications. 
\end{enumerate}
\end{remark}

As we mentioned at the beginning of this section, Lemma~\ref{lem:classif-gamma} implies that the flow line of \textit{any} regular point $z \in \Omega$ belongs to at least one of the sets $G$, $W$, $L$. 
In other words, $G$, $W$, $L$ together describe all the (nontrivial) flow lines in $\Omega$.
In Section~\ref{sec:definitions-neumanndomains}, we show that the remaining part of $\overline{\Omega}$ is described in terms of the boundaries of $G$, $W$, $L(r)$ and it consists of dividing flow lines (separatrices), arcs of critical points, and isolated critical points.

\subsection{Neumann domains and lines}\label{sec:definitions-neumanndomains}

With the sets $G$, $W$, $L$ at hand, we are ready to propose definitions of Neumann domains and lines of an arbitrary nonconstant eigenfunction $u$ of \eqref{cutproblem2d}, cf.\ definitions from \cite{band2016topological,mcdonald2014neumann} for Morse eigenfunctions.

\begin{definition}\label{definition:NS} 
The \textit{Neumann line set} of $u$ is
$$
\mathcal{N}(u) 
=
\overline{(\partial G \cup \partial W \cup \bigcup_{r \in \mathcal{S}} \partial L(r)) \setminus \Gamma^D}.
$$
\end{definition}

\begin{definition}\label{definition:ND}
A \textit{Neumann domain} of $u$ is any connected component of the set 
$$
\left(
G \cup W \cup \mathrm{Int}(L)
\right)
\setminus 
\bigcup_{r \in \mathcal{S}} \partial L(r).
$$
A \textit{boundary Neumann domain} is any connected component of the set $G \setminus \bigcup_{r \in \mathcal{S}} \partial L(r)$. 
Any other Neumann domain is called \textit{inner Neumann domain}. 
\end{definition}

\begin{remark}\label{rem:GWL}
Let us provide a few comments on the definitions stated above:

\begin{enumerate}[label={\rm(\roman*)}]
\item In Theorem~\ref{thm:main-properties}~\ref{thm:main-properties:decomposition}, we establish that Neumann domains are connected components of $\Omega \setminus \mathcal{N}(u)$.

\item In view of Remark~\ref{rem:L}~\ref{rem:L:boundaryL-Lr}, we have
\begin{equation}\label{eq:LLL1}
	L \setminus 
	\bigcup_{r \in \mathcal{S}} \partial L(r)
	=
	(\mathrm{Int}(L) \cup \partial L)
	\setminus 
	\bigcup_{r \in \mathcal{S}} \partial L(r)
	=
	\mathrm{Int}(L)
	\setminus 
	\bigcup_{r \in \mathcal{S}} \partial L(r).
\end{equation}
Thus, Neumann domains can be equivalently defined as connected components of the set
$$
\left(
G \cup W \cup L
\right)
\setminus 
\bigcup_{r \in \mathcal{S}} \partial L(r).
$$

\item  By \eqref{eq:LLL1} and Remark~\ref{rem:L}~\ref{rem:L:boundaryL-Lr}, one gets
\begin{align*}
	\mathrm{Int}(L)
	\setminus 
	\bigcup_{r \in \mathcal{S}} \partial L(r)
	&=
	L \setminus 
	\bigcup_{r \in \mathcal{S}} \partial L(r)
	=
	\bigcup_{r \in \mathcal{S}} \overline{L(r)}  
	\setminus 
	\bigcup_{r \in \mathcal{S}} \partial L(r)
	\\
	&=
	\bigcup_{r \in \mathcal{S}} (\mathrm{Int}(L(r)) \cup \partial L(r))
	\setminus 
	\bigcup_{r \in \mathcal{S}} \partial L(r)
	\\
	&=
	\bigcup_{r \in \mathcal{S}}\mathrm{Int}(L(r))
	\setminus 
	\bigcup_{r \in \mathcal{S}} \partial L(r).
\end{align*}
Therefore, Neumann domains can be equivalently characterized as connected components of the set
\begin{equation}
\label{eq:NDequiv1}
	(
	G \cup W \cup \bigcup_{r \in \mathcal{S}}\mathrm{Int}(L(r))
	)
	\setminus 
	\bigcup_{r \in \mathcal{S}} \partial L(r).
\end{equation}

\item It will be shown in Lemma~\ref{lem:partialGW-finite-flow-lines} that any flow line contained in $(\partial G \cup \partial W) \setminus \Gamma^N$ converges to a saddle point. 
Using this fact and Theorem~\ref{thm:main-properties}~\ref{thm:main-properties:critpoints}, it can be deduced that
$$
\mathcal{N}(u)
=
\overline{(\bigcup_{r \in \mathcal{S}} \partial L(r)) \setminus \Gamma^D} \cup \mathcal{C} \cup \Gamma^N,
$$
where $\mathcal{C}$ is the critical set of $u$ (see Section~\ref{sec:properties}).

\item The set closure and the subtraction of $\Gamma^D$ are presented in Definition~\ref{definition:NS} to avoid those regular points of $u$ on $\Gamma^D$ whose flow lines are not separatrices.

\item\label{rem:GWL:G} The set $G$ alone does not necessarily describe all boundary Neumann domains even when $u$ is the first eigenfunction. 
For example, this happens on a symmetric dumbbell domain $\Omega$ with sufficiently thin handle and $\partial\Omega = \Gamma^D$ (see Figures~\ref{fig:saddle-node4} and~\ref{fig:saddle-node6-5c} in Section~\ref{sec:trivialcrit}).
In this case, the first eigenfunction $u$ shares the symmetries of $\Omega$, and it is natural to expect that $u$ is a Morse function with two points of global maximum and one saddle point $r$, see Figure~\ref{fig:saddle-node6-5c}.
The set $W^{\mathfrak{u}}(r)$ reaches $\partial\Omega$ at two points in a finite time, which results in the presence of \textit{two} boundary Neumann domains according to the definition from \cite[Section~3]{band2016topological} (and to Definition~\ref{definition:ND}). 
On the other hand, we have $G = \Omega \setminus W^{\mathfrak{s}}(r)$, and hence $G$ is a connected set. 
Nevertheless, the sets $G$ and $\partial G$ might be useful by themselves, see the discussion of the works \cite{hers,weineffect} in Section~\ref{sec:intro}.

\item When $\partial\Omega = \Gamma^D$ and $u$ is a Morse function, it can be shown that Definitions~\ref{definition:NS} and~\ref{definition:ND} describe the same Neumann line set and Neumann domains as definitions from \cite[Section~3]{band2016topological}, except that Neumann domains in our definition do not include $\Gamma^D$. 
\end{enumerate}
\end{remark}

Let us now provide several fundamental properties of Neumann domains and Neumann lines. 
Similar results can be found in \cite[Theorem~3.13]{band2016topological} in the case of Dirichlet boundary conditions under additional assumptions that $\Omega$ is simply connected and the eigenfunction $u$ is a Morse function having at least one saddle point (i.e., $\mathcal{S} \neq \emptyset$). 
In our settings, we do not require the nondegeneracy of critical points of $u$ and $\mathcal{S} \neq \emptyset$, and we deal with more general domains and boundary conditions.
\begin{theorem}\label{thm:main-properties}
Let $u$ be a nonconstant eigenfunction of \eqref{cutproblem2d}.
Then the following assertions on its Neumann domains hold:
\begin{enumerate}[label={\rm(\roman*)}]
\item\label{thm:main-properties:opennes} Each Neumann domain is an open set.
\hyperlink{thm:main-properties:opennes:proof}{{\tiny \hfill $\to\!\!\!\!\!\!  \qed$}}
\item\label{thm:main-properties:finite-number} The number of Neumann domains is finite.
\hyperlink{thm:main-properties:finite-number:proof}{{\tiny \hfill $\to\!\!\!\!\!\!  \qed$}}
\item\label{thm:main-properties:sign-changing} $u$ is sign-changing in any inner Neumann domain.
\hyperlink{thm:main-properties:sign-changing:proof}{{\tiny \hfill $\to\!\!\!\!\!\!  \qed$}}
\item\label{thm:main-properties:sign-const} $u$ is strictly positive or strictly negative in any boundary Neumann domain. 
\hyperlink{thm:main-properties:sign-const:proof}{{\tiny \hfill $\to\!\!\!\!\!\!  \qed$}}
\end{enumerate}
Moreover, the following assertions on the Neumann line set $\mathcal{N}(u)$ hold:
\begin{enumerate}[label={\rm(\arabic*)}]
\item\label{thm:main-properties:critpoints} $\Gamma^N \subset \mathcal{N}(u)$ and $\mathcal{C}\subset \mathcal{N}(u)$.
\hyperlink{thm:main-properties:critpoints:proof}{{\tiny \hfill $\to\!\!\!\!\!\!  \qed$}}
\item\label{thm:main-properties:decomposition} If $\mathcal{D}(u)$ is a union of all Neumann domains, then
$$
\Omega = \mathcal{D}(u) \cup (\mathcal{N}(u) \setminus \partial \Omega).
$$
Consequently, any Neumann domain is a connected component of $\Omega \setminus \mathcal{N}(u)$.
\hyperlink{thm:main-properties:decomposition:proof}{{\tiny \hfill $\to\!\!\!\!\!\!  \qed$}}
\item\label{thm:main-properties:finite-length} $\mathcal{N}(u)$ consists of at most finitely many flow lines, isolated critical points, and arcs of critical points of $u$. 
Any isolated critical point in $\mathcal{N}(u)$ is either an end point of a flow line from $\mathcal{N}(u)$ or an isolated point of $\mathcal{N}(u)$. 
\hyperlink{thm:main-properties:finite-length:proof}{{\tiny \hfill $\to\!\!\!\!\!\!  \qed$}}   
\item\label{thm:main-properties:finite-length2}
The length of $\mathcal{N}(u)$ is finite 
and $\partial u/\partial \nu = 0$ along any flow line and arc of critical points in $\mathcal{N}(u)$. 
\hyperlink{thm:main-properties:finite-length2:proof}{{\tiny \hfill $\to\!\!\!\!\!\!  \qed$}}
\end{enumerate}
\end{theorem}

This theorem is proved in Section~\ref{sec:proofs-main}, after establishing main properties of the sets $G$, $W$, $L$ in Section~\ref{sec:GWL}.

\section{Properties of \texorpdfstring{$G$}{G}, \texorpdfstring{$W$}{W}, \texorpdfstring{$L$}{L}}\label{sec:GWL}
In this section, we provide several most important properties of the sets $G$, $W$, $L$, which will be needed to prove Theorem~\ref{thm:main-properties}.
We start with the openness of $G$ and $W$. 

\begin{lemma}\label{lem:G:open}
$G$ is open and $\Gamma^D \subset \partial G$.
\end{lemma}
\begin{proof}
If $G=\emptyset$, then there is nothing to prove. 
Let $G \neq \emptyset$ and take any $z_0\in G$. 
By the definition~\eqref{def:G} of $G$, there exists  $t_0\in \R$ such that $q:=\ga(t_0,z_0)\in \Gamma^D$, and hence $q$ is a regular point of $u$.
Thanks to the implicit function theorem, we see that $\Gamma^D$ is a smooth curve in a neighborhood of $q$.
Assuming, without loss of generality, that $t_0 > 0$, we have
\begin{equation}\label{eq:unalbau=0}
u(q) = 0
\quad
\text{and}
\quad
-|\nabla u(q)| = \frac{\partial u}{\partial \nu}(q) < 0,
\end{equation}
where $\nu$ is pointed outwards. 
Consider the analytic extension of $u$ to a neighborhood of $q$. 
Thanks to \eqref{eq:unalbau=0} and the strict monotonicity of $u$ along flow lines, there exists $t_1>t_0$ such that $q_1 := \gamma(t_1,z_0) \in \mathbb{R}^2 \setminus \overline{\Omega}$.
Let us choose $\varepsilon>0$ such that $B_\varepsilon(q_1) \subset \mathbb{R}^2 \setminus \overline{\Omega}$.
By the continuous dependence of the Cauchy problem \eqref{eq:gradflow2} on the initial data (see, e.g., Lemma~\ref{lem:Cuchy-contin}), there exists $\delta>0$ such that $\gamma(t_1,z)\in B_\varepsilon(q_1)$ for all $z\in B_\delta(z_0)$.
Since $B_\varepsilon(q_1) \subset \mathbb{R}^2 \setminus \overline{\Omega}$ and thanks to Remark~\ref{rem:traj-boundary}, for any $z\in B_\delta(z_0)$ there exists a unique $t_z \in (0,t_1)$ such that $\gamma(t_z,z) \in \Gamma^D$.
This implies that $B_\delta(z_0) \subset G$, and hence $G$ is open.

If $p \in \Gamma^D$ is a regular point of $u$, then $p \in \partial G$ 
since $\gamma(t,p) \in \Omega$ for all sufficiently small $t<0$ (respectively, $t>0$) whenever $\frac{\partial u}{\partial \nu}(p) < 0$ (respectively, $\frac{\partial u}{\partial \nu}(p) > 0$).
By Lemma~\ref{lem:crit}, $\Gamma^D$ contains finitely many critical points, and hence each critical point $p \in \Gamma^D$ can be approximated by regular points from $\Gamma^D$, which yields $p \in \partial G$.
Thus, the inclusion $\Gamma^D \subset \partial G$ follows.
\end{proof}

\begin{lemma}\label{lem:W:open}
$W$ is open.
\end{lemma}
\begin{proof}
If $W  = \emptyset$, then there is nothing to prove.
Assume that  $W  \neq \emptyset$ and take any $z_0 \in W$.
By the definition \eqref{def:W} of $W$, 
there exist a local minimum point $q \in \overline{\Omega}$ and a local maximum point $p \in \overline{\Omega}$ such that 
$z_0 \in (W^{\mathfrak{s}}({q}) \cap W^{\mathfrak{u}}({p})) \cap \Omega$.
Note that if $q$ or $p$ is a boundary point, then it belongs to $\Gamma^N$, as it follows from Lemma~\ref{lem:crit}~\ref{lem:crit:singsadleisolated}.
Our aim is to show the existence of a neighborhood of $z_0$ which is contained in $W$.
We divide the proof in two cases depending on the isolation of $q$ and $p$.

(1) Assume that both $q$ and $p$ are \textit{isolated}. 
Let $\varepsilon>0$ be sufficiently small so that 
$q$ is the only critical point in $\overline{B_\varepsilon(q)} \cap \overline{\Omega}$.
By Proposition~\ref{prop:absil}, there exists $\delta>0$ such that $\gamma(+\infty, \bar{z}) \in \overline{B_\varepsilon(q)} \cap \overline{\Omega}$ for any $\bar{z} \in B_\delta(q) \cap \overline{\Omega}$, and hence we have $\gamma(+\infty, \bar{z}) = q$.
On the other hand, $\ga(t,z_0) \in \Omega$ for any $t > 0$ thanks to Corollary~\ref{cor:boundary}, and hence there exists $t_0>0$ such that $\ga(t_0,z_0) \in B_{\delta}(q) \cap {\Omega}$. 
Using Lemma~\ref{lem:Cuchy-contin}, we can find $\sigma_1>0$ such that $\ga(t_0, z)\in B_\delta(q) \cap \Omega$ for any $z\in B_{\sigma_1}(z_0) \subset \Omega$. 
Therefore, we conclude that 
$\gamma(+\infty,z)=q$ for all $z\in B_{\sigma_1}(z_0)$.

Repeating the same arguments as above with respect to the negative times, we can find $\sigma_2>0$ such that $\gamma(-\infty,z)=p$ 
for all $z\in B_{\sigma_2}(z_0)$.
This yields 
$$
B_{\min\{\sigma_1,\sigma_2\}}(z_0)\subset (W^{\mathfrak{s}}({q})\cap W^{\mathfrak{u}}({p})) \cap \Omega \subset W.
$$

(2) Assume, without loss of generality, that $q$ is \textit{nonisolated}.
Thanks to Lemma~\ref{lem:isol}, $q$ belongs to an analytic curve $\theta$ of critical points of $u$, and we can choose $\varepsilon>0$ such that all critical points of $u$ in $\overline{B_{\varepsilon}(q)} \cap \overline{\Omega}$ belong to $\theta$. 
According to Corollary~\ref{cor:no-saddles-on-theta}, $\theta$  consists of local minimum points.
As above, by combining Proposition~\ref{prop:absil} and Lemma~\ref{lem:Cuchy-contin}, we get $\sigma_1>0$ such that $\gamma(+\infty, z) \in \overline{B_\varepsilon(q)} \cap \overline{\Omega}$ for any $z \in B_{\sigma_1}(z_0) \subset \Omega$. 
Therefore, noting that $\gamma(+\infty, z)$ is a critical point of $u$, we conclude that $\gamma(+\infty,z) \in \theta$ and it is a local minimum point of $u$ for any $z \in B_{\sigma_1}(z_0)$.

Let us now consider the behavior of flow lines of points $z \in B_{\sigma_2}(z_0)$ as $t \to -\infty$ for a sufficiently small $\sigma_2>0$.
If $p$ is isolated, we argue as in the case (1) to conclude that $\gamma(-\infty,z) = p$.
If $p$ is not isolated, we argue as above and conclude that $\gamma(-\infty,z)$ is a local maximum point of $u$ in a neighborhood of $p$, and it belongs to a curve of critical points.

Taking now any $z \in B_{\min\{\sigma_1,\sigma_2\}}(z_0)$, we denote $q_z = \gamma(+\infty,z)$ and $p_z = \gamma(-\infty,z)$, where
$p_z=p$ if $p$ is isolated.
We see that $z \in (W^{\mathfrak{s}}(q_z) \cap W^{\mathfrak{u}}(p_z)) \cap \Omega$ and hence $z \in W$.
Consequently, $B_{\min\{\sigma_1,\sigma_2\}}(z_0) \subset W$.
\end{proof}

Since $G, W$ are open and $L$ is closed, the following simple observation follows from the definitions of $G$, $W$, $L$.
\begin{lemma}\label{lem:GWL-intersection}
$G \cap \overline{W} = \emptyset$, $\overline{G} \cap W = \emptyset$, and $W \cap L = \emptyset$.
\end{lemma}

\subsection{Characterization of \texorpdfstring{$\partial G$}{dG}, \texorpdfstring{$\partial W$}{dW}, \texorpdfstring{$\partial L(r)$}{dL(r)} via flow lines}

In the following three lemmas we show that, outside of $\Gamma^D$, regular parts of the boundaries of the sets $G$, $W$, $L(r)$ can be characterized by flow lines.
Recall that $I$ stands for the maximal interval of existence of the flow line $\ga(\cdot, z_0)$ in $\overline{\Omega}$ for $z_0 \in \overline{\Omega}$, see Section~\ref{sec:gradflow}.

\begin{lemma}\label{lem:G:boundary}
Let $z_0 \in \partial G \setminus \Gamma^D$.
Then $\gamma(t,z_0) \in \partial G \setminus \Gamma^D$ for any $t \in \mathbb{R}$.
\end{lemma}

\begin{proof}
Since $G$ is open by Lemma~\ref{lem:G:open}, we have
$z_0 \not\in G$. 
Thus, $\gamma(\cdot,z_0)$ cannot reach $\Gamma^D$ in a finite time, and hence $\gamma(t,z_0) \in \Omega \cup \Gamma^N$ for any $t \in I = \mathbb{R}$. If possible, let $t_0$ be such that $z = \gamma(t_0,z_0) \not\in \pa G$.
Consequently, we have either $z \in G$ or $z \in (\Omega \cup \Gamma^N) \setminus \overline{G}$.
The case $z \in G$ is impossible since the definition \eqref{def:G} of $G$ would imply $\gamma(\cdot,z_0) \subset G$ and hence $z_0 \in G$.
Let us show that the case $z \in (\Omega \cup \Gamma^N) \setminus \overline{G}$ is also impossible. 
Since $\mathrm{dist}(z,\overline{G})>0$, there exists a neighborhood $V$ of $z$ such that $V \cap \Omega \subset \Omega \setminus \overline{G}$. 
By Lemma~\ref{lem:Cuchy-contin}, for a sufficiently small neighborhood $U$ of $z_0$ we have $\gamma(t_0,U \cap \Omega) \subset \Omega \setminus \overline{G}$.
However, since $z_0 \in \partial G$, $U$ contains a point $z_1 \in G \subset \Omega$ and hence $\gamma(\cdot,z_1) \subset G$, which contradicts the fact that $\gamma(t_0,z_1) \in \Omega \setminus \overline{G}$. 
\end{proof}

\begin{lemma}\label{lem:W:boundary}
Let $z_0 \in \partial W$.
Then $\gamma(t,z_0) \in \partial W$ for any $t \in \mathbb{R}$.
\end{lemma}
\begin{proof}
The proof is similar to that of Lemma~\ref{lem:G:boundary}. We provide details for clarity.
Since $W$ is open by Lemma~\ref{lem:W:open}, we get $z_0 \not\in W$. 
By Lemma~\ref{lem:GWL-intersection}, we have $G \cap \overline{W}=\emptyset$, which implies that $\gamma(\cdot,z_0)$ cannot reach $\Gamma^D$ in a finite time.
Consequently, $I = \mathbb{R}$.
Suppose, by contradiction, that there exists $t_0 \in \mathbb{R}$ such that $z = \gamma(t_0,z_0) \not\in \pa W$.
That is, we have either $z \in W$ or $z \in (\Omega \cup \Gamma^N) \setminus \overline{W}$.
The former case is impossible since the definition \eqref{def:W} of $W$ would imply $\gamma(\cdot,z_0) \subset W$ and hence $z_0 \in W$.
In the latter case, 
there exists a neighborhood $V$ of $z$ such that $V \cap \Omega \subset \Omega \setminus \overline{W}$. 
Using Lemma~\ref{lem:Cuchy-contin}, for a sufficiently small neighborhood $U$ of $z_0$ we have $\gamma(t_0,U \cap \Omega) \subset \Omega \setminus \overline{W}$.
However, since $z_0 \in \partial W$, $U$ contains a point $z_1 \in W \subset \Omega$ and hence $\gamma(\cdot,z_1) \subset W$, which is a contradiction. 
\end{proof}

\begin{lemma}\label{lem:Lr:boundary}
Let $r\in \mathcal{S}$ and $z_0\in \pa L(r) \setminus \Gamma^D$. 
Then $\ga(t, z_0) \in \pa L(r)$ for any $t \in I$. 
\end{lemma}
\begin{proof}  
Suppose, by contradiction, that there exists $t_0 \in I$ such that $z = \gamma(t_0,z_0) \not\in \pa L(r)$.
Consider the case $z \in \Omega \cup \Gamma^N$.
Since $z \not\in \pa L(r)$, we have either $z \in \mathrm{Int}(L(r))$ or $z \in (\Omega \cup \Gamma^N) \setminus \overline{L(r)}$. 
Suppose that $z \in \mathrm{Int}(L(r))$. 
Then there is a neighborhood $U$ of $z$ contained in $\mathrm{Int}(L(r))$,
i.e., the flow line of every point from $U$ converges to $r$. 
Moreover, since $z \in \mathrm{Int}(L(r)) \subset \Omega$, we have $z_0 \in \Omega$ by Corollary~\ref{cor:boundary}.
Hence, using Lemma~\ref{lem:Cuchy-contin}, we see that $z_0 \in \mathrm{Int}(L(r))$, which contradicts our assumption $z_0\in \pa L(r)$.
Suppose that $z\in (\Om \cup \Gamma^N) \setminus \overline{L(r)}$.
Since $\mathrm{dist}(z,\overline{L(r)})>0$, there exists a neighborhood $V$ of $z$ such that $V \cap (\Omega \cup \Gamma^N) \subset (\Omega \cup \Gamma^N) \setminus \overline{L(r)}$. 
Therefore, by Lemmas~\ref{lem:Cuchy-contin} and~\ref{lem:traj-neum}, for a sufficiently small neighborhood $U$ of $z_0$ we have $\gamma(t_0,U \cap (\Omega \cup \Gamma^N)) \subset (\Omega \cup \Gamma^N) \setminus \overline{L(r)}$.
However, since $z_0 \in \partial L(r) \setminus \Gamma^D$, $U$ contains a point $z_1 \in (\Omega \cup \Gamma^N) \cap L(r)$ and hence $\gamma(\cdot,z_1) \subset (\Omega \cup \Gamma^N) \cap L(r)$, which is a contradiction to  $\gamma(t_0,z_1) \in (\Omega \cup \Gamma^N) \setminus \overline{L(r)}$. 
Thus, we proved that $\gamma(t,z_0) \in \partial L(r)$ whenever $\gamma(t,z_0) \in \Omega \cup \Gamma^N$ for $t \in I$.
Now we assume that $z \in \Gamma^D$.
In view of the previous case, the continuity of $\gamma(\cdot,z_0)$ and the closendness of $\partial L(r)$ guarantee that $z \in \partial L(r)$.
This completes the proof.
\end{proof}    

\begin{corollary}\label{cor:lem:Lr:boundary}
Let $r\in \mathcal{S}$ and $z_0 \in \mathrm{Int}(L(r))$.
Then $\gamma(t,z_0) \in \mathrm{Int}(L(r))$ for any $t \in \widetilde{I}$, where $\widetilde{I}$ is the maximal open interval of existence of $\ga(\cdot, z_0)$ in $\Omega$.
\end{corollary}

\subsection{Structure of \texorpdfstring{$\partial L(r)$}{dL(r)}}
The aim of this section is to show that the boundary $\partial L(r)$ consists of at most finitely many flow lines (together with their end points), arcs of critical points, and parts of $\Gamma^D$. 
We start by providing several auxiliary result about the structure of $\partial L(r)$.
We refer to \cite{andronov1974qualitative} for various results of related nature.

For any $r \in \mathcal{S}$ we consider its level set
$$
E(r) = \{z:~ u(z)=u(r)\},
$$
where we assume $u$ to be extended to a neighborhood of $r$ whenever $r \in \partial \Omega$.
Thanks to Lemma~\ref{lem:numsaddlesOmega}, there exists a neighborhood $U$ of $r$ that does not contain any other critical point of $u$. 
By Remark~\ref{rem:level-sets-diffeomorhpic}, we can take $U$ smaller so that $E(r) \cap U$ is homeomorphic to the zero level set of a corresponding prototypical function (described in Section~\ref{sec:classification}) in a neighborhood of $(0,0)$, and hence $U \setminus E(r)$ consists of \textit{finitely many} open connected components, each of which has $r$ on its boundary.
We denote those connected components that intersect with $\Omega$ as 
$K_1,\dots,K_l$, where $l \geq 1$.
\begin{lemma}\label{lem:observation}
Let $r \in \mathcal{S}$. 
Then $\partial L(r)$ contains at least $l$ and at most $2l$ flow lines converging to $r$, and $\mathrm{Int}(L(r)) \cap U$ has at most $l$ connected components.
\end{lemma}

\begin{proof}
Assume that $r \in \Omega$ and take any connected component $K := K_n$, where $n \in \{1,\ldots,l\}$. 
Assume, without loss of generality, that $u(z) > u(r)$ for any $z \in K$. 
Denote by $\partial K^1$ and $\partial K^2$ those parts of the boundary of $K$ that contain $r$. 
In particular, $\partial K^1, \partial K^2 \subset E(r)$. 
Consider the sets $A_1$ and $A_2$ defined as unions of flow lines emanating from $\partial K^1$ and $\partial K^2$ towards $K$, respectively, i.e.,
$$
A_i = K \cap \bigcup_{z \in \partial K^i \setminus \{r\}} \gamma(\cdot,z),
\quad i=1,2,
$$
see Figure~\ref{fig:lem48}.
Since $r$ is the only critical point in $U$, any $z \in \partial K^i \setminus \{r\}$ is a regular point. 
Therefore, it can be shown in much the same was as in Lemma~\ref{lem:G:open} that $A_i$ is an open set.
Since $u$ is strictly monotone along flow lines, we see that $A_1 \cap A_2 = \emptyset$. 

\begin{figure}[!ht]
\begin{center}
\includegraphics[width=0.5\linewidth]{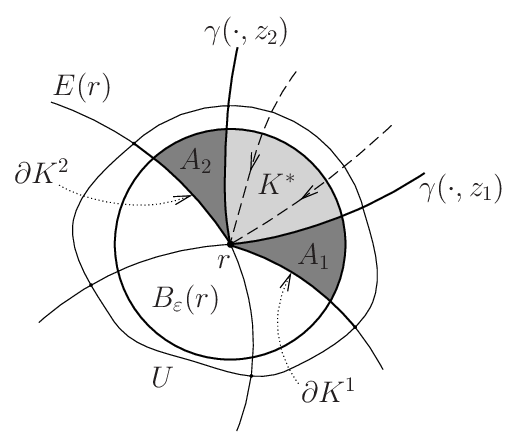}
\end{center}
\caption{Schematic plot of construction.}
\label{fig:lem48}
\end{figure}

 Observe that, by the openness of $K,$ we have $\partial A_i \cap K \cap \partial K^i=\emptyset$. Let us take any disk $B_\varepsilon(r) \subset U$.
By the continuity of $u$ and closedness of $\partial A_i$, we have 
$$
m_i := \min_{\xi \in \partial A_i \cap K \cap \partial B_\varepsilon(r)} u(\xi) > u(r).
$$
Let us now take any $z_i \in \partial A_i \cap K \cap B_\varepsilon(r)$ so that $u(z_i)<m_i$. 
Arguing in much the same way as in Lemma~\ref{lem:G:boundary} and recalling that $u(z_i)<m_i$ and $u$ is strictly monotone along flow lines, we see that $\gamma(t,z_i) \in \partial A_i \cap K \cap B_\varepsilon(r)$ for any $t \geq 0$. 
Hence, $\gamma(\cdot,z_i)$ converges to $r$ as $t \to +\infty$ since $r$ is the only critical point in $U$.

Thus, we get $\gamma(\cdot,z_i) \subset L(r)$, and in fact we have
$\gamma(\cdot,z_i) \subset \partial L(r)$ by construction of the sets $A_i$.
If $\gamma(\cdot,z_1)$ does not coincide with $\gamma(\cdot,z_2)$, then using Proposition~\ref{prop:absil-x} over a curvilinear sector $K^* \subset K$ bounded by $\gamma(\cdot,z_1)$, $\gamma(\cdot,z_2)$, and $\partial B_\varepsilon(r)$, we get the existence of a disk $B_\sigma(r)$ such that the flow line of any point $z \in B_\sigma(r) \cap K^*$ converges to $r$, i.e., $B_\sigma(r) \cap K^* \subset \mathrm{Int}(L(r))$. 
In particular, there are no flow lines from $\partial L(r) \cap K$ converging to $r$ other than $\gamma(\cdot,z_1)$ and $\gamma(\cdot,z_2)$.
On the other hand, if $\gamma(\cdot,z_1)$ coincides with $\gamma(\cdot,z_2)$, then this is the only flow line in $\partial L(r) \cap K$ converging to $r$ and $\mathrm{Int}(L(r)) \cap K = \emptyset$.
Therefore, the claim of the lemma is proved in the case $r \in \Omega$. 

If $r \in \partial \Omega$, then we argue in the same way as above for the extension of $u$ to a neighborhood of $r$. 
Note that if $r \in \Gamma^D$, then $\Gamma^D \subset E(r)$, while if $r \in \Gamma^N$, then $\Gamma^N$ contains two flow lines converging to $r$, and hence they belong to $\partial L(r)$.
Thus, restricting back to $\overline{\Omega}$, we obtain the desired claim.
\end{proof}

\begin{lemma}\label{lem:Lr-Int-empty}
Let $r \in \mathcal{S}$. 
If $\mathrm{Int}(L(r)) = \emptyset$, then $L(r) \subset \partial L(r)$, $\partial L(r) \setminus L(r)$ consists of a finite number of critical points, and any flow line in $\partial L(r)$ converges to $r$. 
\end{lemma}
\begin{proof}
The inclusion $L(r) \subset \partial L(r)$ is a trivial set-theoretic consequence of the assumption $\mathrm{Int}(L(r)) = \emptyset$.
In particular, any flow line from $L(r)$ (which converges to $r$ by definition) belongs to $\partial L(r)$ and hence, by Lemma~\ref{lem:observation}, $L(r)$ consists of finitely many flow lines. 
Consequently, $\partial L(r) \setminus L(r)$ consists of a finite number of critical points, and therefore any flow line in $\partial L(r)$ converges to $r$.
\end{proof}

Let us note that $\partial L(r)$ may contain flow lines that do not converge to $r$. 
Using Lemmas~\ref{lem:observation} and~\ref{lem:Lr-Int-empty}, we obtain the following result on the global behavior of $\partial L(r)$. 

\begin{proposition}\label{prop:partialL-finite-flow-lines}
Let $r \in \mathcal{S}$. 
Then $\partial L(r)$ contains at most finitely many flow lines, each of which converges to a saddle point.
\end{proposition}
\begin{proof}
If $\mathrm{Int}(L(r)) = \emptyset$, then the desired claim follows from  Lemmas~\ref{lem:observation} and~\ref{lem:Lr-Int-empty}.
Thus, we are interested in the case $\mathrm{Int}(L(r)) \neq \emptyset$.

First, we investigate the set $\pa L(r) \cap \Omega$.
Thanks to Lemma~\ref{lem:observation}, $\pa L(r) \cap \Omega$ contains finitely many flow lines converging to $r$.
Thus, our aim is to show that there are at most finitely many flow lines in $\pa L(r) \cap \Omega$ which do not converge to $r$.
Let us take any regular point $z\in \pa L(r) \cap \Omega$ and
let $\gamma = \gamma(\cdot,z) \subset \partial L(r)$ be its flow line satisfying $\gamma(+\infty,z) \neq r$ and $\gamma(-\infty,z) \neq r$.
According to Lemma~\ref{lem:classif-gamma}, there are only the following possible options on the behavior of $\gamma$:
\begin{enumerate}[label={(\arabic*)}]
\item\label{prop:proof:1} Both ends of $\gamma$ are local extremum points.

\item\label{prop:proof:2} One end of $\gamma$ is a local extremum point and another end is a saddle point different from $r$.

\item\label{prop:proof:3} Both ends of $\gamma$ are saddle points different from $r$. 

\item\label{prop:proof:4} One end of $\gamma$ is a local extremum point and $\gamma$ reaches $\Gamma^D$ in a finite time.

\item\label{prop:proof:5} One end of $\gamma$ is a saddle point different from $r$ and $\gamma$ reaches $\Gamma^D$ in a finite time. 
\end{enumerate}
Let us analyze each case separately. 
Recall that $\mathcal{M}_+$ and $\mathcal{M}_-$ denote the sets of local maximum and minimum points of $u$ in $\overline{\Om}$, respectively; see Section~\ref{sec:properties}.

\noindent
\ref{prop:proof:1} 
Since
$\gamma(+\infty,z) \in \mathcal{M}_-$ and $\gamma(-\infty,z) \in \mathcal{M}_+$, we have $z \in W$. 
However, $W \cap L = \emptyset$ by Lemma~\ref{lem:GWL-intersection}.
Thus, this case cannot occur.

\noindent
\ref{prop:proof:2} 
Without loss of generality, assume that $\xi := \gamma(+\infty,z) \in \mathcal{M}_-$ and $\eta := \gamma(-\infty,z) \in \mathcal{S} \setminus \{r\}$, where the latter inclusion means $\gamma \subset L(\eta)$.
Thanks to Lemma~\ref{lem:Lr:boundary}, we distinguish two subcases:

\begin{enumerate}[label={(\alph*)}]
\item\label{prop:proof:2:a} $\gamma \subset \partial L(\eta)$.
Since $\gamma$ converges to $\eta$ and $\mathcal{S}$ is finite, we use Lemma~\ref{lem:observation} to conclude that this subcase can occur at most finitely many times. 
In other words, there are at most finitely many flow lines in $\partial L(r) \cap \Omega$ realizing this subcase.

\item\label{prop:proof:2:b} $\gamma \subset \mathrm{Int}(L(\eta))$. 
Then there exists a neighborhood $U \subset \mathrm{Int}(L(\eta))$ of $z$. 
That is, the flow line of any point from $U$ converges to $\eta$ as $t \to -\infty$.
On the other hand, since $\xi \in \mathcal{M}_-$, we use
Proposition~\ref{prop:absil} in combination with 
Lemma~\ref{lem:Cuchy-contin} 
to deduce that $U$ can be decreased so that the flow line of any point from $U$ converges either to $\xi$ (if $\xi$ is isolated) or to a critical point in a neighborhood of $\xi$ (if $\xi$ is nonisolated) as $t \to +\infty$.
In the latter case, we can assume that this critical point is not $r$, thanks to Lemma~\ref{lem:numsaddlesOmega}.
Consequently, $U \cap L(r) = \emptyset$.
However, our assumption $z \in \partial L(r)$  implies that $U \cap L(r) \neq \emptyset$, which is a contradiction.
Thus, this subcase cannot occur.
\end{enumerate}

\noindent
\ref{prop:proof:3} Let $\xi := \gamma(+\infty,z) \in \mathcal{S} \setminus \{r\}$ and $\eta := \gamma(-\infty,z) \in \mathcal{S} \setminus \{r\}$.
We distinguish two subcases:

\begin{enumerate}[label={(\alph*)}]
\item\label{prop:proof:3:a} 
$\gamma \subset \partial L(\xi)$ or $\gamma \subset \partial L(\eta)$. 
As in the case \ref{prop:proof:2}-\ref{prop:proof:2:a}, Lemma~\ref{lem:observation} implies that this subcase can occur at most finitely many times.

\item\label{prop:proof:3:b} 
$\gamma \subset \mathrm{Int}(L(\xi))$ and $\gamma \subset \mathrm{Int}(L(\eta))$. 
Similarly to the case \ref{prop:proof:2}-\ref{prop:proof:2:b}, we conclude that the flow line of any point from a neighborhood $U \subset \mathrm{Int}(L(\xi)) \cap \mathrm{Int}(L(\eta))$ of $z$ converges to $\xi$ as $t \to +\infty$ and to $\eta$ as $t \to -\infty$. That is, $U \cap L(r) = \emptyset$, which contradicts the assumption $z \in \partial L(r)$.
Thus, this subcase cannot occur.
\end{enumerate}

\noindent
\ref{prop:proof:4} 
Without loss of generality, assume that $\xi := \gamma(+\infty,z) \in \mathcal{M}_-$ and $\eta := \gamma(t_0,z) \in \Gamma^D$ for some $t_0 < 0$. 
The latter yields $z \in G$, and hence Lemma~\ref{lem:G:open} implies the existence of a neighborhood $U \subset G$ of $z$.
On the other hand, similarly to the case \ref{prop:proof:2}-\ref{prop:proof:2:b}, we deduce from
Proposition~\ref{prop:absil} in combination with 
Lemma~\ref{lem:Cuchy-contin} that $U$ can be decreased so that the flow line of any point from $U$ converges either to $\xi$ (if $\xi$ is isolated) or to a critical point in a neighborhood of $\xi$ different from $r$ (if $\xi$ is nonisolated) as $t \to +\infty$.
Again, we get $U \cap L(r) = \emptyset$, which is impossible since $z \in \partial L(r)$.
Thus, this case cannot occur.

\noindent
\ref{prop:proof:5} 
Without loss of generality, assume that $\xi := \gamma(+\infty,z) \in \mathcal{S} \setminus \{r\}$ and $\eta := \gamma(t_0,z) \in \Gamma^D$ for some $t_0 < 0$. 
We distinguish two subcases:

\begin{enumerate}[label={(\alph*)}]
\item\label{prop:proof:5:a} 
$\gamma \subset \partial L(\xi)$. 
As in the case \ref{prop:proof:2}-\ref{prop:proof:2:a}, Lemma~\ref{lem:observation} implies that this subcase can occur at most finitely many times.

\item\label{prop:proof:5:b}  $\gamma \subset \mathrm{Int}(L(\xi))$. 
Since the sets $\mathrm{Int}(L(\xi))$ and $G$ are open, there exists a neighborhood $U \subset \mathrm{Int}(L(\xi)) \cap G$ of $z$.
That is, the flow line of any point from $U$ does not converge to $r$, i.e., $U \cap L(r) = \emptyset$, which contradicts $z \in \partial L(r)$.
Thus, this subcase cannot occur.
\end{enumerate}

\smallskip
Taking now all the subcases into account, we end up only with the following possibilities: \ref{prop:proof:2}-\ref{prop:proof:2:a}, \ref{prop:proof:3}-\ref{prop:proof:3:a}, and \ref{prop:proof:5}-\ref{prop:proof:5:a}. 
Each of these cases provides at most \textit{finitely many} flow lines in $\pa L(r) \cap \Omega$, 
each of which converges to some saddle point.

Let us now investigate the set $\partial L(r) \cap \partial \Omega$.
We have $\gamma(\cdot,z) \cap \Gamma^D = \{z\}$ for any regular point $z \in \Gamma^D$ (see Remark~\ref{rem:traj-boundary}). 
In particular, $\partial L(r) \cap \Gamma^D$ does not contain flow lines. 
Recall that any flow line starting on $\pa L(r) \setminus \Gamma^D$ stays on $\pa L(r)$ for all admissible times (see Lemma~\ref{lem:Lr:boundary}).
Using this fact, Lemma~\ref{lem:traj-neum} and Corollary~\ref{cor:boundary-GN} imply that $\partial L(r) \cap \Gamma^N$ contains at most finitely many flow lines.
It remains to show that any such flow line converges to a saddle point. 
Arguing by contradiction, we end up with the case~\ref{prop:proof:1}. 
Using Proposition~\ref{prop:absil} and
Lemma~\ref{lem:Cuchy-contin} in much the same way as in the proof of \ref{prop:proof:2}-\ref{prop:proof:2:b}, we reach a contradiction.
\end{proof}

\begin{lemma}\label{lem:partialL-isolated}
Let $r \in \mathcal{S}$. 
Then $\partial L(r)$ does not contain isolated points.
\end{lemma}
\begin{proof}
Suppose, by contradiction, that there exists an isolated point $z$ of $\partial L(r)$.
(Note that $z \neq r$, see Lemma~\ref{lem:observation}.)
That is, there is a neighborhood $U$ of $z$ such that $U \cap \partial L(r) = \{z\}$.
Therefore, $(U \cap L(r))\setminus\{z\} = U \cap \mathrm{Int}(L(r))$ and 
\begin{equation}\label{eq:decompU}
(U \cap \Omega) \setminus \{z\} = 
[U \cap \mathrm{Int}(L(r))]
\cup
[U \cap ({\Omega} \setminus \overline{L(r)})].
\end{equation}
Since $z\in \partial L(r)$, we  must have $(U \cap L(r))\setminus\{z\}\neq\emptyset$ and hence $U \cap\mathrm{Int}(L(r)) \neq \emptyset$. 
Clearly, the sets in the square brackets on the right-hand side of \eqref{eq:decompU} are disjoint open sets.
Moreover, we can assume $U \cap {\Omega}$ to be connected. 
Therefore, \eqref{eq:decompU} implies that $U \cap ({\Omega} \setminus \overline{L(r)}) = \emptyset$, and, consequently, 
$(U \cap {\Omega}) \setminus \{z\} \subset \mathrm{Int}(L(r)) \subset L(r)$.

Assume first that $z \in \partial \Omega$. 
In view of the inclusion $(U \cap {\Omega}) \setminus \{z\} \subset L(r)$, any $\xi \in U \cap \partial \Omega$ can be approximated by points from $L(r)$ and hence $\xi \in \partial L(r)$.
Thus, $z$ is not an isolated point of $\partial L(r)$, a contradiction.

Assume now that $z \in \Omega$.
If $z$ is a regular point of $u$, then $\gamma(\cdot,z) \subset \partial L(r)$ by Lemma~\ref{lem:Lr:boundary}, which is impossible since $z$ is isolated in $\partial L(r)$.
Let $z$ be a critical point of $u$.
For convenience, we decrease the neighborhood $U$ obtained in the first paragraph so that $U \subset \Omega$. 
Since $U \setminus \{z\} \subset L(r)$, $z \neq r$, and $u$ is strictly monotone along flow lines, we can assume, without loss of generality, that $u(z)>u(r)$. 
By the continuity of the flow, $\gamma(t, \partial U)$ encapsulates $z$ for all $t \in I$.
However, since $\partial U$ is compact, Lemma~\ref{lem:Cuchy-contin} implies that for any neighborhood $V$ of $r$ there exists $T>0$ such that $\gamma(t, \partial U) \subset V$ for any $t \geq T$.
Recalling that $\mathrm{dist}(z,r)>0$ and taking $V$ small enough, we get a contradiction.
\end{proof}

In the following lemma, we describe the behavior of $\partial L(r)$ near curves of critical points of $u$. 
\begin{lemma}\label{lem:partialL-theta}
Let $\theta$ be a curve of critical points of $u$.
Let $r \in \mathcal{S}$ and $z \in \theta$ be such that $z \in \partial L(r)$ and there is no flow line $\gamma \subset \partial L(r)$ converging to $z$.
Then there exists a neighborhood $U$ of $z$ such that
$\theta \cap U = \partial L(r) \cap U$.
\end{lemma}
\begin{proof}
In view of Corollary~\ref{cor:no-saddles-on-theta}, we have $z \not\in \mathcal{S}$, and hence we assume, without loss of generality, that $z$ is a local minimum point of $u$.
By Lemma~\ref{lem:isol}~\ref{lem:isol:simple}, $\theta$ is isolated in $\mathcal{C}$, i.e., there exists a neighborhood $V$ of $z$ such that all critical points of $u$ in $V \cap \overline{\Omega}$ belong to $\theta$.
Moreover, by Lemma~\ref{lem:morse-bott}~\ref{lem:morse-bott:gamma-N}, we can further assume that if $z \in \Gamma^N$, then either $V \cap \theta = V \cap \Gamma^N$ or $V \cap \theta \cap \Gamma^N = \{z\}$, see Figure~\ref{fig:theta12x}.

Suppose that there exists a sequence $\{\eta_n\} \subset \partial L(r) \setminus \theta$ such that $\eta_n \to z$. 
By the choice of $V$, each $\eta_n$ is a regular point.
If infinitely many flow lines $\gamma(\cdot,\eta_n)$ are mutually different, then we get a contradiction to Lemma~\ref{lem:Lr:boundary} and Proposition~\ref{prop:partialL-finite-flow-lines}. 
On the other hand,
if infinitely many flow lines $\gamma(\cdot,\eta_n)$ coincide, then $\gamma(\cdot,\eta_n) \subset \partial L(r)$ is a flow line converging to $z$, which contradicts the assumption of the lemma.
Therefore, there exists a neighborhood $U \subset V$ of $z$ such that $U \cap (\partial L(r) \setminus \theta) = \emptyset$.
Consequently, we have
\begin{equation}\label{eq:proof:4.13:1}
\partial L(r) \cap U \subset \theta \cap U
\end{equation}
and $U \cap (L(r) \setminus \theta) = U \cap (\mathrm{Int}(L(r)) \setminus \theta)$, and hence any connected component of $(U \cap \Omega) \setminus \theta$ is a subset of either $\mathrm{Int}(L(r))$ or $\Omega \setminus \overline{L(r)}$.

Since $z \in \partial L(r)$, we have $U \cap (L(r) \setminus \theta) \neq \emptyset$, and thus there exists at least one connected component of $(U \cap \Omega) \setminus \theta$ which belong to $\mathrm{Int}(L(r))$.
If $z \in \Omega$, then $(U \cap \Omega) \setminus \theta$ consists of two connected components (up to shrinking $U$, if necessary), each of which has $\theta \cap U$ as a part of its boundary. 
Therefore, we get $\theta \cap U \subset \partial L(r) \cap U$. 
Combining this set inclusion with \eqref{eq:proof:4.13:1}, we obtain the desired claim in the case $z \in \Omega$.
If $z \in \Gamma^N$ and $U \cap \theta = U \cap \Gamma^N$, then $(U \cap \Omega) \setminus \theta$ is connected, which yields the same result.

In view of the choice of the neighborhood $V$ (see the first paragraph), it remains to consider the case $U \cap \theta \cap \Gamma^N = \{z\}$. 
We deduce from Lemma~\ref{lem:morse-bott} that $(U \cap \Omega) \setminus \theta$ has two connected components, each of which has a part of $\Gamma^N$ on its boundary, and such a part contains flow lines converging to $z$. 
Since at least one connected component is a subset of $\mathrm{Int}(L(r))$, we find flow lines from $\partial L(r) \cap \Gamma^N$ converging to $z$, which contradicts the assumption of the lemma. 
\end{proof}

Using the results obtained above, we provide the following global description of the structure of $\partial L(r)$.
\begin{proposition}\label{prop:boundary_Lr}
$\bigcup_{r \in \mathcal{S}} \partial L(r) \setminus \Gamma^D$ consists of at most finitely many flow lines (together with their end points outside of $\Gamma^D$) and arcs of critical points. 
\end{proposition}
\begin{proof}
Since the number of saddle points is finite, it is sufficient to prove the statement for any given $r \in \mathcal{S}$.
We already know from Proposition~\ref{prop:partialL-finite-flow-lines} that $\partial L(r)$ contains finitely many flow lines.
Thus, it remains to describe the structure of the part of $\partial L(r) \setminus \Gamma^D$ consisting of critical points which are not end points of flow lines from $\partial L(r)$.

Let us take any such critical point $z\in \pa L(r) \setminus \Gamma^D$. 
Thanks to Lemma~\ref{lem:partialL-isolated}, $z$ is nonisolated point of $\partial L(r) \setminus \Gamma^D$.
Thus, we get a sequence $\{z_n\} \subset (\partial L(r) \setminus \Gamma^D) \setminus \{z\}$ such that $z_n \to z$.
If there exists a subsequence of $\{z_n\}$ consisting of regular points belonging to a flow line $\gamma$, then $\gamma \subset \partial L(r)$ by Lemma~\ref{lem:Lr:boundary}. Hence $\gamma$ converges to $z$, which contradicts our assumption on $z$.
On the other hand, in view of Proposition~\ref{prop:partialL-finite-flow-lines}, only finitely many elements of $\{z_n\}$ can be regular points with mutually different flow lines.
Therefore, except of these finitely many elements, any 
$z_n$ is a critical point of $u$.
According to Lemma~\ref{lem:isol}, there exists a curve $\theta$ of critical points such that $z \in \theta$ and $z_n \in \theta$ for any sufficiently large $n$.

Thanks to Lemma~\ref{lem:partialL-theta}, there exists a neighborhood $U$ of $z$ such that $\theta \cap U = \partial L(r) \cap U$. 
We apply Lemma~\ref{lem:partialL-theta} along $\theta$ until either $\theta$ ends up at $\Gamma^N$ or there appears a point $\zeta \in \theta \cap \partial L(r)$ such that there exists a flow line from $\partial L(r)$ converging to $\zeta$. 
Since $u$ has (at most) finitely many curves of critical points (see Lemma~\ref{lem:isol}) and flow lines in $\partial L(r)$ (see Proposition~\ref{prop:partialL-finite-flow-lines}), we conclude that there exist finitely many arcs of critical points in $\partial L(r)$.
\end{proof}

\subsection{Structure of \texorpdfstring{$\partial G$}{dG} and \texorpdfstring{$\partial W$}{dW}}
In this section, we show that the boundaries $\partial G$ and $\partial W$ outside of $\Gamma^D$ consist of at most finitely many flow lines, arcs of critical points, and isolated critical points, cf.\ Proposition~\ref{prop:boundary_Lr}.

First, we need a result of the same type as Proposition~\ref{prop:partialL-finite-flow-lines}.
\begin{lemma}\label{lem:partialGW-finite-flow-lines}
$\partial G \cup \partial W$ contains at most finitely many flow lines, and any flow line from $(\partial G \cup \partial W) \setminus \Gamma^N$ converges to a saddle point. 
\end{lemma}
\begin{proof}
Since $\Gamma^D$ contains no flow lines (see Remark~\ref{rem:traj-boundary}), we take any regular point $z \in (\partial G \cup \partial W) \setminus \Gamma^D$ and let $\gamma = \gamma(\cdot,z)$ be the corresponding flow line.
Due to Lemmas~\ref{lem:G:boundary} and \ref{lem:W:boundary}, we have $\gamma \subset \partial G \cup \partial W$.
Since $G$, $W$ are open (by Lemmas~\ref{lem:G:open}, ~\ref{lem:W:open}), $\gamma$ cannot reach $\Gamma^D$ in a finite time, and hence $\gamma$ connects two critical points of $u$, see Lemma~\ref{lem:classif-gamma}.
Denote    
$\xi := \gamma(+\infty,z)$ and $\eta := \gamma(-\infty,z)$.
Since $\Gamma^N$ contains at most finitely many flow lines by Corollary~\ref{cor:boundary-GN}, it is sufficient to assume that at least one of the points $\xi, \eta$ does not belong to $\Gamma^N$.
Since $G$, $W$ are open and $G \cap W = \emptyset$ (see Lemma~\ref{lem:GWL-intersection}), $\xi$ and $\eta$ cannot be two local extremum points of $u$, and hence at least one of them is a saddle point.

Without loss of generality, assume that $\eta \in \mathcal{S}$.
Consequently, $\gamma \subset L(\eta)$.
In view of Lemma~\ref{lem:Lr:boundary}, we have either $\gamma \subset \partial L(\eta)$ or $\gamma \subset \mathrm{Int}(L(\eta))$.
By Proposition~\ref{prop:partialL-finite-flow-lines}, $\partial L(\eta)$ contains finitely many flow lines.
Thus, suppose that $\gamma \subset \mathrm{Int}(L(\eta))$.
According to Lemma~\ref{lem:classif-gamma}, we have the following two cases for $\xi$: either $\xi \in \mathcal{M}_-$ or $\xi \in \mathcal{S}$.
We discuss these cases separately. 

\begin{enumerate}[label={(\arabic*)}]
\item\label{proof:lemm414:1} Let $\xi \in \mathcal{M}_-$.
Since $\gamma \subset \mathrm{Int}(L(\eta))$, there exists a neighborhood $U \subset \mathrm{Int}(L(\eta))$ of $z$. 
That is, the flow line of any point from $U$ converges to $\eta$ as $t \to -\infty$.
On the other hand, since $\xi \in \mathcal{M}_-$, we use
Proposition~\ref{prop:absil} in combination with 
Lemma~\ref{lem:Cuchy-contin} to deduce that $U$ can be decreased so that the flow line of any point from $U$ converges either to $\xi$ (if $\xi$ is isolated) or to a critical point in a neighborhood of $\xi$ (if $\xi$ is nonisolated) as $t \to +\infty$.
That is, $U \cap G = \emptyset$ and $U \cap W = \emptyset$, which contradicts the assumption $z \in \partial G \cup \partial W$. 
Thus, this case cannot occur.

\item Let $\xi \in \mathcal{S}$, and hence $\gamma \subset L(\xi)$.
We distinguish two subcases:

\begin{enumerate}[label={\rm(\alph*)}] 
	\item $\gamma \subset \partial L(\xi)$. In this scenario, Proposition~\ref{prop:partialL-finite-flow-lines} implies that there are finitely many such flow lines.
	
	\item $\gamma \subset \mathrm{Int}(L(\xi))$. We argue similarly to the case~\ref{proof:lemm414:1} and conclude that the flow line of any point from a neighborhood $U \subset \mathrm{Int}(L(\xi)) \cap \mathrm{Int}(L(\eta))$ of $z$ converges to $\xi$ as $t \to +\infty$ and to $\eta$ as $t \to -\infty$. 
	Since $\xi, \eta \in \mathcal{S}$, we see that $U \cap G = \emptyset$ and $U \cap W = \emptyset$, which contradicts the assumption 
	$z \in \partial G \cup \partial W$.
	Thus, this subcase cannot occur.
\end{enumerate}
\end{enumerate}	 
As a result, we conclude that $\partial G \cup \partial W$ contains at most finitely many flow lines. 
\end{proof}

\begin{remark}
    If $\Gamma^N \neq \emptyset$, then $\partial G \cap \partial W$ might contain a flow line connecting two local extremum points of $u$ when both such points belong to $\Gamma^N$. 
    For instance, this happens in the case of the second Neumann eigenfunction in a disk.
\end{remark}

Second, we need a result of the same nature as Lemma~\ref{lem:partialL-isolated}.
\begin{lemma}\label{lem:G:parG-finitely-many-points}
$\partial G \cup \partial W$ contains at most finitely many isolated points, each of which is a critical point of $u$ in $\Omega$.
\end{lemma}
\begin{proof}
We start the proof with the analysis of $\partial G$.
In view of Lemmas~\ref{lem:G:open} and~\ref{lem:G:boundary}, $\partial G$ cannot contain \textit{regular} points of $u$ which are isolated in $\partial G$. 
Therefore, let us show that $\partial G$ contains finitely many \textit{critical} points which are isolated in $\partial G$.
Let $z$ be any such point, i.e., there exists a neighborhood $U$ of $z$ such that $\partial G \cap U = \{z\}$. 
Since $G$ is open by Lemma~\ref{lem:G:open}, we argue in the same way in the first paragraph of the proof of Lemma~\ref{lem:partialL-isolated} to conclude that $(U \cap \Omega) \setminus \{z\} \subset G$.
If $z \in \partial\Omega$, then we see that any point from $U \cap \partial \Omega$ can be approximated by points from $G$, and hence $U \cap \partial \Omega \subset \partial G$, which contradicts the isolation of $z$ in $\partial G$.

Let $z \in \Omega$.	
Decreasing $U$, we assume that $U \subset \Omega$, for convenience.
It is clear from the definition of $G$ that it does not contain critical points.
Consequently, the inclusion $U \setminus \{z\} \subset G$ implies that $z$ is an isolated critical point of $u$ in $U$. 
By Lemma~\ref{lem:isol}, $u$ possesses only finitely many isolated critical points, which completes the proof for $\partial G$.

Let us now discuss $\partial W$.
Let $z$ be an isolated point of $\partial W$, i.e., there exists a neighborhood $U$ of $z$ such that $\partial W \cap U = \{z\}$. 
Thanks to Lemmas~\ref{lem:W:open} and~\ref{lem:W:boundary}, we argue in much the same way as above to obtain that $z \in \Omega$, $U \setminus \{z\} \subset W$, and $z$ is an isolated critical point of $u$. 
The application of Lemma~\ref{lem:isol} finishes the proof.	
\end{proof}

\begin{remark}
Unlike $\partial L(r)$, the boundaries $\partial G$ and $\partial W$ \textit{can} possess isolated points.
Simple explicit examples are the first and second radial Dirichlet eigenfunctions in a disk, respectively.	
\end{remark}

Third, we provide a result of the same type as Lemma~\ref{lem:partialL-theta}.
\begin{lemma}\label{lem:partialG-theta}
Let $\theta$ be a curve of critical points of $u$.
Let $z \in \theta$ be such that $z \in \partial G$ (or $z \in \partial W$) and there is no flow line $\gamma \subset \partial G$ (or $\gamma \subset \partial W$) converging to $z$.
Then there exists a neighborhood $U$ of $z$ such that 
$\theta \cap U = \partial G \cap U$ (or $\theta \cap U = \partial W \cap U$).
\end{lemma}

\begin{proof}
The proof goes along the same lines as in Lemma~\ref{lem:partialL-theta}. 
For completeness, we sketchily provide arguments for the case $\partial G$ only. 
Without loss of generality, we can assume that $z$ is a local minimum point of $u$. 
Since $\theta$ is isolated in $\mathcal{C}$, there exists a neighborhood $V$ of $z$ such that all critical points of $u$ in $V \cap \overline{\Omega}$ belong to $\theta$. Moreover, in view of Lemma~\ref{lem:isol}~\ref{lem:isol:gamma-N}, there are two possibilities when $z \in \Gamma^N$: either $V \cap \theta = V \cap \Gamma^N$ or $V \cap \theta \cap \Gamma^N = \{z\}$.

Let there exist a sequence $\{\eta_n\} \subset \partial G \setminus \theta$ such that $\eta_n \to z$. 
We observe that each $\eta_n$ is a regular point of $u$ by the choice of $V$. 
Using similar arguments as in the proof of Lemma~\ref{lem:partialL-theta} along with the help of Lemmas~\ref{lem:G:boundary} and \ref{lem:partialGW-finite-flow-lines}, we get a contradiction. 
Therefore, one can find a neighborhood $U \subset V$ of $z$ such that $U \cap (\partial G \setminus \theta) = \emptyset$. 
Consequently, we have
\begin{equation}\label{eq:proof:4.19-1}
\partial G \cap U\subset \theta \cap U,
\end{equation}
and, recalling that $G$ is open (see Lemma~\ref{lem:G:open}), any connected component of $(U \cap \Omega) \setminus \theta$ is a subset of either $G$ or $\Omega \setminus \overline{G}$.

Since $z \in \partial G$, it follows that $U \cap G \neq \emptyset$, and hence there exists at least one connected component of $(U \cap \Omega) \setminus \theta$ that fully lies in $G$.
If $z \in \Omega$, then $(U \cap \Omega) \setminus \theta$ consists of two connected components (up to shrinking $U$, if necessary), each of which has $\theta \cap U$ as a part of its boundary. 
Therefore, we get $\theta \cap U \subset \partial G \cap U$.
Combining this set inclusion with \eqref{eq:proof:4.19-1}, we obtain the desired result in the case $z \in \Omega$.
If $z \in \Gamma^N$ and $U \cap \theta = U \cap \Gamma^N$, then $(U \cap \Omega) \setminus \theta$ is connected, which yields the same result. 

Now, in view of the choice of the neighborhood $V$, it remains to consider the case $U \cap \theta \cap \Gamma^N = \{z\}$. 
We deduce from Lemma~\ref{lem:morse-bott} that $(U \cap \Omega) \setminus \theta$ has two connected components, each of which has a part of $\Gamma^N$ on its boundary, and such a part contains flow lines converging to $z$. 
Since at least one connected component is a subset of $G$, we find flow lines from $\partial G \cap \Gamma^N$ converging to $z$, which contradicts the assumption of the lemma. 
\end{proof}

Finally, we provide the main result of this subsection, which has the same nature as Proposition~\ref{prop:boundary_Lr}.
\begin{proposition}\label{prop:boundary_GW}
$(\partial G \cup \partial W) \setminus \Gamma^D$ consists of finitely many flow lines (together with their end points outside of $\Gamma^D$), arcs of critical points, and isolated critical points. 
\end{proposition}
\begin{proof}
We already know from Lemmas~\ref{lem:partialGW-finite-flow-lines} and~\ref{lem:G:parG-finitely-many-points} that $\partial G \cup \partial W$ contains at most finitely many flow lines (together with their end points) and isolated points which are critical points.
Let us now describe the structure of the remaining part of $(\partial G \cup \partial W) \setminus \Gamma^D$ consisting of critical points of $u$ which are neither end points of flow lines from $\partial G \cup \partial W$ nor isolated points of $\partial G \cup \partial W$.

Let $z$ be any such critical point.
Since $z$ is nonisolated point of $\partial G \cup \partial W$, there is a sequence $\{z_n\} \subset \partial G \cup \partial W$ such that $z_n \to z$. 
If there exists a subsequence of $\{z_n\}$ consisting of regular points belonging to a flow line $\gamma$, then $\gamma \subset \partial G \cup \partial W$ by Lemmas~\ref{lem:G:boundary} and~\ref{lem:W:boundary}, and hence $\gamma$ converges to $z$, which contradicts our assumption on $z$.
On the other hand, in view of Lemma~\ref{lem:partialGW-finite-flow-lines}, only finitely many elements of $\{z_n\}$ can be regular points with mutually different flow lines.
Thus, expect of these finitely many elements, $\{z_n\}$ consists of critical points of $u$.
Then, in view of Lemma~\ref{lem:isol}, there exists a curve $\theta$ of critical points of $u$ such that $z \in \theta$.

Thanks to Lemma~\ref{lem:partialG-theta}, we can find a neighborhood $U$ of $z$ such that
$(\partial G \cup \partial W) \cap U = \theta \cap U$. 
We apply this lemma along $\theta$ until either $\theta$ ends up at $\Gamma^N$ or there appears a point $\zeta \in \theta \cap (\partial G \cup \partial W)$ such that there exists a flow line from $\partial G \cup \partial W$ converging to $\zeta$. 
Since $u$ has (at most) finitely many curves of critical points (see Lemma~\ref{lem:isol}) and flow lines in $\partial G \cup \partial W$ (see Lemma~\ref{lem:partialGW-finite-flow-lines}), we conclude that there exists finitely many arcs of critical points in $\partial G \cup \partial W$.
\end{proof}

\section{Proof of Theorem \ref{thm:main-properties}}
\label{sec:proofs-main}

We prove each assertion separately.
Since proofs of some assertions are based on others, we prove them in an order different from the order of formulation.

\begin{proof}[Proof of the assertion~\ref{thm:main-properties:opennes}]
\hypertarget{thm:main-properties:opennes:proof}{}
The sets $G$ and $W$ are open by 
Lemmas~\ref{lem:G:open} and \ref{lem:W:open}, respectively. 
Since the set of saddle points $\mathcal{S}$ is finite by Lemma~\ref{lem:numsaddlesOmega}, $\bigcup_{r \in \mathcal{S}} \partial L(r)$ is closed. 
Thus, we conclude that any Neumann domain is open, that is, it is indeed a \textit{domain} (i.e., a connected open set).
\end{proof}

\begin{proof}[Proof of the assertion~\ref{thm:main-properties:critpoints}]
\hypertarget{thm:main-properties:critpoints:proof}{}
First, we show that $\Ga^N\subset \mathcal{N}(u)$. 
If $\Ga^N=\emptyset$, it holds trivially. 
Take any $z_0 \in \Gamma^N$ and let $\{z_n\} \subset \Omega$ be a sequence of regular points of $u$ converging to $z_0$. 
Let us consider the behavior of flow lines $\gamma(\cdot,z_n)$. 
According to Lemma~\ref{lem:classif-gamma} and the definitions of the sets $G$, $W$, $L$, each $z_n$ belongs to either of these three sets.
Thus, up to a subsequence, we have three possible cases: $\{z_n\} \subset G$, $\{z_n\} \subset W$, and $\{z_n\} \subset L$.
Since $G$ and $W$ are open by Lemmas~\ref{lem:G:open} and \ref{lem:W:open}, in the first two cases we have $z_0 \in \partial G$ and $z_0 \in \partial W$, respectively.
In the third case $\{z_n\} \subset L$, we get $z_0 \in L$ since $L$ is closed. 
Noting that $z_0 \in \Gamma^N$, we get $z_0 \in \partial L$ and hence $z_0 \in \partial L(r)$ for some $r \in \mathcal{S}$, see Remark~\ref{rem:L}~\ref{rem:L:boundaryL-Lr}.
Consequently, $\Ga^N\subset \mathcal{N}(u)$.

Second, we prove that $\mathcal{C} \subset \mathcal{N}(u)$.
For any $r \in \mathcal{S}$ there is a flow line in $\partial L(r)$ converging to $r$ (see Lemma~\ref{lem:observation}), and hence $r \in \mathcal{N}(u)$.
Let us take any local extremum point $q$ of $u$. 
Since $\mathcal{M}_\pm \cap \Gamma^D = \emptyset$ by Lemma~\ref{lem:crit}~\ref{lem:crit:singsadleisolated} and we already proved that $\Ga^N\subset \mathcal{N}(u)$, we can additionally assume that $q \in \Omega$.
By Proposition~\ref{prop:absil} (if $q$ is isolated) and Lemmas~\ref{lem:isol} and \ref{lem:morse-bott} (if $q$ is nonisolated), there exists a flow line $\gamma$ converging to $q$.
If $\gamma$ connects $q$ with another local extremum point of $u$, then $\gamma \subset W$, and hence $q \in \partial W$. 
If $\gamma$ reaches $\Gamma^D$ in a finite time, then $\gamma \subset G$, and hence $q \in \partial G$.
Finally, if $\gamma$ connects $q$ with a saddle point $z_0 \in \mathcal{S}$, then $\gamma \subset L(z_0)$, and hence 
$q \in \partial L(z_0)$. 
That is, we get $\mathcal{C} \subset \mathcal{N}(u)$.
\end{proof}

\begin{proof}[Proof of the assertion~\ref{thm:main-properties:decomposition}]
\hypertarget{thm:main-properties:decomposition:proof}{}
Recall that $\mathcal{D}(u)$ stands for the union of all Neumann domains.
Since each Neumann domain is an open subset of $\Omega$ by Theorem~\ref{thm:main-properties}~\ref{thm:main-properties:opennes}, and $\mathcal{N}(u) \subset \overline{\Omega}$, we have
$$
\mathcal{D}(u) \cup (\mathcal{N}(u) \setminus \partial \Omega) \subset \Omega.
$$
To justify the reverse set inclusion, it is enough to show that if $z\in \Om \setminus \mathcal{N}(u)$, then $z\in \mathcal{D}(u)$.  
Notice that such $z$ is a regular point of $u$ in view of the inclusion $\mathcal{C} \subset \mathcal{N}(u)$, see Theorem~\ref{thm:main-properties}~\ref{thm:main-properties:critpoints}.
According to the behavior of the corresponding flow line $\gamma(\cdot,z)$ described in Lemma~\ref{lem:classif-gamma}, $z$ belongs to either of the sets $G$, $W$, $L$.
Since $z \not\in \mathcal{N}(u)$, we have $z \not\in \bigcup_{r \in \mathcal{S}} \partial L(r)$. 
On the other hand,
\begin{align*}
L&= \mathrm{Int}(L) \cup \partial L\subset \mathrm{Int}(L) \cup \bigcup_{r \in \mathcal{S}} \partial L(r),
\end{align*}
see Remark~~\ref{rem:L}~\ref{rem:L:boundaryL-Lr} for the last set inclusion.
Consequently, we get $z \in (G \cup W \cup \mathrm{Int}(L)) \setminus \bigcup_{r \in \mathcal{S}} \partial L(r)$, that is, $z \in \mathcal{D}(u)$.
\end{proof}

\begin{proof}[Proof of the assertion~\ref{thm:main-properties:finite-length}]
\hypertarget{thm:main-properties:finite-length:proof}{}
The result follows from Propositions~\ref{prop:boundary_Lr} and~\ref{prop:boundary_GW}.
\end{proof}

\begin{proof}[Proof of the assertion~\ref{thm:main-properties:finite-length2}]
\hypertarget{thm:main-properties:finite-length2:proof}{}
If $\mathcal{N}(u)$ consists only of isolated critical points, then there is nothing to prove.
Assume that $\mathcal{N}(u)$ contains at least one flow line or arc of critical points. 
Any of such curve has a finite length, as it follows from the Lojasiewicz inequality (see Section~\ref{sec:gradflow}) or Lemma~\ref{lem:isol}, respectively.
Thanks to Theorem~\ref{thm:main-properties}~\ref{thm:main-properties:finite-length}, there are at most finitely many flow lines and arcs of critical points in $\mathcal{N}(u)$, which implies that the length of $\mathcal{N}(u)$ is finite.
By construction, we have $\partial u/\partial \nu = 0$ along any flow line and arc of critical points.
\end{proof}

\begin{proof}[Proof of the assertion~\ref{thm:main-properties:finite-number}]
\hypertarget{thm:main-properties:finite-number:proof}{}
By Theorem~\ref{thm:main-properties}~\ref{thm:main-properties:finite-length}, $\mathcal{N}(u)$ contains of at most finitely many flow lines (together with their end points) and arcs of critical points. 
Since $u$ is strictly monotone along flow lines, we have the following relation between flow lines and arcs of critical points: any two flow lines can meet each other at most at two points, a flow line can meet an arc of critical points at most at one point, and if two arcs of critical points meet, then they belong to one curve of critical points (see Lemma~\ref{lem:isol}).     
Therefore, $\Omega \setminus \mathcal{N}(u)$ has a finite number of connected components.
In view of the decomposition given by Theorem~\ref{thm:main-properties}~\ref{thm:main-properties:decomposition}, we conclude that the number of Neumann domains is finite.
\end{proof}

\begin{proof}[Proof of the assertion~\ref{thm:main-properties:sign-changing}]
\hypertarget{thm:main-properties:sign-changing:proof}{}
Let $\Omega_0$ be an inner Neumann domain. 
Consider a domain $\widetilde{\Omega}_0 = \mathrm{Int}(\overline{\Omega_0})$.
The domain $\widetilde{\Omega}_0$ is more regular than $\Omega_0$ since it does not contain possible ``cracks''. 
Indeed, if there exists $z \in \partial \Omega_0$ such that a neighborhood $U$ of $z$ satisfies $U \cap (\mathbb{R}^N \setminus \overline{\Omega_0}) = \emptyset$, then $U \subset \mathrm{Int}\,\overline{\Omega_0}$ and hence $z \not\in \partial \widetilde{\Omega}_0$.
Apart from such ``cracks'', the boundaries $\partial \widetilde{\Omega}_0$ and $\partial \Omega_0$ coincide.
Consequently, $\partial \widetilde{\Omega}_0$ is piecewise analytic and consists of finitely many pieces (see Theorem~\ref{thm:main-properties}~\ref{thm:main-properties:finite-length}).

Since $u$ is analytic in $\overline{\Omega}$, its restriction to the closure of $\widetilde{\Omega}_0$ is also analytic.
Moreover, $\partial u/\partial \nu = 0$ on each analytic piece of $\partial \widetilde{\Omega}_0$ by Theorem~\ref{thm:main-properties}~\ref{thm:main-properties:finite-length2}.
Integrating the equation in \eqref{cutproblem2d} over $\widetilde{\Omega}_0$ and applying the divergence theorem (see, e.g., \cite[Theorem~9.6]{maggi2012sets}), we get
$$
\lambda
\int_{\widetilde{\Omega}_0} u \,dx
=
-
\int_{\widetilde{\Omega}_0} \Delta u \,dx
=
-
\int_{\partial \widetilde{\Omega}_0}
\frac{\partial u}{\partial \nu} \,dS
=
0.
$$
Recalling that $\lambda >0$ (see Section~\ref{sec:properties}), we conclude that $u$ is sign-changing in $\widetilde{\Omega}_0$ and hence in $\Omega_0$ by the continuity.
\end{proof}

\begin{proof}[Proof of the assertion~\ref{thm:main-properties:sign-const}]
\hypertarget{thm:main-properties:sign-const:proof}{}
Let $\Omega_0$ be a boundary Neumann domain, i.e, a connected component of $G \setminus \bigcup_{r \in \mathcal{S}} \partial L(r)$. 
Let us take any $z_1,z_2 \in \Omega_0$ and suppose that $u(z_1) \leq 0 \leq u(z_2)$. 
By the continuity, there exists $z_3 \in \Omega_0$ such that $u(z_3)=0$.
According to the definition of $G$, the flow line $\gamma(\cdot,z_3)$ reaches $\Gamma^D$ in a finite time.
However, this is impossible since $u(z_3)=0$ and $u$ is strictly monotone along flow lines.
Thus, $u$ restricted to $\Omega_0$ is either strictly positive or strictly negative.
\end{proof}

\section{Appearance of isolated semi-degenerate critical points}\label{sec:trivialcrit}

In this section, we provide a numerically based evidence that every type of isolated semi-degenerate critical points can appear in the critical set of the first (positive) Dirichlet eigenfunction $u$ on a dumbbell shaped domain.
We do not pursue the rigor of arguments and only aim to demonstrate that the occurrence of such critical points is natural to expect and is not extraordinary. 
In fact, a rigorous proof of existence of a saddle-node in the critical set of a higher Dirichlet eigenfunction in a square is given in \cite{chenmyrtaj2019}.
The construction of our examples is substantially different.

Recall that at an isolated semi-degenerate critical point $z_0$ the first (positive) eigenfunction $u$ is right-equivalent to one of the functions $(x,y) \mapsto u(z_0) \pm x^k \pm y^2$ with $k \geq 3$ in a neighborhood of $(0,0)$, see Section~\ref{sec:classification}.
If $k$ is odd, then $z_0$ is called saddle-node (equivalently, trivial point, see Remark~\ref{rem:terminology}), and if $k$ is even, then $z_0$ is either a local maximum point or a topological saddle.
Moreover, we note that since $u$ is positive, it does not have fully degenerate critical points, see Section~\ref{sec:classification}.

Let us construct a two-parametric family of domains $\Omega(R,s)$ as follows.
Let $B_1(0,0)$ be the unit disk centered at the origin, and let $B_R(2+R,0)$ be the disk of radius $R>0$ centered at the point $(2+R,0)$, so that the distance between $B_1(0,0)$ and $B_R(2+R,0)$ is always $1$.
Let $T(R) \subset (-1,2+R) \times (-R,R)$ be an isosceles trapezoid (considered as an open set) such that 
$$
B_1(0,0) \cup B_R(2+R,0) \cup T(R) = \text{convex hull of}~ B_1(0,0) \cup B_R(2+R,0),
$$ 
see Figure~\ref{fig:saddle-node1}.
Denote the lower-right, upper-right, upper-left, and lower-left vertices of $T(R)$ as 
$$
(x_{lr},y_{lr}),~~ (x_{ur},y_{ur}),~~ (x_{ul},y_{ul}),~~ (x_{ll},y_{ll}),
$$
respectively. 
Consider now an isosceles trapezoid $T(R,s)$ spanned on the points
$$
(x_{lr},y_{lr}+s),~~ (x_{ur},y_{ur}-s),~~ (x_{ul},y_{ul}-s),~~ (x_{ll},y_{ll}+s),
$$
where $s \geq 0$ is such that $T(R,s)$ remains a trapezoid. 
That is, $T(R,s)$ shrinks $T(R)$ along the $y$-coordinate, see Figure~\ref{fig:saddle-node2}.
Finally, denote
$$
\Omega(R,s)
=
B_1(0,0) \cup B_R(2+R,0) \cup T(R,s).
$$

\begin{figure}[!ht]
  \begin{subfigure}{0.49\textwidth}
    \includegraphics[width=\linewidth]{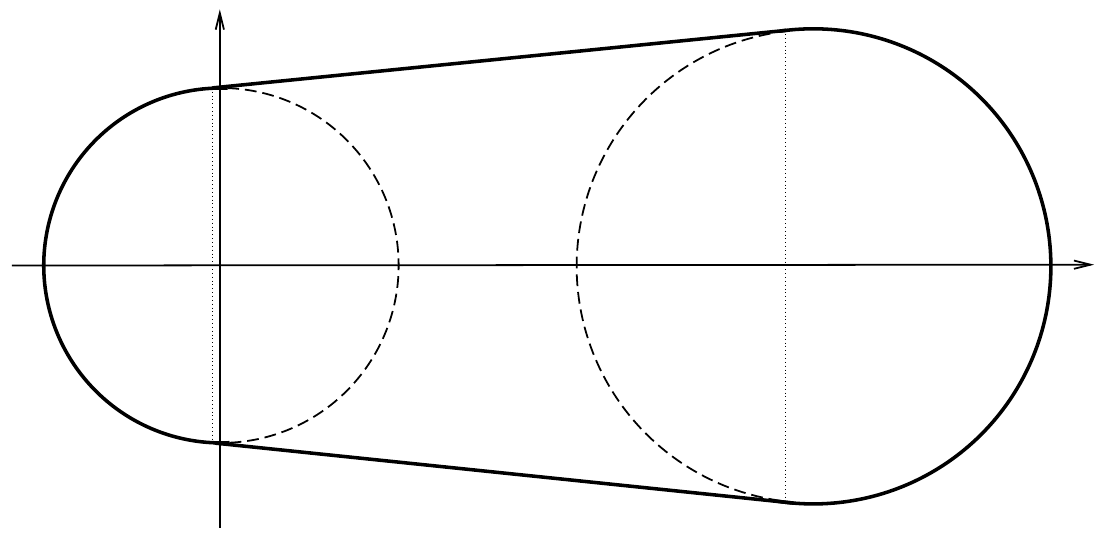}
    \caption{$R>1$, $s=0$} \label{fig:saddle-node1}
  \end{subfigure}%
  \hspace*{\fill} 
  \begin{subfigure}{0.49\textwidth}
    \includegraphics[width=\linewidth]{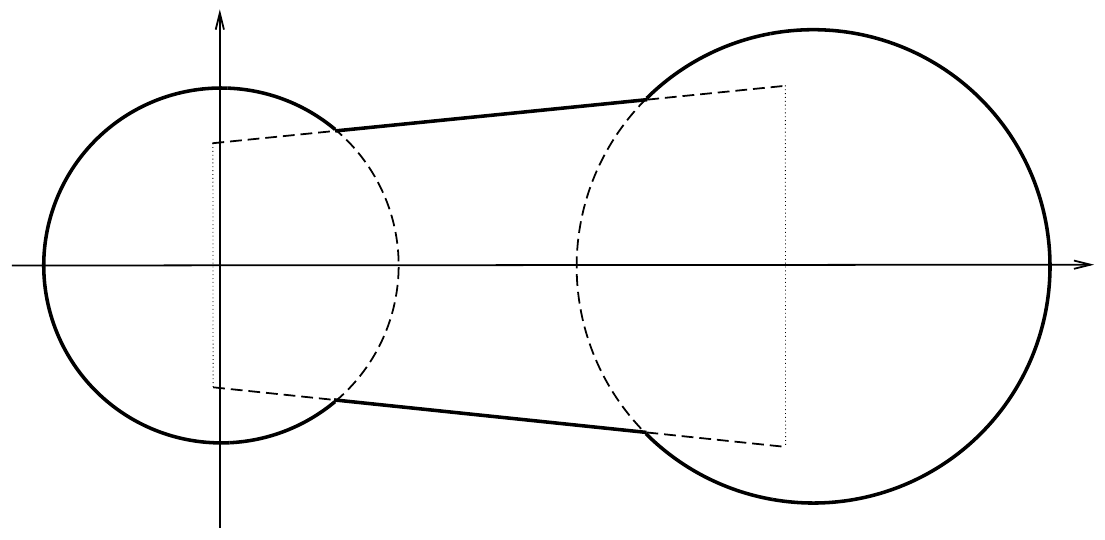}
    \caption{$R>1$, $s > 0$} \label{fig:saddle-node2}
  \end{subfigure}%
    \\[1em]
   \begin{subfigure}{0.49\textwidth}
    \includegraphics[width=\linewidth]{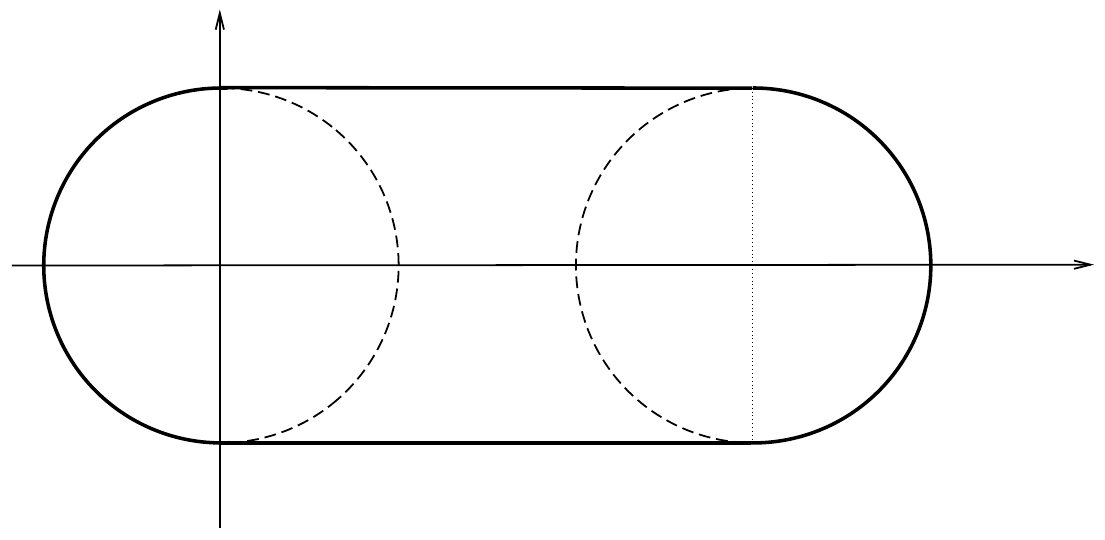}
    \caption{$R=1$, $s=0$} \label{fig:saddle-node3}
  \end{subfigure}%
  \hspace*{\fill} 
  \begin{subfigure}{0.49\textwidth}
    \includegraphics[width=\linewidth]{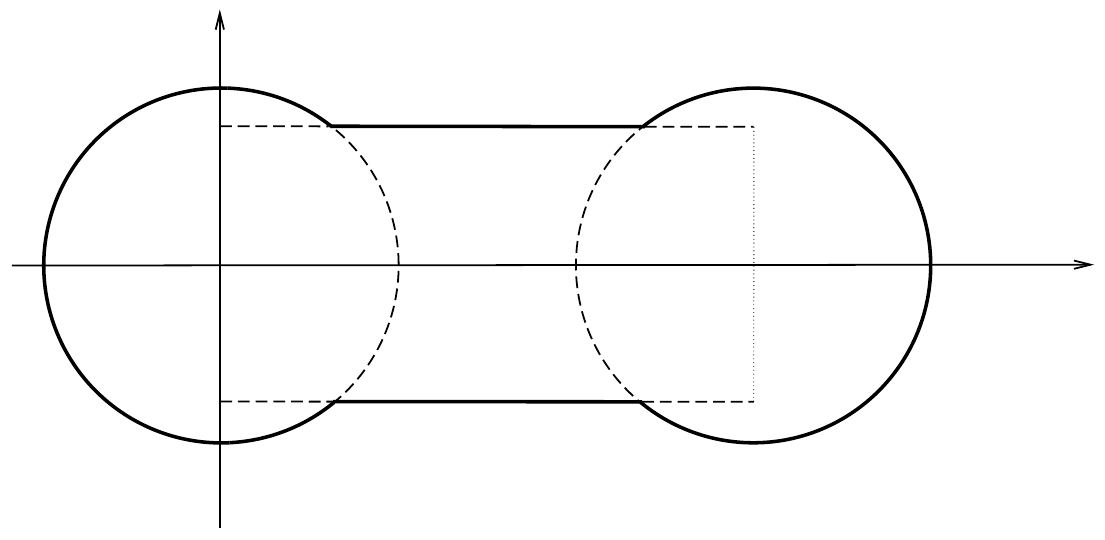}
    \caption{$R=1$, $s>0$} \label{fig:saddle-node4}
  \end{subfigure}%
    
\caption{$\Omega(R,s)$ with different values of $R$ and $s$.}
\end{figure}

\noindent
By construction, $\Omega(R,s)$ is symmetric with respect to the $x$-axis and convex in the $y$-direction, i.e., $\Omega(R,s)$ is Steiner symmetric with respect to the $x$-axis.
We see that if $R=1$ and $s=0$, then $\Omega(1,0)$ is a stadium domain, and it is convex, see Figure~\ref{fig:saddle-node3}.
If $R=1$ and $s>0$, then $\Omega(1,s)$ is a dumbbell domain symmetric with respect to two orthogonal lines, 
see Figure~\ref{fig:saddle-node4}.

For convenience, we denote the first (positive) Dirichlet eigenfunction in $\Omega(R,s)$ as $u_{R,s}$.
Thanks to the uniqueness of the first eigenfunction modulo scaling, $u_{R,s}$ shares the symmetries of $\Omega(R,s)$.
In particular, $u_{R,s}$ is always symmetric with respect to the $x$-axis. 
If $R=1$, then $u_{1,s}$ is also symmetric with respect to the line $\{x=3/2\}$, and $u_{1,s}$ is centrally symmetric with respect to the point $(3/2,0)$.

Let us now consider properties of critical points of $u_{R,s}$.
In view of the Steiner symmetry, the moving plane method applied in the $y$-direction (see, e.g., \cite[Theorem~1.3]{berestycki1991method}) shows that
$(u_{R,s})'_y<0$ for any $y \neq 0$.
In particular, \textit{all} critical points of $u_{R,s}$ are located on the $x$-axis.
We therefore consider the function $f_{R,s}$ defined as $f_{R,s}(x)=u_{R,s}(x,0)$ for $x \in (-1,2+2R)$.
Since $u_{R,s}$ is analytic, so is $f_{R,s}$, and hence $f_{R,s}$ has only isolated critical points whose number is finite. 
Moreover, the monotonicity of $u_{R,s}$ in the $y$-direction yields the following relation between critical points of $f_{R,s}$ and $u_{R,s}$:

\begin{enumerate}
	\item $x_0$ is a local maximum point of $f_{R,s}$ if and only if $(x_0,0)$ is a local maximum point of $u_{R,s}$. 
	In addition, $x_0$ is degenerate (i.e., $f_{R,s}''(x_0) = 0$) if and only if 
	$(x_0,0)$ is semi-degenerate (i.e., $(x_0,0)$ is a semi-degenerate local maximum point).
	
	\item $x_0$ is a local minimum point of $f_{R,s}$ if and only if $(x_0,0)$ is a topological saddle of $u_{R,s}$. 
	In addition, $x_0$ is degenerate if and only if $(x_0,0)$ is semi-degenerate (i.e., $(x_0,0)$ is a semi-degenerate topological saddle).
	
	\item $x_0$ is a saddle point (i.e., inflection point) of $f_{R,s}$ if and only if $(x_0,0)$ is a saddle-node of $u_{R,s}$.	
\end{enumerate}

We also comment on the behavior of critical points of $f_{R,s}$ (and $u_{R,s}$) with respect to the parameters $R$, $s$. 
Namely, a critical point either moves continuously, or collides with another critical points, or splits into several critical points, or disappears without interrelation with other critical points.
Since non-degenerate critical points are \textit{structurally stable} (i.e., their existence and Morse indices are preserved under $C^1$-small perturbations, see, e.g., \cite[Section~4.1]{perko2013differential}) and $u_{R,s}$ is expected to depend smoothly on $R$ and $s$,  
the collision/splitting and appearance/disappearance can happen only with semi-degenerate critical points. 
(Recall that $u_{R,s}$ has no fully degenerate critical points since $u_{R,s}>0$, while fully degenerate critical points can appear only in the nodal set of $u_{R,s}$.)

Relying on these facts, we proceed with discussing the appearance of semi-degenerate critical points.
Note that if $s=0$, then $\Omega(R,0)$ is a convex domain (see Figure~\ref{fig:saddle-node1}), and hence $u_{R,0}$ has \textit{exactly one} critical point which is a global maximum critical point (see, e.g., \cite{de2021uniqueness}).
On the other hand, if $R=1$ and $s \to 1$, then the ``handle'' of the symmetric dumbbell domain $\Omega(1,s)$ is narrowing (see Figure~\ref{fig:saddle-node4}) and $u_{1,s}$ converges in an appropriate sense 
to the union of equally normalized first eigenfunctions on the disks $B_1(0,0)$ and $B_1(3,0)$. 
Therefore, we conclude that the number of critical points of $u_{1,s}$ is \textit{at least three} for any sufficiently large $s \in (0,1)$, i.e., two local maximum points and one saddle.
In view of the dependence of critical points on the parameters discussed above, we conclude that there exists the smallest $s_0 \in [0,1)$ such that the unique global maximum critical point of $u_{1,s_0}$ is semi-degenerate and splits into several critical points for $s \in (s_0, 1)$, see Figure~\ref{fig:saddle-node-counter0}.
That is, such a critical point has the same type as $(0,0)$ of the function $(x,y) \mapsto u(z_0) - x^k - y^2$ with even $k \geq 4$.

\begin{figure}[!ht]
  \begin{subfigure}{0.33\textwidth}
    \includegraphics[width=\linewidth]{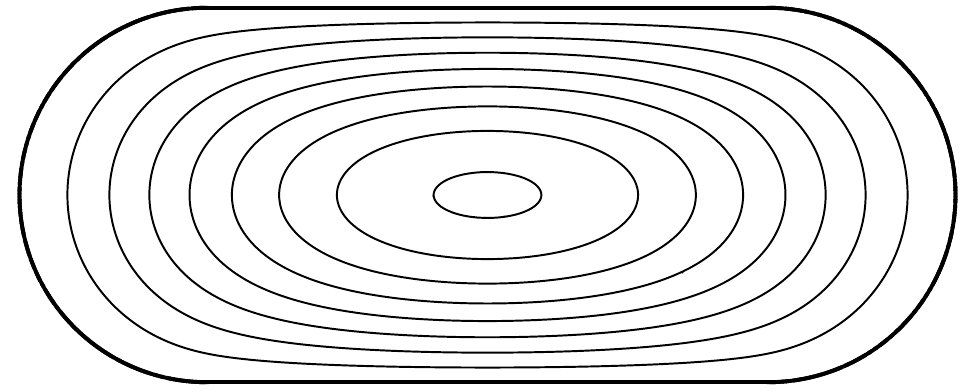}
  \end{subfigure}%
  \hspace*{\fill} 
  \begin{subfigure}{0.33\textwidth}
    \includegraphics[width=\linewidth]{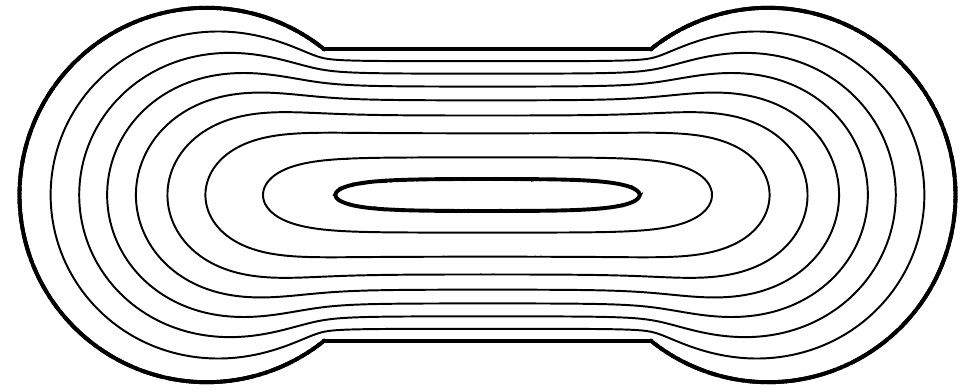}
  \end{subfigure}%
  \hspace*{\fill}  
  \begin{subfigure}{0.33\textwidth}
    \includegraphics[width=\linewidth]{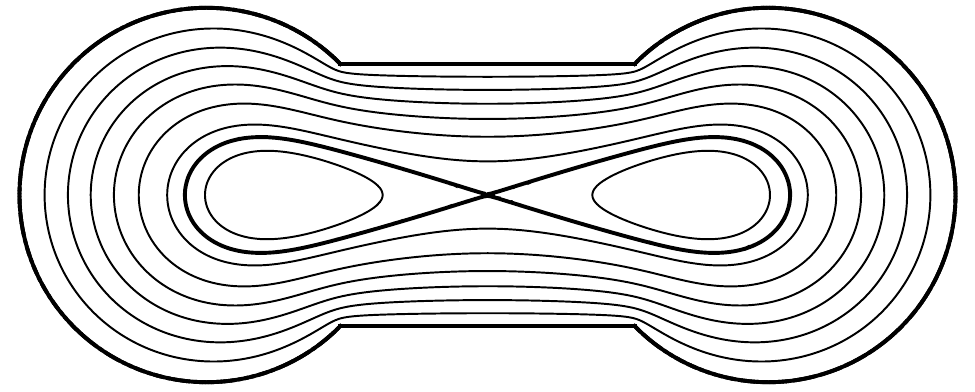}
  \end{subfigure}%
    \\[1em]
   \begin{subfigure}{0.33\textwidth}
    \includegraphics[width=\linewidth]{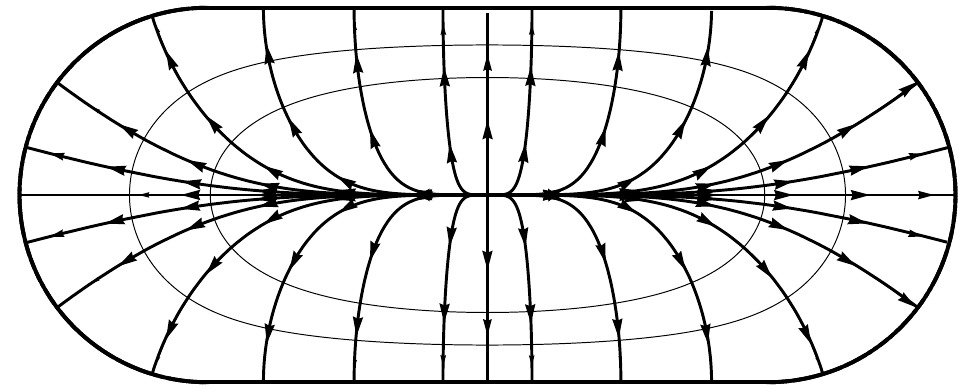}
  \end{subfigure}%
  \hspace*{\fill} 
  \begin{subfigure}{0.33\textwidth}
    \includegraphics[width=\linewidth]{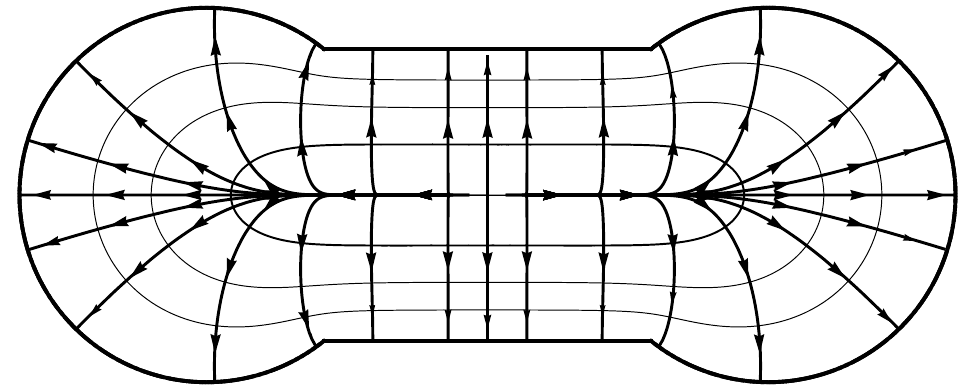}
  \end{subfigure}%
  \hspace*{\fill}  
  \begin{subfigure}{0.33\textwidth}
    \includegraphics[width=\linewidth]{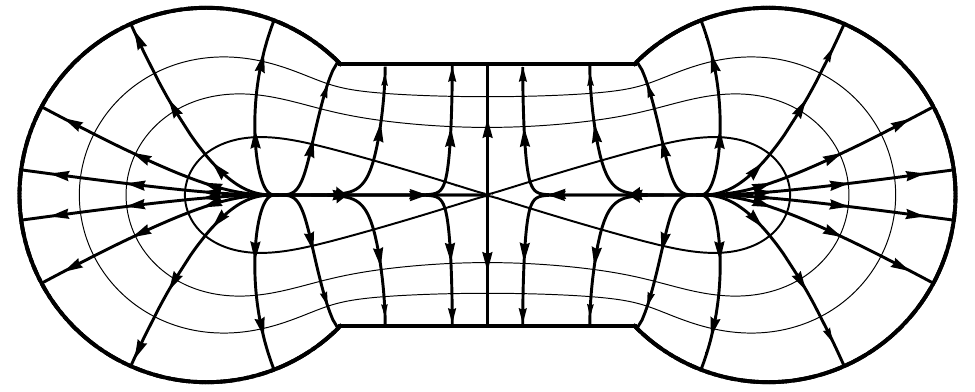}
  \end{subfigure}%
    \\[1em]
   \begin{subfigure}{0.33\textwidth}
    \includegraphics[width=\linewidth]{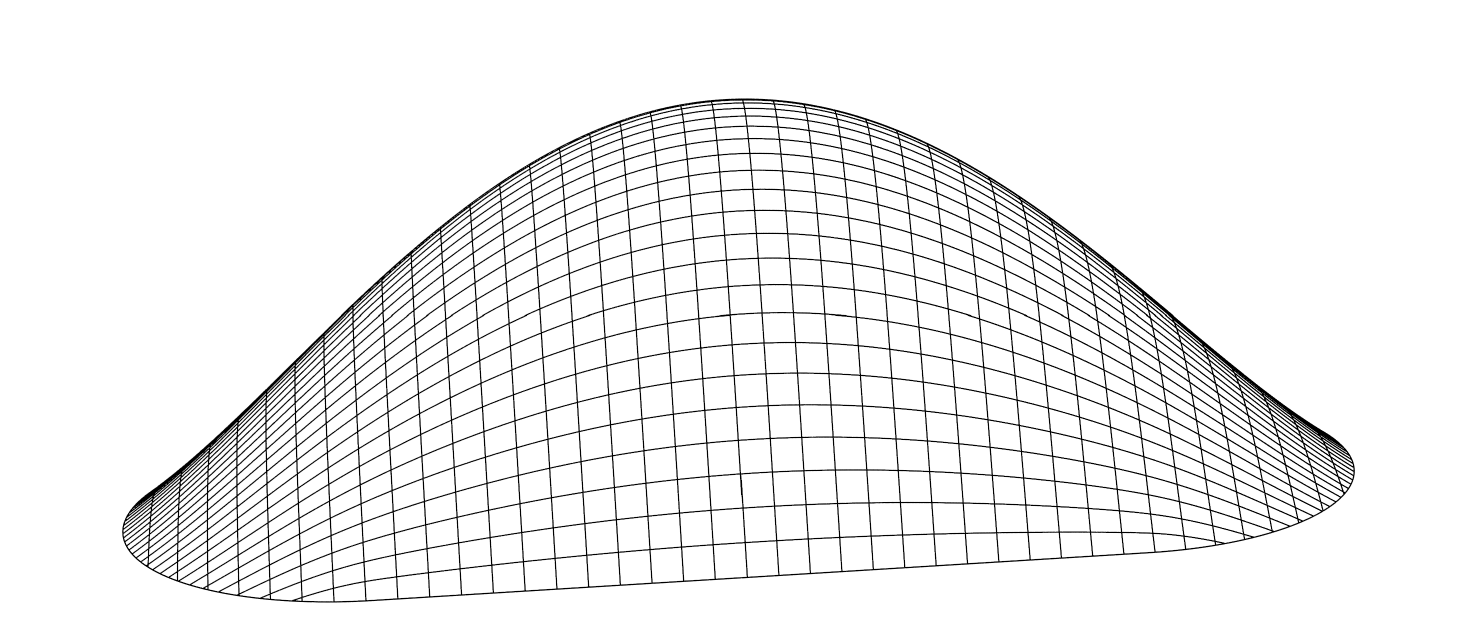}
    \caption{$s=0$}
  \end{subfigure}%
  \hspace*{\fill}   
  \begin{subfigure}{0.33\textwidth}
    \includegraphics[width=\linewidth]{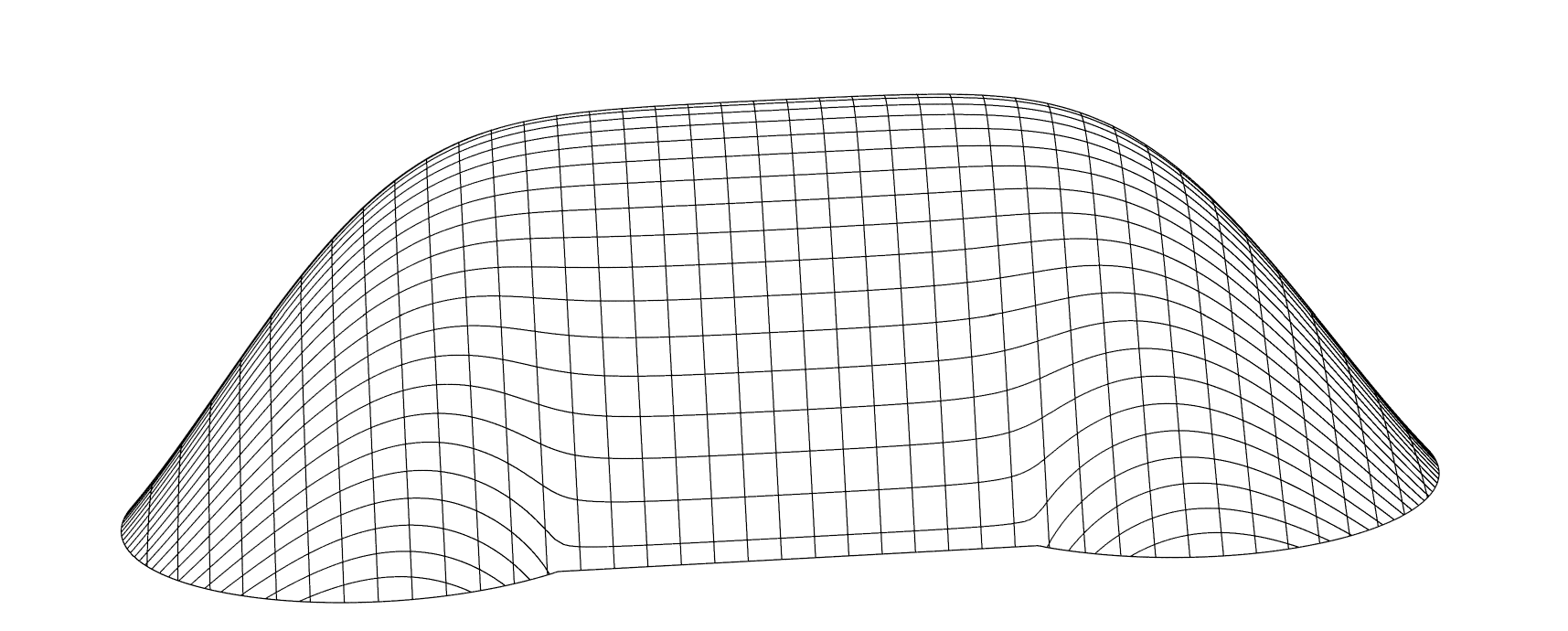}
    \caption{$s=0.22$}%
  \end{subfigure}%
  \hspace*{\fill}  
  \begin{subfigure}{0.33\textwidth}
    \includegraphics[width=\linewidth]{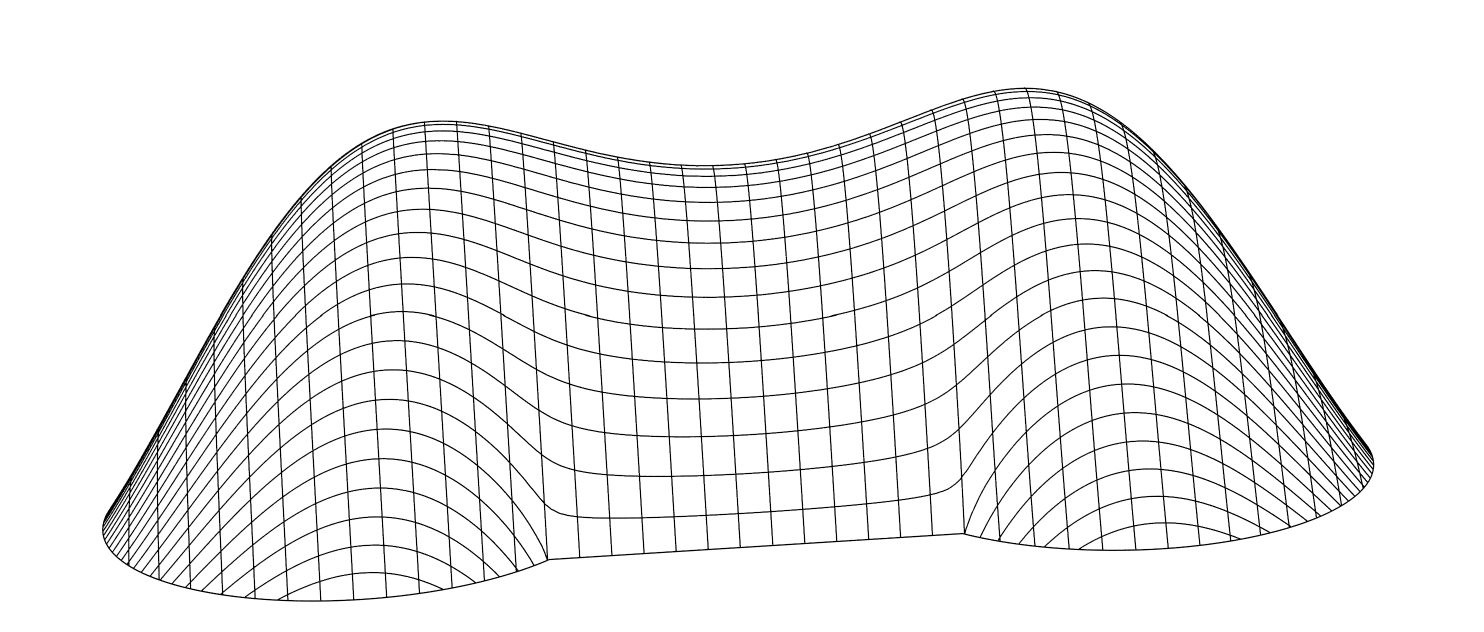}
    \caption{$s=0.3$}
    \label{fig:saddle-node6-5c}
  \end{subfigure}%
\caption{$\Omega(R,s)$ with $R=1$ and varying $s$. Level sets, flow lines, and perspective view of $u_{R,s}$.}
\label{fig:saddle-node-counter0}
\end{figure}

We notice that, in essence, all the above arguments can be made  mathematically rigorous with a little effort. 
Numerically-based conclusions which are harder to establish rigorously come in the following paragraphs.

Let us now discuss the appearance of the saddle-node.
This will come from the development of the previous case with $R \neq 1$, that is, during the transition from Figure~\ref{fig:saddle-node2} to Figure~\ref{fig:saddle-node1}.
By the continuous dependence of $u_{R,s}$ with respect to $R$ and $s$ and the properties of $u_{1,s}$ (i.e., with $R=1$), $u_{R,s}$ still has \textit{at least three} critical points (two local maximum points and one saddle) for a sufficiently large $s \in (0,1)$, provided $R \neq 1$ sufficiently close to $1$. 
For simplicity, assume that there are exactly three critical points.
Fixing any such $R$ and sending $s$ to $0$, we claim that the local (not global) maximum point collides with the saddle, thereby forming a saddle-node, and after collision the saddle-node disappears, leaving only one critical point - the global maximum point.
We do not prove this claim rigorously, but numerical experiments indicate exactly this behavior, see Figure~\ref{fig:saddle-node-counter} for some plots obtained using build-in methods of \textit{Mathematica}.

\begin{figure}[!ht]
  \begin{subfigure}{0.33\textwidth}
    \includegraphics[width=\linewidth]{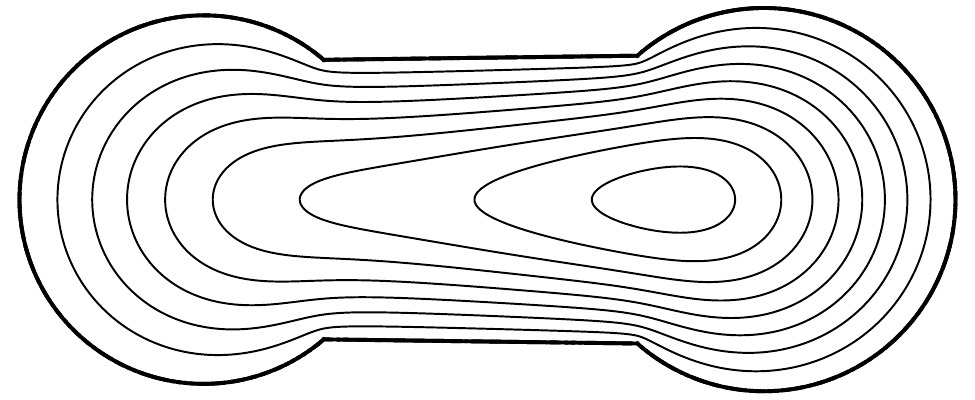}
  \end{subfigure}%
  \hspace*{\fill} 
  \begin{subfigure}{0.33\textwidth}
    \includegraphics[width=\linewidth]{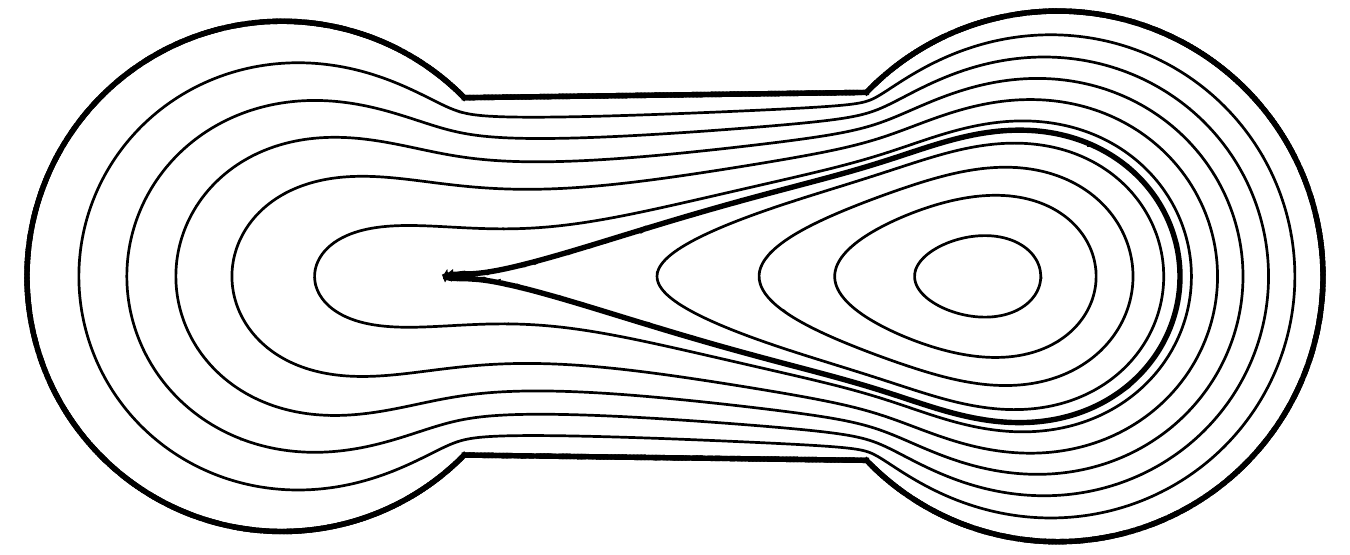}
  \end{subfigure}%
  \hspace*{\fill}  
  \begin{subfigure}{0.33\textwidth}
    \includegraphics[width=\linewidth]{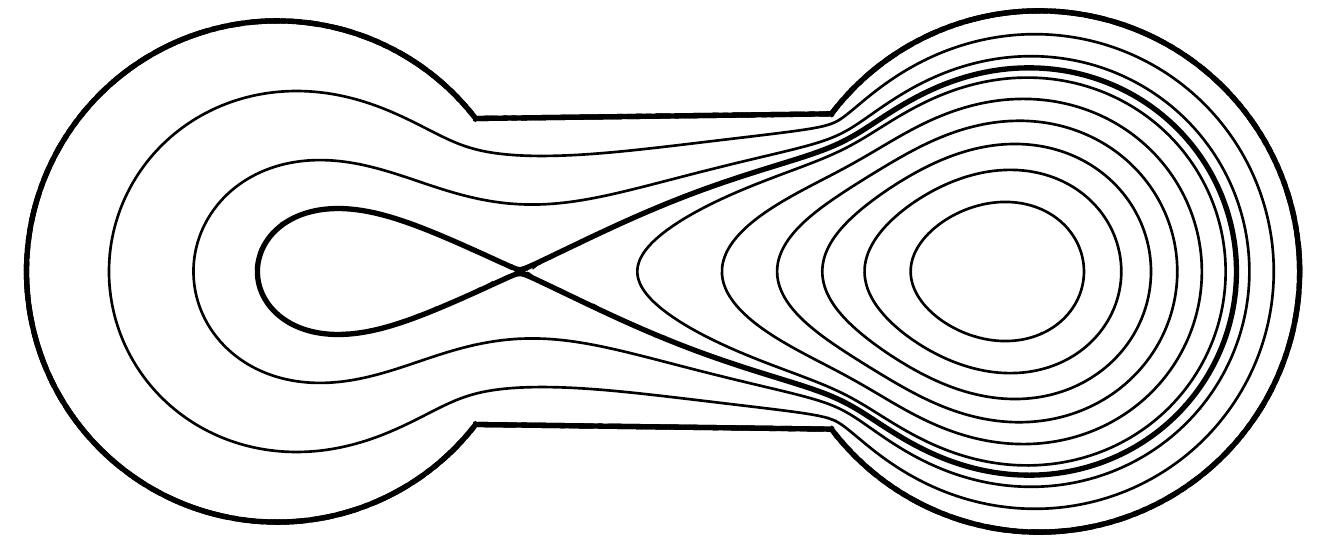} 
  \end{subfigure}%
    \\[1em]
   \begin{subfigure}{0.33\textwidth}
    \includegraphics[width=\linewidth]{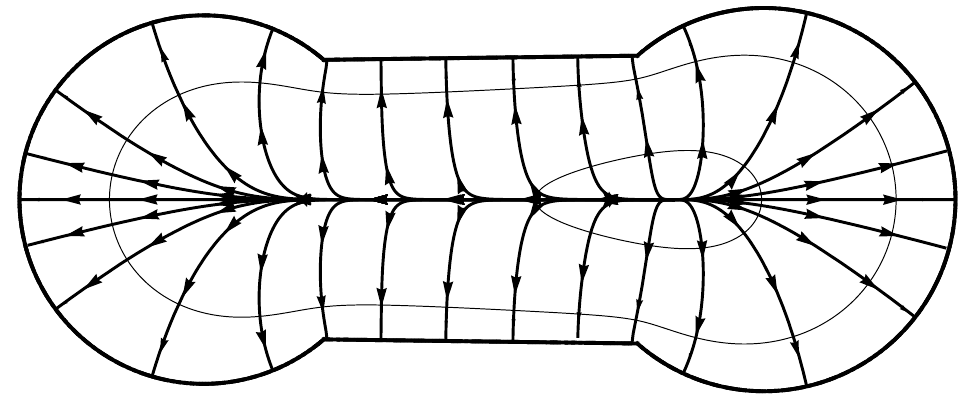} 
  \end{subfigure}%
  \hspace*{\fill}  
  \begin{subfigure}{0.33\textwidth}
    \includegraphics[width=\linewidth]{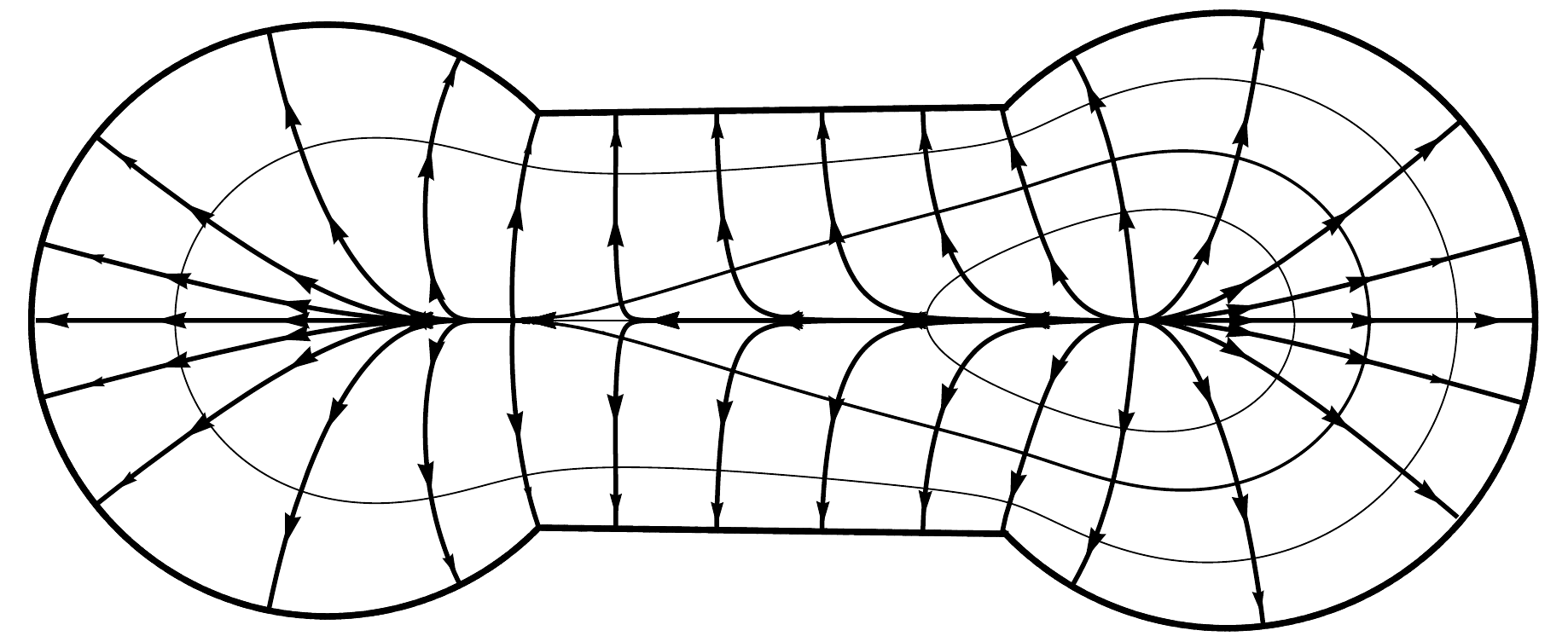}
  \end{subfigure}%
  \hspace*{\fill}  
  \begin{subfigure}{0.33\textwidth}
    \includegraphics[width=\linewidth]{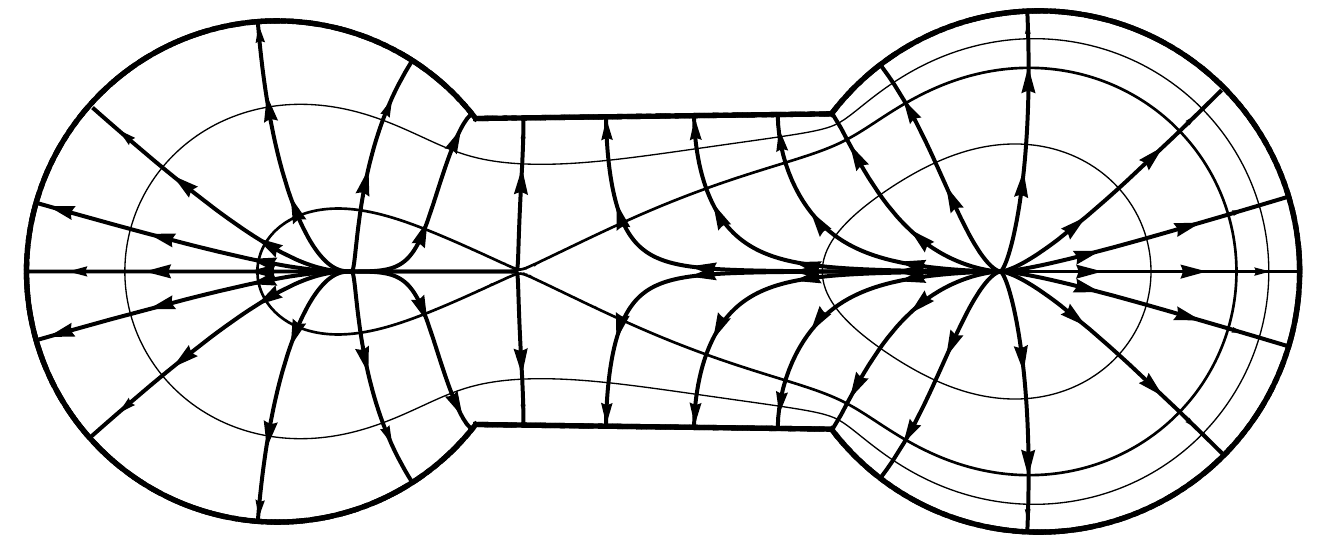} 
  \end{subfigure}%
    \\[1em]
   \begin{subfigure}{0.33\textwidth}
    \includegraphics[width=\linewidth]{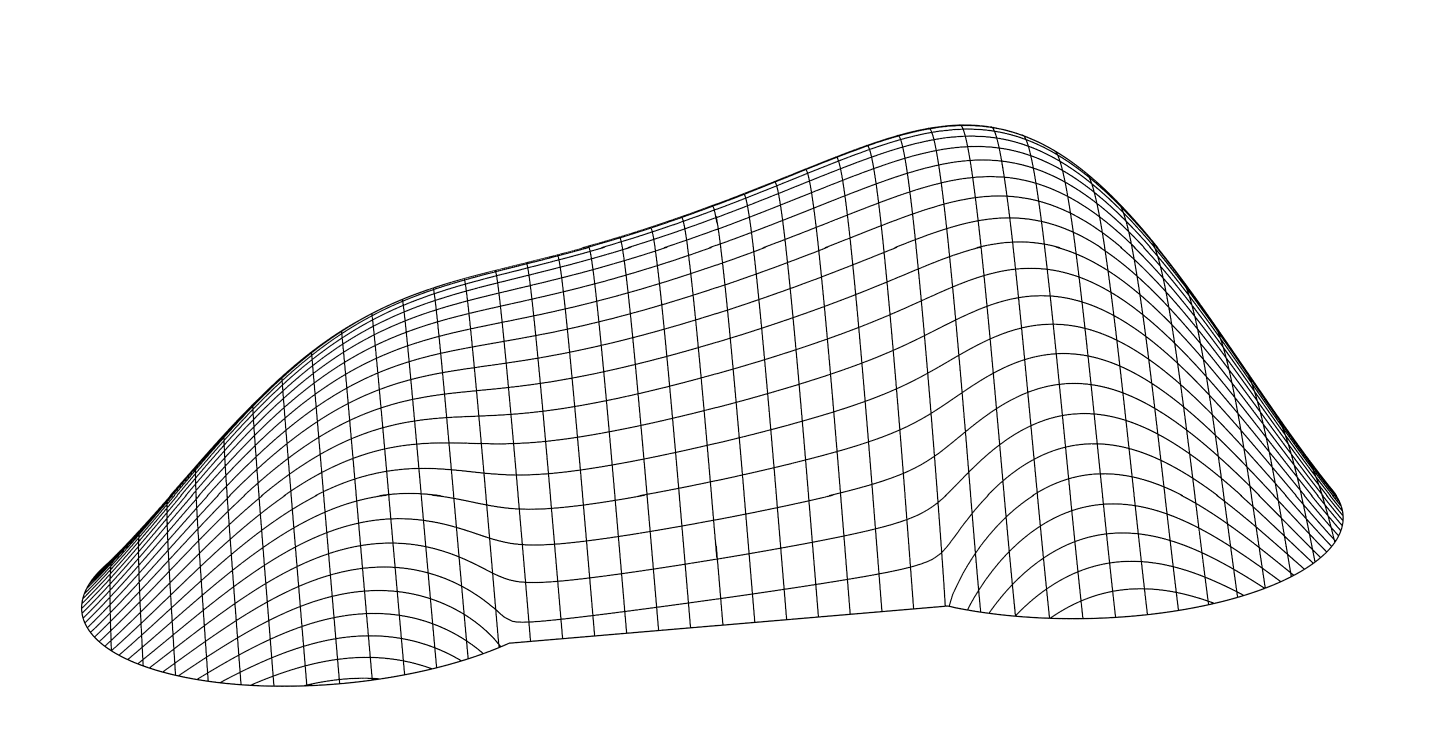}
    \caption{$s=0.25$}
  \end{subfigure}%
  \hspace*{\fill}  
  \begin{subfigure}{0.33\textwidth}
    \includegraphics[width=\linewidth]{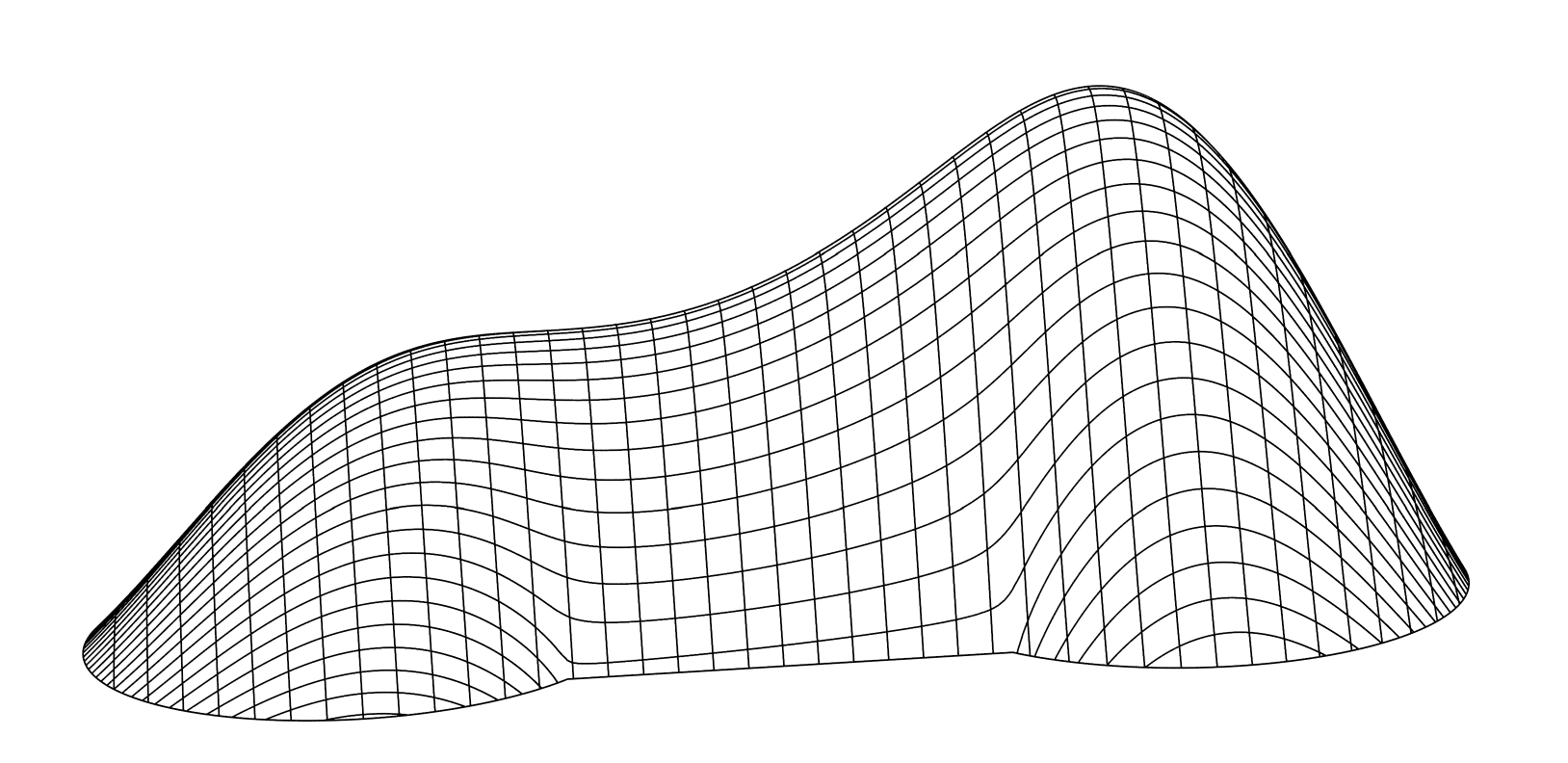}
    \caption{$s=0.309$}
  \end{subfigure}%
  \hspace*{\fill}  
  \begin{subfigure}{0.33\textwidth}
    \includegraphics[width=\linewidth]{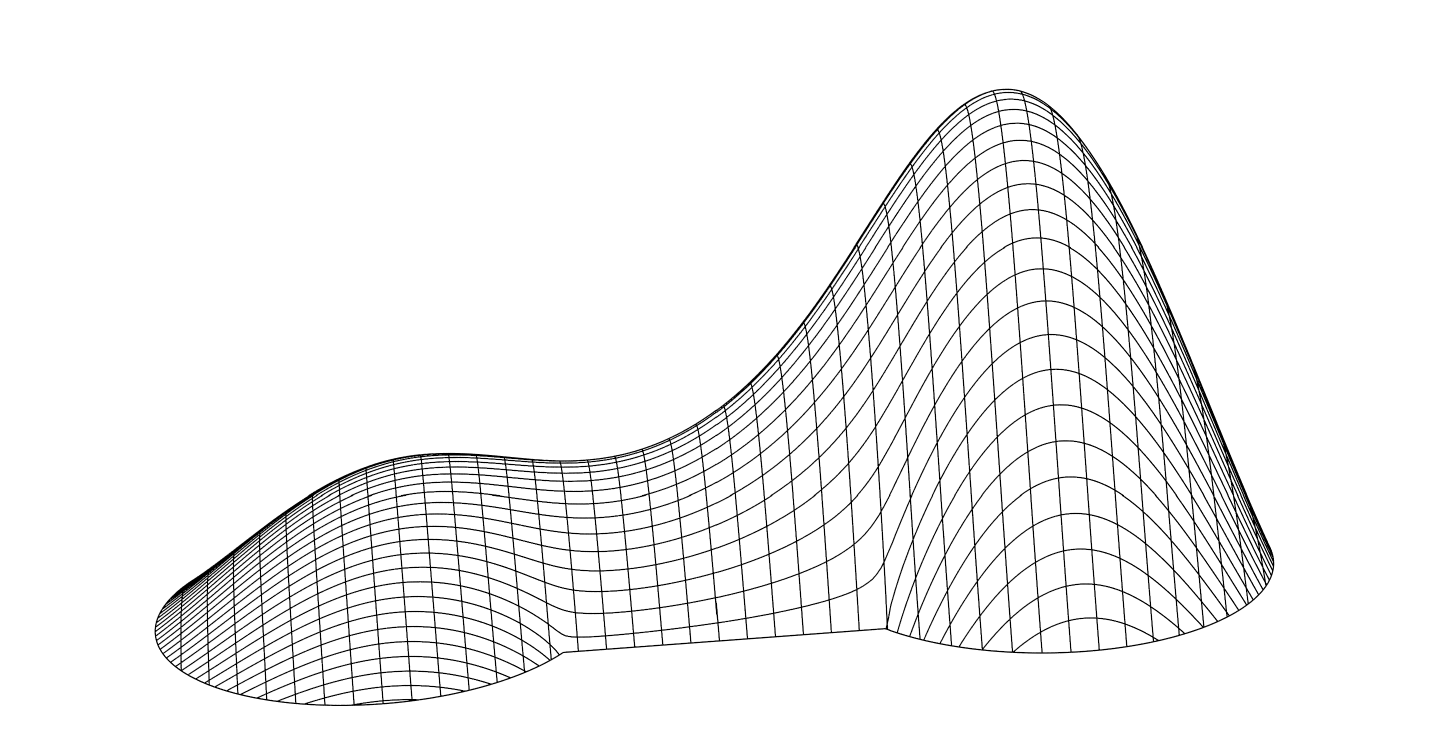}
    \caption{$s=0.4$}
  \end{subfigure}%
\caption{$\Omega(R,s)$ with $R=1.04$ and varying $s$. Level sets, flow lines, and perspective view of $u_{R,s}$.}\label{fig:saddle-node-counter}
\end{figure}

The appearance of the remaining semi-degenerate point $z_0$, namely, the topological saddle point as $(0,0)$ of the function $(x,y) \mapsto u(z_0) + x^k - y^2$ with even $k \geq 4$, is indicated using similar facts as above for a dumbbell-type domain $\Omega(R,s,r)$ defined as
$$
\Omega(1,s,r) = \Omega(1,s) \cup B_r(3/2,0),
$$
where $B_r(3/2,0)$ is the disk of radius $r>0$ centered at the point $(3/2,0)$, which is the center of symmetry of the domain $\Omega(1,s)$; see Figure~\ref{fig:dumbbel3}.
Clearly, if $r \leq 1-s$, then $\Omega(1,s,r) = \Omega(1,s)$, and hence the corresponding first eigenfunction has \textit{at least three} critical points (two local maximum points and one saddle), see Figure~\ref{fig:dumbbel3c}.
On the other hand, for some $r \in (1-s, 1]$, it is natural to expect that 
the eigenfunction has \textit{at least five} critical points (three local maximum  points and two saddles), see Figure~\ref{fig:dumbbel3a}. 
(This fact is even more clear if we initially stretch $\Omega(1,s)$ in the $x$-direction so that the ``handle'' becomes sufficiently long.)
Numerical results indicate that the semi-degenerate topological saddle $z_0$ is obtained for an intermediate value of $r$ as the joining of three critical points - one local maximum point (located at $(3/2,0)$) and two saddles, see Figure~\ref{fig:dumbbel3b}.

\begin{figure}[!ht]
  \begin{subfigure}{0.33\textwidth}
    \includegraphics[width=\linewidth]{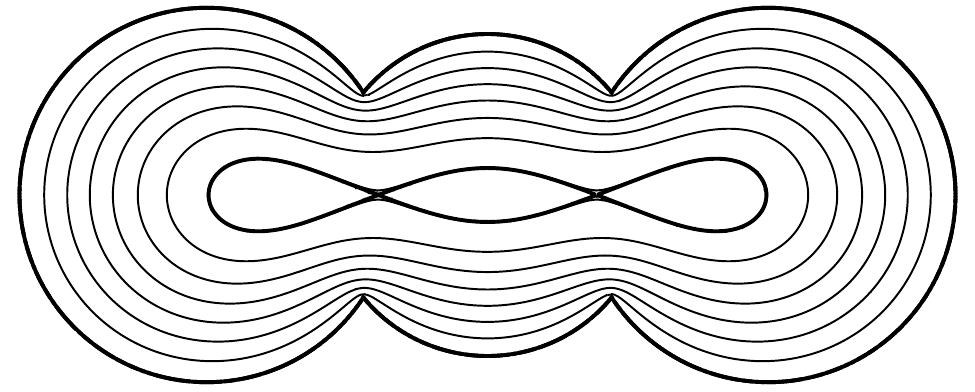}
  \end{subfigure}%
  \hspace*{\fill}  
  \begin{subfigure}{0.33\textwidth}
    \includegraphics[width=\linewidth]{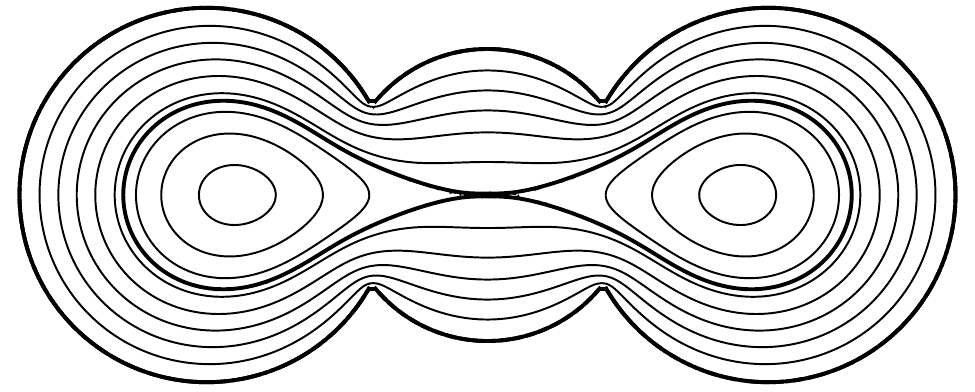}
 \end{subfigure}%
  \hspace*{\fill}  
  \begin{subfigure}{0.33\textwidth}
    \includegraphics[width=\linewidth]{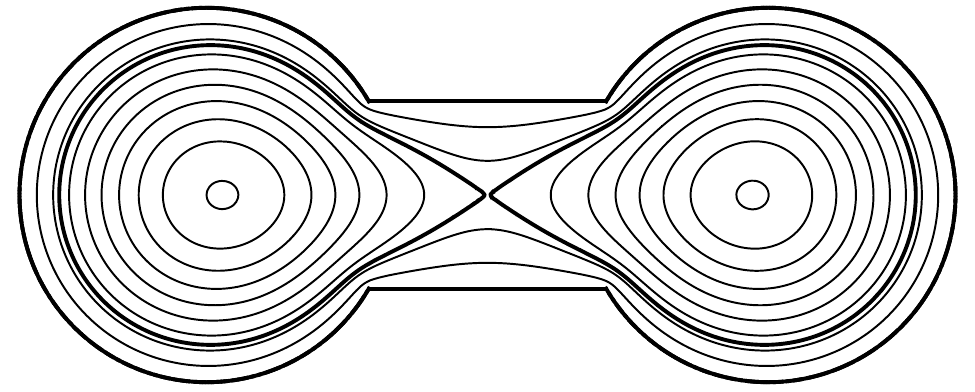} 
  \end{subfigure}%
    \\[1em]
   \begin{subfigure}{0.33\textwidth}
    \includegraphics[width=\linewidth]{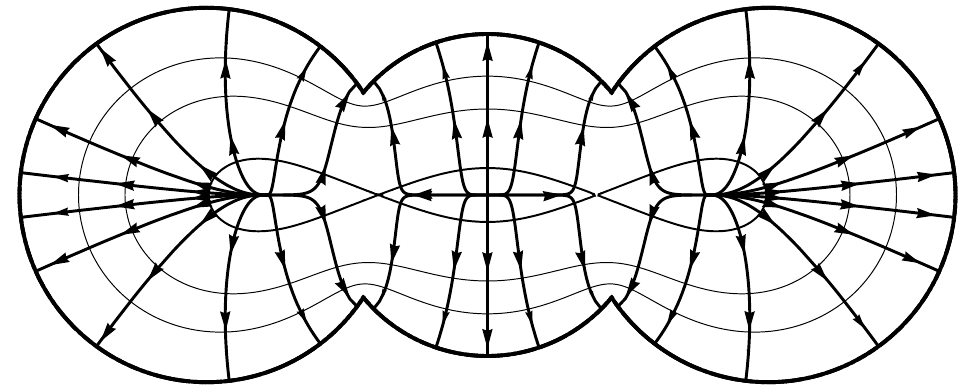} 
  \end{subfigure}%
  \hspace*{\fill} 
  \begin{subfigure}{0.33\textwidth}
    \includegraphics[width=\linewidth]{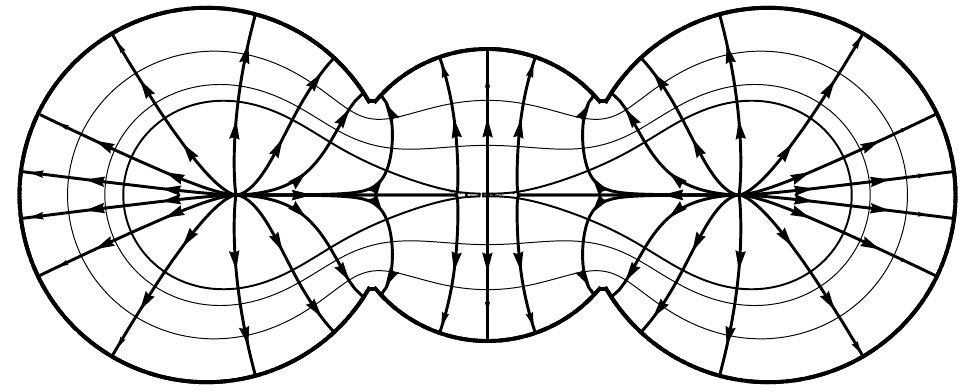}
  \end{subfigure}%
  \hspace*{\fill} 
  \begin{subfigure}{0.33\textwidth}
    \includegraphics[width=\linewidth]{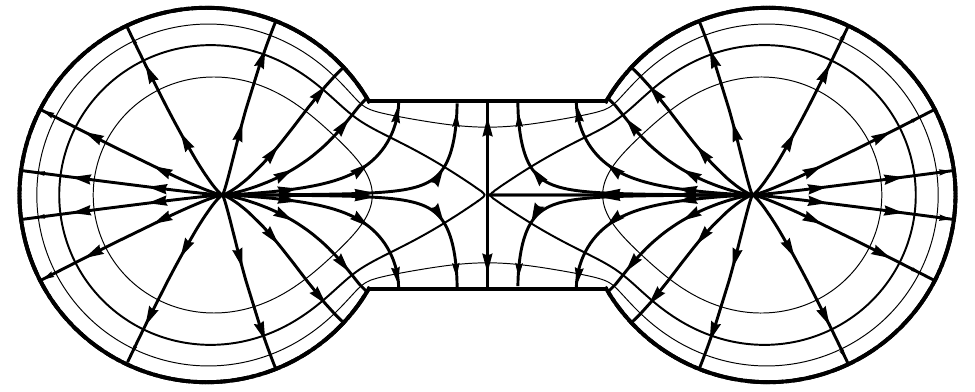} 
  \end{subfigure}%
    \\[1em]
   \begin{subfigure}{0.33\textwidth}
    \includegraphics[width=\linewidth]{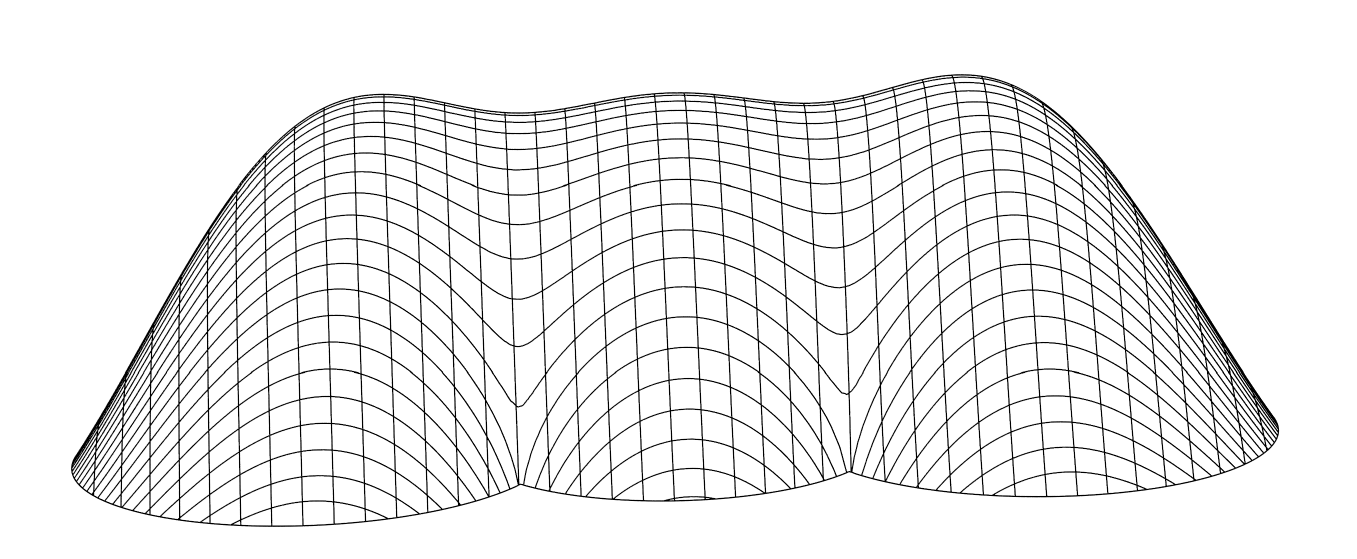}
    \caption{$r=0.86$} \label{fig:dumbbel3a}
  \end{subfigure}%
  \hspace*{\fill} 
  \begin{subfigure}{0.33\textwidth}
    \includegraphics[width=\linewidth]{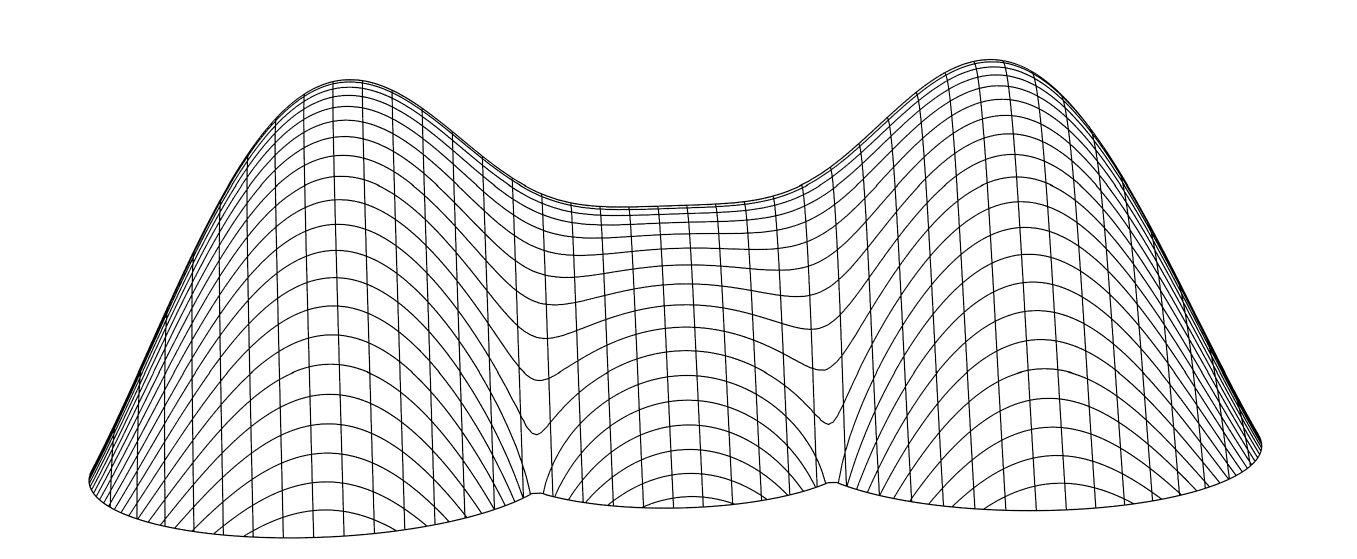}
    \caption{$r=0.78$} \label{fig:dumbbel3b}
  \end{subfigure}%
  \hspace*{\fill} 
  \begin{subfigure}{0.33\textwidth}
    \includegraphics[width=\linewidth]{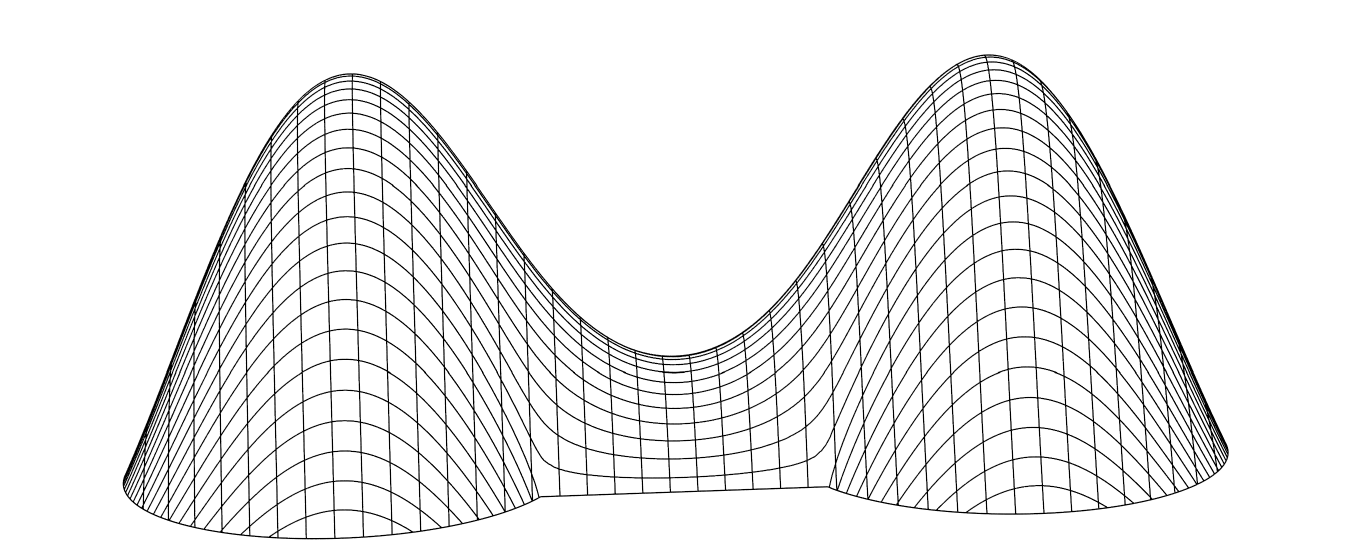}
    \caption{$r=0$} \label{fig:dumbbel3c}
  \end{subfigure}%
\caption{$\Omega(1,s,r)$ with $s=0.5$ and varying $r$. Level sets, flow lines, and perspective view of the corresponding first Dirichlet eigenfunction.}\label{fig:dumbbel3}
\end{figure}

\begin{remark}
    We refer to \cite{grebenkov2013geometrical,maji2023eigenfunction} and references therein for various results and overview on the behavior of eigenvalues and eigenfunctions of \eqref{cutproblem2d} in dumbbell shaped domains. 
\end{remark}

\begin{remark}
    Note that \cite[Examples~4.4 and~4.5]{grossi2020morse} (see also
	\cite[Example~2.2]{massimo2021number}) provide explicit \textit{torsion functions}, i.e., positive solutions of the equation $-\Delta u = 1$ in $\Omega$ with zero Dirichlet boundary conditions, with isolated semi-degenerate critical points of local maximum and saddle-node types.
\end{remark}

\section{Neumann domain count for disk and rectangles}\label{sec:counting}

Throughout this section, we denote by $u_k$ an eigenfunction corresponding to $\lambda_k$.
Let $\mu(u_k)$ be the number of \textit{Neumann domains}.
We are interested in investigating the value
$$
\mathfrak{N}(\Omega) 
=
\limsup_{k \to +\infty} \frac{\mu(u_k)}{k}
$$
for two model domains -- a disk and rectangle. 
The constant $\mathfrak{N}(\Omega)$ is introduced in the natural analogy with the Pleijel constant $\mathfrak{P}(\Omega)$ of $\Omega$ defined as
$$
\mathfrak{P}(\Omega)
=
\limsup_{k \to +\infty} \frac{\varpi(u_{k})}{k},
$$
where $\varpi(u_{k})$ stands for the number of \textit{nodal domains} of $u_k$, see \cite{Pleijel}.
It was conjectured by \textsc{I.~Polterovich} \cite[Remark~2.2]{Polterovich} that for \textit{any} bounded domain $\Omega \subset \mathbb{R}^2$ it holds
\begin{equation}\label{eq:countx1}
	\mathfrak{P}(\Omega)
	\leq
	\frac{2}{\pi}.
\end{equation}
The value $2/\pi$ is attained at least when $\Omega= (0,a)\times (0,b)$ is a rectangle with irrational $a^2/b^2$ (see \cite{helffer2015review} and also \cite[Section 6]{Pleijel}), while the validity of \eqref{eq:countx1} remains an open problem already when $\Omega$ is a square (see \cite{jung2020bounding}).

In Sections~\ref{subsect:rect} and~\ref{subsect:disk}, we show that when $\Omega$ is either a disk or rectangle $(0,a)\times (0,b)$ with irrational $a^2/b^2$, one has
\begin{equation}\label{eq:polt-neum}
	\mathfrak{N}(\Omega) 
	=
	2 \, \mathfrak{P}(\Omega) 
	\leq \frac{4}{\pi}.
\end{equation}
That is, for the maximizing sequence $\{u_k\}$, there are twice more Neumann domains than nodal domains as $k \to +\infty$.
In analogy with Polterovich's conjecture, it would be interesting to know whether \eqref{eq:polt-neum} holds for any domain $\Omega$ (under  reasonable regularity assumptions), or whether there exists $\Omega$ for which $\mathfrak{N}(\Omega) >
2 \mathfrak{P}(\Omega)$.
Heuristics based on the approach of \cite{Pleijel} (see also \cite[Remark 2.2]{Polterovich}) seem to support \eqref{eq:polt-neum}. 
Notice also the following lower bound established in \cite[Corollary~5.1]{band2016topological} for a Morse eigenfunction $u_k$:
\begin{equation}\label{eq:count-band1}
\mu(u_{k}) \geq \frac{\varpi(u_k)}{2},
\end{equation}
which might suggest that
$$
\mathfrak{N}(\Omega)
\geq 
\frac{\mathfrak{P}(\Omega)}{2}.
$$
We also refer to \cite[Sections~5,~6]{band2016topological} for further discussion on the Neumann domain count.

\subsection{Rectangles}\label{subsect:rect}
Let $\mathcal{R}(a,b) := \Omega= (0,a)\times (0,b)$ be a rectangle with sides $a,b>0$.
Consider the Dirichlet eigenvalue problem in $\mathcal{R}(a,b)$, i.e., \eqref{cutproblem2d} with $\Gamma^N = \emptyset$.
Classically, any eigenvalue is given by
$$
\lambda_{n,m} =\frac{\pi^2 n^2}{a^2} + \frac{\pi^2 m^2}{b^2},
\quad n,m \in \mathbb{N},
$$
and there is a basis of eigenfunctions consisting of
$$
u_{n,m}(x,y) = \sin\left(\frac{\pi n x}{a}\right) \sin \left(\frac{\pi m y}{b}\right),
\quad (x,y) \in \mathcal{R}(a,b).
$$
Notice that any $u_{n,m}$ is a Morse function.
If we additionally assume that $a^2/b^2$ is irrational, then each eigenvalue is simple.

\begin{proposition}
Let $a,b>0$ be such that $a^2/b^2$ is irrational. 
Then 
\begin{equation}\label{eq:rect}
	\mathfrak{N}(\mathcal{R}(a,b)) 
	=
	2 \, \mathfrak{P}(\mathcal{R}(a,b)) 
	= 
	\frac{4}{\pi}.
\end{equation}
\end{proposition}
\begin{proof}
Thanks to the nondegeneracy of critical points of $u_{n,m}$, it is not hard to analyze the number of its Neumann domains by direct inspection and deduce that $u_{n,m}$ has exactly 
\begin{enumerate}
\item $n(m-1)+(n-1)m$ inner Neumann domains (i.e., those Neumann domains whose boundary is fully consisted of Neumann lines),
\item $2m+2n$ boundary Neumann domains (i.e., those Neumann domains whose boundary intersects with $\partial\Omega$ by an arc),
\end{enumerate}
see Figure~\ref{fig1} for an example.
That is, we have
$$
\mu(u_{n,m}) =  2nm +n+m.
$$

\begin{figure}[ht]
\centering
\includegraphics[width=.6\linewidth]{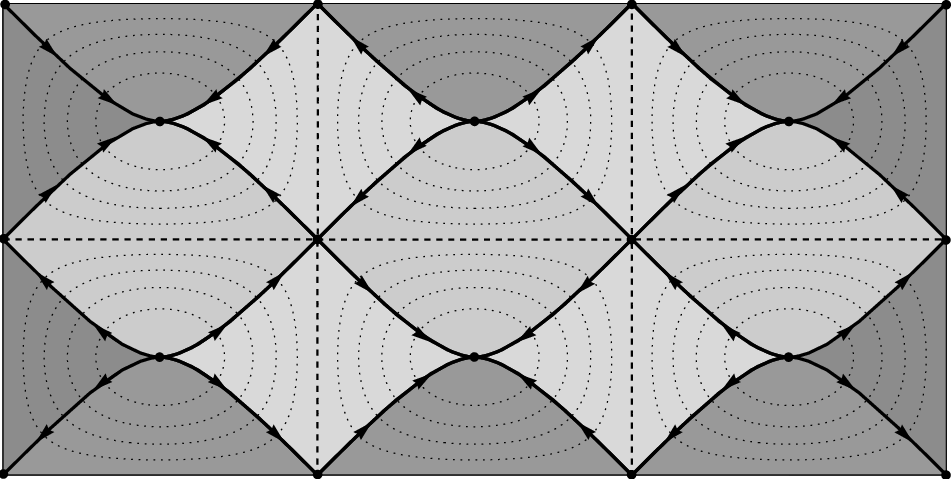}
\captionof{figure}{Neumann domains of the eigenfunction $u_{3,2}$ in $\mathcal{R}(2,1)$: $10$ boundary Neumann domains (dark grays) and $7$ inner Neumann domains (light grays). Dashed lines are nodal lines, dotted lines are some level lines.}
\label{fig1}
\end{figure}

Although the exact relation between the increasing sequence of eigenvalues $\{\lambda_k\}$ and the two-indexed sequence of eigenvalues $\{\lambda_{n,m}\}$ is hardly possible to obtain in the general case, we can use the Weyl law stating that
$$
k = (2\pi)^{-2} \, \pi \, ab \, \lambda_k + o(\lambda_k),
$$
where $\pi$ is the volume of the unit disk and $|\mathcal{R}(a,b)|=ab$. 
Thus, we get
$$
k = \frac{\pi \, ab}{4}  \left(\frac{n_{k}^2}{a^2} + \frac{m_{k}^2}{b^2}\right) + o(\lambda_k),
$$
where $\lambda_k = \lambda_{n_k,m_k}$ for a \textit{unique} pair of numbers $n_k, m_k \in \mathbb{N}$.

Using these facts, for the value $\mathfrak{N}(\mathcal{R}(a,b))$ we have
\begin{align}
&\mathfrak{N}(\mathcal{R}(a,b)) 
= 
\limsup_{k \to +\infty} \frac{\mu(u_k)}{k}\\
\label{eq:rect0}
&=
\limsup_{n+m \to +\infty} \frac{2nm +n+m}{\frac{\pi \, ab}{4}  \left(\frac{n^2}{a^2} + \frac{m^2}{b^2}\right) + o(\lambda_{n,m})}
=
\limsup_{n+m \to +\infty} \frac{2nm +n+m}{\frac{\pi \, ab}{4}  \left(\frac{n^2}{a^2} + \frac{m^2}{b^2}\right)}\\
&=
\frac{4}{\pi}\limsup_{n+m \to +\infty} \frac{2nm +n+m}{ab\left(\frac{n^2}{a^2} + \frac{m^2}{b^2}\right)}.
\end{align}
By the subadditivity, we get
\begin{equation}\label{eq:rect1}
\mathfrak{N}(\mathcal{R}(a,b))
\leq
\frac{8}{\pi}\limsup_{n+m \to +\infty} \frac{\frac{n}{a}\cdot\frac{m}{b}}{\frac{n^2}{a^2} + \frac{m^2}{b^2}}
+
\frac{4}{\pi a b}\limsup_{n+m \to +\infty} \frac{n+m}{\frac{n^2}{a^2} + \frac{m^2}{b^2}}.
\end{equation}
Applying the inequality of arithmetic and geometric means to the first term on the right-hand side of \eqref{eq:rect1}, we obtain
$$
\frac{\frac{n}{a}\cdot\frac{m}{b}}{\frac{n^2}{a^2} + \frac{m^2}{b^2}} \leq \frac{1}{2},
$$
and hence
$$
\frac{8}{\pi}\limsup_{n+m \to +\infty} \frac{\frac{n}{a}\cdot\frac{m}{b}}{\frac{n^2}{a^2} + \frac{m^2}{b^2}} \leq \frac{4}{\pi}.
$$
As for the second term on the right-hand side of \eqref{eq:rect1}, we have
$$
\frac{4}{\pi a b}\limsup_{n+m \to +\infty} \frac{n+m}{\frac{n^2}{a^2} + \frac{m^2}{b^2}} = 0.
$$
Thus, we conclude that
$$
\mathfrak{N}(\mathcal{R}(a,b))
\leq
\frac{4}{\pi}.
$$
This upper bound is actually an exact value of $\mathfrak{N}(\mathcal{R}(a,b))$.
Indeed, taking any sequence $\{(n_k,m_k)\}\subset \mathbb{N}^2$ such that $\lim\limits_{k \to +\infty} \frac{n_k}{k} = a$ and $\lim\limits_{k \to +\infty} \frac{m_k}{k} = b$, 
we get
$$
\mathfrak{N}(\mathcal{R}(a,b))
\geq 
\frac{4}{\pi}\lim_{k \to +\infty} \frac{2n_k m_k +n_k+m_k}{ab\left(\frac{n_k^2}{a^2} + \frac{m_k^2}{b^2}\right)}
=
\frac{4}{\pi}\lim_{k \to +\infty}\frac{2ab k^2 + ak + bk}{ab (k^2+k^2)}
=
\frac{4}{\pi}.
$$
Summarizing, we proved that $\mathfrak{N}(\mathcal{R}(a,b)) = 4/\pi$.
The first equality in \eqref{eq:rect} then follows from \cite[Proposition~5.1]{helffer2015review} which asserts that $\mathfrak{P}(\mathcal{R}(a,b)) = 2/\pi$.
\end{proof}

We also refer to \cite{mcdonald2014neumann,band2021spectral,band2016topological,band2020defining} for some analytic and numeric results on Neumann domains and lines when $\Omega$ is a rectangle under Dirichlet, Neumann, and periodic boundary conditions.

\subsection{Disk}\label{subsect:disk}
Let $B := \Omega$ be the unit disk centered at the origin. 
Separating variables in polar coordinates $(\varrho,\vartheta)$, one can find a basis of eigenfunctions of \eqref{cutproblem2d} in $B$ with zero Dirichlet boundary conditions consisting of
\begin{align}
\label{eq:neum:disk:1}
&u_{n,m}(\varrho,\vartheta) 
= 
J_n(j_{n,m} \varrho) \cos(n \vartheta),
\quad 
n \in \mathbb{N} \cup \{0\},~
m \in \mathbb{N},\\
\label{eq:neum:disk:0}
&\widetilde{u}_{n,m}(\varrho,\vartheta) 
= 
J_n(j_{n,m} \varrho) \sin(n \vartheta),
\quad 
n \in \mathbb{N},~
m \in \mathbb{N},
\end{align}
where $j_{n,m}$ stands for the $m$-th positive zero of the Bessel function $J_n$.
Both $u_{n,m}$ and $\widetilde{u}_{n,m}$ correspond to the eigenvalue $\lambda_{n,m} = j_{n,m}^2$. 
Any eigenfunction $u_{0,m}$ is radial, while any other eigenfunction \eqref{eq:neum:disk:1}, \eqref{eq:neum:disk:0} has a ``dihedral'' symmetry. 

\begin{remark}\label{rem:disk}
Let us discuss the degeneracy of critical points of $u_{n,m}$.
Observe that 
	\begin{equation}\label{eq:phi-derivative}
	\frac{\partial u_{n,m}}{\partial \varrho} 
	=
	j_{n,m} J_n'(j_{n,m} \varrho) \cos(n \vartheta),
	\quad
	\frac{\partial u_{n,m}}{\partial \vartheta} 
	= 
	- n J_n(j_{n,m} \varrho) \sin(n \vartheta),
	\end{equation}
	and the Hessian determinant of $u_{n,m}$ is
	$$
	\text{Hess}(u_{n,m}) 
	=
	- j_{n,m}^2 n^2 \left[J_n(j_{n,m} \varrho) J_n''(j_{n,m} \varrho) \cos^2(n \vartheta) + (J_n'(j_{n,m} \varrho))^2 \sin^2(n \vartheta)\right].
	$$
	Since positive zeros of $J_n$ and $J_n'$ are interlacing (see, e.g., \cite[Section~15.23, p.~480]{watson}),
	we deduce from \eqref{eq:phi-derivative} that, for $n \geq 1$, $(\varrho_0,\vartheta_0)$ with $\varrho_0>0$ is a critical point of $u_{n,m}$ if and only if either 
	$$
	J_n(j_{n,m} \varrho_0) \neq 0, \quad \cos(n \vartheta_0) \neq 0
	\quad \text{and} \quad
	J_n'(j_{n,m} \varrho_0) = 0, \quad \sin(n \vartheta_0)=0,
	$$
	or 
	$$
	J_n(j_{n,m} \varrho_0) = 0, \quad \cos(n \vartheta_0) = 0
	\quad \text{and} \quad
	J_n'(j_{n,m} \varrho_0) \neq 0, \quad \sin(n \vartheta_0) \neq 0.
	$$
	In the latter case, we deduce that $\text{Hess}(u_{n,m}) \neq 0$.
	In the former case, we obtain the same conclusion by additionally noting that  $J_n'$ and $J_n''$ do not have common positive zeros (see, e.g., \cite{baricz2018zeros}).
	If $n = 0$, then a critical point $(\varrho_0,\vartheta_0)$ with $\varrho_0>0$ is necessarily degenerate and, more precisely, semi-degenerate since $\frac{\partial^2 u_{0,m}}{\partial \varrho^2}(\varrho_0,\vartheta_0) \neq 0$.
    Indeed, we have $\frac{\partial^2 u_{0,m}}{\partial \varrho^2}(\varrho_0,\vartheta_0) \neq 0$ because $J_0'$ and $J_0''$ do not have common positive zeros, which follows, e.g., by noting that $J_0' = -J_1$.
    Finally, since $J_n(s) \sim c_1 s^n - c_2 s^{n+2}$ for some $c_1,c_2>0$ as $s \to 0$ (cf.\ \cite[Chapter~III]{watson}), we deduce that the origin is a non-degenerate critical point if $n=0$ and $n=2$, and it is a fully-degenerate critical point if $n \geq 3$.
\end{remark}

 \begin{remark}\label{rem:disk2}
	Let us now comment on the behavior of flow lines of $u_{n,m}$ in a neighborhood of the origin.
	The Cauchy problem \eqref{eq:gradflow2} in the polar coordinates $(\varrho,\vartheta)$ transfers to
	\begin{align}
		\frac{d \varrho}{dt} 
		&= 
		-\frac{\partial u_{n,m}}{\partial \varrho}
		=
		-j_{n,m} J_n'(j_{n,m} \varrho) \cos(n \vartheta),\\
		\frac{d \vartheta}{dt} 
		&= 
		-\frac{1}{\varrho^2}
		\frac{\partial u_{n,m}}{\partial \vartheta}
		=
		\frac{n}{\varrho^2} J_n(j_{n,m} \varrho) \sin(n \vartheta),
	\end{align} 
	with the initial conditions $(\varrho(0),\vartheta(0)) = (\varrho_0,\vartheta_0)$.
	Let $n \geq 1$.
	It is not hard to see that if $\sin(n \vartheta_0) = 0$ and $\varrho_0>0$ is sufficiently close to $0$, then $\vartheta(t) = const$ for all $t$ and $\varrho(t) \to 0$ as $t \to +\infty$ or $t \to -\infty$, i.e., the corresponding flow line in the Cartesian coordinates converges to the origin, cf.\ Figure~\ref{fig:1g}.
	Let us show that there are no other flow lines converting to the origin.
	Assuming $J_n(j_{n,m} \varrho_0) \neq 0$ (which is not restrictive for sufficiently small $\varrho_0>0$) and $\cos(n \vartheta_0) \neq 0$, we obtain
	$$
	\frac{d \varrho}{d \vartheta} 
	= 
	-\frac{j_{n,m} \varrho^2}{n} \frac{J_n'(j_{n,m} \varrho)}{J_n(j_{n,m} \varrho)} \cot(n \vartheta).
	$$
	Integrating, we deduce that
	$$
	\frac{n^2}{j_{n,m}} \int_{\varrho_0}^\varrho \frac{J_n(j_{n,m} r)}{r^2 J_n'(j_{n,m} r)} \, dr
	=
	-n \int_{\vartheta_0}^\vartheta \cot(n \theta) \, d\theta
	=
	-\ln (\sin(n\vartheta)) + \ln(\sin(n \vartheta_0)),
	$$
	and hence
	\begin{equation}\label{eq:neum:disk:3}
		\vartheta
		=
		\frac{1}{n} \arcsin \left( \sin(n \vartheta_0) \exp\left[-\frac{n^2}{j_{n,m}} \int_{\varrho_0}^\varrho \frac{J_n(j_{n,m} r)}{r^2 J_n'(j_{n,m} r)} \, dr\right]\right).
	\end{equation}
    This is the equation describing flow lines.  
	Noting that $J_n(s) \sim s^n$ and $J_n'(s) \sim s^{n-1}$ as $s \to 0$, we see that if $\varrho \to 0$, then the expression in the square brackets in \eqref{eq:neum:disk:3} goes to $+\infty$, which leads to a contradiction provided $\sin(n \vartheta_0) \neq 0$.
	That is, $\sin(n \vartheta_0) = 0$ is a necessary assumption for a flow line of $u_{n,m}$ to converge to the origin.
\end{remark}

Now we state a result on the Neumann domain count for $B$.
\begin{proposition}\label{prop:neuman-count:disk}
	Let $B$ be the unit disk. Then
	\begin{equation}\label{eq:disk}
		\mathfrak{N}(B) 
		=
		2 \, \mathfrak{P}(B)
		=
		16 \, \sup_{s>0} \left\{ s \left(\cos \theta(s)\right)^2 \right\} 
		= 0.9226... < \frac{4}{\pi},
	\end{equation}
 where $\theta=\theta(s)$ is the solution of the equation
\begin{equation}\label{eq:theta0}
\tan \theta - \theta = \pi s, 
\quad 
\theta \in \left(0, \frac{\pi}{2}\right).
\end{equation}
\end{proposition}
\begin{proof}
By straightforward calculations and Remark~\ref{rem:disk}, one can observe the following facts:
\begin{enumerate}
	\item\label{enum:disk:1} If $n=0$, then $u_{0,m}$ has $m-1$ circles of semi-degenerate local extremum points and one non-degenerate local extremum point at the origin.
	Therefore, $u_{0,m}$ has $m-1$ inner Neumann domains and $1$ boundary Neumann domain (see Figure~\ref{fig:disk0}), and hence 
	\begin{equation}\label{eq:neum:disk:0m}
	\mu(u_{0,m}) = m.
	\end{equation}
	\item\label{enum:disk:2} If $n=1$, then $u_{1,m}$ has $2m$ non-degenerate local extremum points and
	$2m$ non-degenerate saddle points. 
	Therefore, it has $4(m-1)+1$ inner Neumann domains and $2$ boundary Neumann domains (see Figure~\ref{fig:disk1}), and hence 
	\begin{equation}\label{eq:neum:disk:1m}
	\mu(u_{1,m}) = 4m-1.
\end{equation}
	\item\label{enum:disk:3} If $n = 2$, then $u_{2,m}$ has $4m$ non-degenerate local extremum points and $4 m + 1$ non-degenerate saddle points.
	Therefore, it has $4(2m-1)$ inner Neumann domains and $4$ boundary Neumann domains (see Figure~\ref{fig:disk2}), and hence 
	\begin{equation}\label{eq:neum:disk:2m}
	\mu(u_{2,m}) = 8m.
\end{equation}
	\item\label{enum:disk:4} If $n \geq 3$, then $u_{n,m}$ has $2n m$ non-degenerate local extremum points, $2n m$ non-degenerate saddle points, and $1$ fully-degenerate saddle point at the origin.
	Therefore, it has $2n (2m-1)$ inner Neumann domains and $2n$ boundary Neumann domains (see Figure~\ref{fig:disk3}), and hence 
	\begin{equation}\label{eq:neum:disk:nm}
	\mu(u_{n,m}) = 4 n m.
\end{equation}
\end{enumerate}

\begin{figure}[!ht]
  \begin{subfigure}{0.49\textwidth}
    \includegraphics[width=0.9\linewidth]{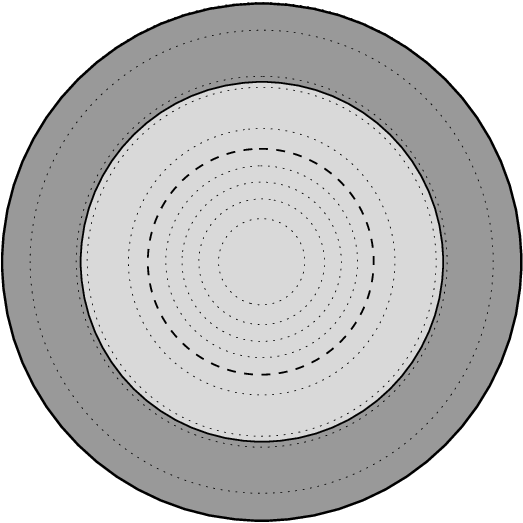}
    \caption{$n=0$, $m=2$} \label{fig:disk0}
  \end{subfigure}%
  \hspace*{\fill} 
  \begin{subfigure}{0.49\textwidth}
    \includegraphics[width=0.9\linewidth]{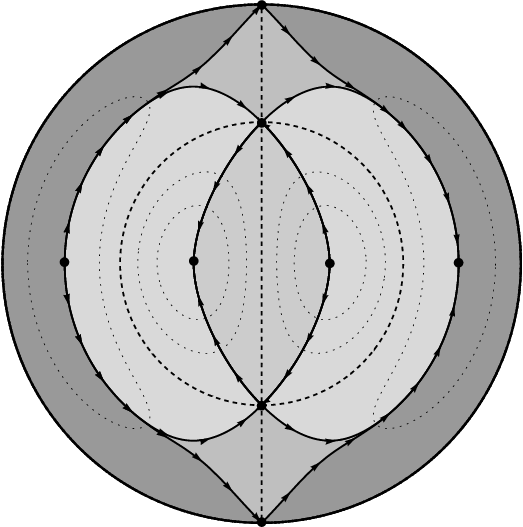}
    \caption{$n=1$, $m=2$} \label{fig:disk1}
  \end{subfigure}%
    \\[1em]
   \begin{subfigure}{0.49\textwidth}
    \includegraphics[width=0.9\linewidth]{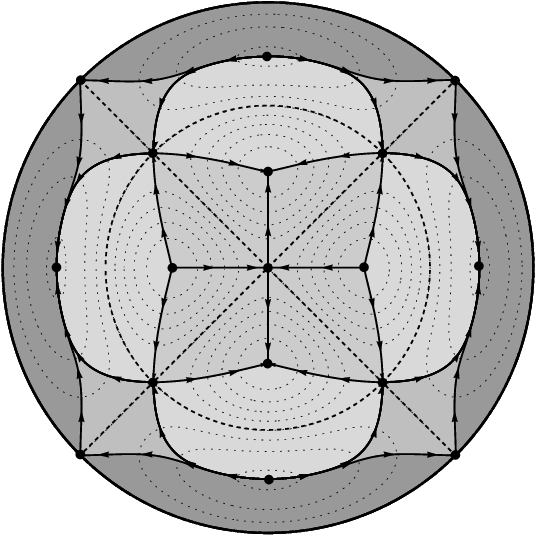}
    \caption{$n=2$, $m=2$} \label{fig:disk2}
  \end{subfigure}%
  \hspace*{\fill} 
  \begin{subfigure}{0.49\textwidth}
    \includegraphics[width=0.9\linewidth]{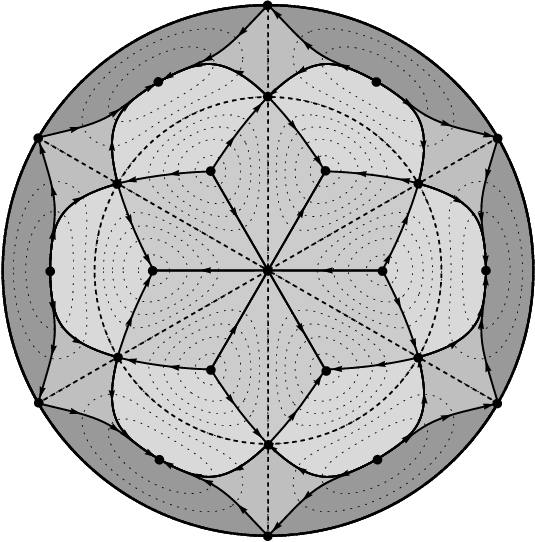}
    \caption{$n=3$, $m=2$} \label{fig:disk3}
  \end{subfigure}%
    
\caption{Neumann domains of the eigenfunction $u_{n,m}$: boundary Neumann domains - dark gray, inner Neumann domains - light grays, nodal lines - dashed, some level lines - dotted.}
\end{figure}

In particular, if $n=0$ or $n \geq 3$, then $u_{n,m}$ has degenerate critical points, and hence definitions and results from \cite[Section~3]{band2016topological} are not directly applicable to such eigenfunctions.
On the other hand, our definitions and results do apply.

Let us observe that any ``radial'' eigenvalue $\lambda_{0,m}$ has multiplicity $1$ and any other eigenvalue has multiplicity $2$, which follows from the validity of Bourget's hypothesis asserting that $J_n$ and $J_{n+l}$ do not have common positive zeros for any $n \in \mathbb{N} \cup \{0\}$ and $l \in \mathbb{N}$, see \cite[p.~484]{watson}.
Consequently, any eigenfunction of the considered problem is either of the form \eqref{eq:neum:disk:1} or its rotation, and hence the number of its Neumann domains is described as above.

Using these facts, the value $\mathfrak{N}(B)$ can be now computed in much the same was as in \cite{bobkov}.
We provide details for the sake of clarity.
In order to relate the two-indexed sequence $\{\lambda_{n,m}\}$ with the increasing one-indexed sequence $\{\lambda_k\}$, we use the Weyl law to get
$$
k = (2\pi)^{-2} |B|^2 \lambda_k + o(\lambda_k) = \frac{\lambda_k}{4} + o(\lambda_k).
$$
Recalling the validity of Bourget's hypothesis, we see that for any $\lambda_k$ there exists a unique pair $(n_k,m_k)$ such that $\lambda_k = \lambda_{n_k,m_k} = j_{n_k,m_k}^2$. 
Consequently, 
$$
k = \frac{j_{n_k,m_k}^2}{4} + o(j_{n_k,m_k}^2).
$$
Therefore, we get
\begin{equation}\label{eq:neum:disk:N}
\mathfrak{N}(B)
=
\limsup_{k \to +\infty} \frac{\mu(u_k)}{k}
=
4 \, \limsup_{k \to +\infty} \frac{\mu(u_{n_k,m_k})}{j_{n_k,m_k}^2}
=
4 \, \limsup_{n+m \to +\infty} \frac{\mu(u_{n,m})}{j_{n,m}^2}.
\end{equation}
We will see below that this limit is strictly positive.
Let us now use the following inequality from \cite[Corollary, p.~102]{mccann}:
\begin{equation}\label{eq:mccann}
j_{n,m} 
> 
\left(n^2 + \pi^2 \left(m-\frac{1}{4}\right)^2\right)^{1/2}
\quad 
\text{for any } n \in \mathbb{N} \cup \{0\}
\text{ and } m \in \mathbb{N}.
\end{equation}
In particular, \eqref{eq:mccann} implies that $j_{0,m} > \pi(m-1)$, and hence, by \eqref{eq:neum:disk:0m},
$$
\lim_{m \to +\infty} \frac{\mu(u_{0,m})}{j_{0,m}^2}
=
\lim_{m \to +\infty} \frac{m}{j_{0,m}^2} = 0.
$$
That is, $\{u_{0,m}\}$ is not a maximizing sequence.
Similarly, in view of \eqref{eq:neum:disk:1m} and \eqref{eq:mccann},
$$
\lim_{m \to +\infty} \frac{\mu(u_{1,m})}{j_{1,m}^2}
=
\lim_{m \to +\infty} \frac{4m-1}{j_{1,m}^2} = 0,
$$
and hence $\{u_{1,m}\}$ in not a maximizing sequence, too.
Therefore, by \eqref{eq:neum:disk:2m} and \eqref{eq:neum:disk:nm} we have
\begin{equation}\label{eq:neum:disk:N2}
\limsup_{n+m \to +\infty} \frac{\mu(u_{n,m})}{j_{n,m}^2}
=
\limsup_{n+m \to +\infty} \frac{4n m}{j_{n,m}^2}.
\end{equation}
Let $\{(n,m)\}$ be a maximizing sequence.
Arguing exactly as in  \cite[p.~806]{bobkov}, it can be shown that $m = n s_0 + o(n)$ for some $s_0 > 0$. 
In particular, we have $n \geq 3$.
Let us now use the result of \textsc{Elbert \& Laforgia} \cite{elbert1994lower} (see \cite[Section 1.5]{elbert2001some} for the precise statement employed in \eqref{3}) which states that	
\begin{equation}\label{3}
\lim_{n \to +\infty} \frac{j_{n,n s}}{n} = \frac{1}{\cos \theta(s)}, 
\quad s > 0,
\end{equation}
where $\theta=\theta(s)$ is the solution of the transcendental equation
\begin{equation}\label{eq:theta}
\tan \theta - \theta = \pi s, 
\quad 
\theta \in \left(0, \frac{\pi}{2}\right).
\end{equation}
Using \eqref{3}, we deduce from \eqref{eq:neum:disk:N} and \eqref{eq:neum:disk:N2} that 
\begin{align*}
\mathfrak{N}(B)
&=
4 \,
\limsup_{n+m \to +\infty} \frac{\mu(u_{n,m})}{j_{n,m}^2}
=
16 \, \limsup_{n+m \to +\infty} \frac{n m}{j_{n,m}^2}
=
16 \, \lim_{n \to +\infty} \frac{s_0 n^2}{j_{n,s_0 n}^2}
\\
&=
16 \, s_0 \, \left(\cos \theta(s_0)\right)^2 = 16 \, \sup_{s>0} \left\{ s \left(\cos \theta(s)\right)^2 \right\} 
= 0.9226...
\end{align*}
The first equality in \eqref{eq:disk} is given by \cite[Theorem~1.3]{bobkov}, which finishes the proof.
\end{proof}

\section{Comments and remarks}\label{sec:remarks}

\begin{enumerate}
    \item The application of \cite{weineffect} in \cite{hers}, which was mentioned in Section~\ref{sec:intro}, might require an inspection and further investigation, since the regularity of the effectless cut is not discussed in \cite{hers,weineffect}. 
    The regularity of the Neumann line set $\mathcal{N}(u)$ was recently investigated in	\cite{band2021spectral,band2020defining} under the assumption that the eigenfunction is a Morse function.

\item 
It would be interesting to know whether any inner Neumann domain intersects with exactly two nodal domains, as it happens in the Morse case, see \cite[Theorem~1.4~(vii)]{band2016topological}.

\item Neumann domains are not necessarily simply connected, which can be seen by considering radial Dirichlet eigenfunctions in a disk or concentric rings.
On the other hand, Neumann domains are simply connected if $\Omega$ is simply connected, $u$ is a Morse function, and the set $\mathcal{S}$ of saddle points of $u$ is nonempty, see \cite[Theorem~3.13]{band2016topological}. 
In the present settings, it would be also interesting to know whether any Neumann domain is simply connected if the critical set $\mathcal{C}$ consists only of isolated critical points and $\mathcal{S} \neq \emptyset$.

\item The analyticity of $\Gamma^N$ is particularly used in Lemma~\ref{lem:isol} (and hence in Lemma~\ref{lem:morse-bott}). If $\Gamma^N$ is only piecewise analytic, the interrelation between $\Gamma^N$ and curves of critical points might be different.
\end{enumerate}

	\smallskip
	\noindent
	\textbf{Acknowledgments.}
        The authors~are grateful to K.~Pagani for a discussion about Remark~\ref{rem:cheng}, to Yu.~Kordyukov, A.~Shavlukov, B.~Suleimanov for discussions on Sections~\ref{sec:classification} and~\ref{sec:classification-of-manifolds}, and to S.~Sasi for bringing the work \cite{chenmyrtaj2019} to their attention. 
        The work was performed during a series of visits of V.B.\ and M.G.\ to IIT Madras, and T.V.A.\ to IM UFRC RAS. 
        The former visits were supported by the Office of Global Engagement of the IIT Madras.
        The latter visit was supported by the Theoretical Physics and Mathematics Advancement Foundation ``BASIS'', Grant No.~23-3-4-1-1.
	T.V.A also acknowledges the Core Research Grant  (CRG/2023/005344)  by ANRF, and 
        M.G. is supported by TIFR Centre for Applicable Mathematics (TIFR-CAM).
\bibliographystyle{abbrvurl}
\bibliography{ref}

\end{document}